\documentclass[11pt]{article}
\usepackage{geometry}
\usepackage{mathrsfs}

\usepackage{verbatim}
\usepackage{amsmath,amsfonts,amsthm,amssymb,epsfig}
\usepackage[all]{xy}
\usepackage{enumitem}
\usepackage[usenames,dvipsnames]{xcolor}
\usepackage[bookmarks=true,colorlinks=true, linkcolor=Mulberry, citecolor=Periwinkle]{hyperref}
\usepackage{cleveref} 

\newcommand{\triplearrows}{\begin{smallmatrix} \to \\ \to \\ 
\to \end{smallmatrix} }

\newcommand{\qcoh}{\mathrm{QCoh}}
\newcommand{\sk}{\mathrm{sk}}
\newcommand{\sF}{\mathcal{F}}
\newcommand{\md}{\mathrm{Mod}}
\newcommand{\sh}{\mathrm{Sh}}
\renewcommand{\sp}{\mathrm{Sp}}
\usepackage{graphicx}
\usepackage{float}

\newcommand{\et}{\mathrm{et}}
\newcommand{\algspc}{\mathrm{AlgSpc}}
\newtheorem{theorem}{Theorem}[section]
\newcommand{\psh}{\mathrm{PSh}}
\newcommand{\perf}{\mathrm{Perf}}
\newtheorem{lemma}[theorem]{Lemma}
\newtheorem{proposition}[theorem]{Proposition}
\newtheorem{corollary}[theorem]{Corollary}
\renewcommand{\hom}{\mathrm{Hom}}
\theoremstyle{definition}
\newcommand{\fun}{\mathrm{Fun}}
\newtheorem{construction}[theorem]{Construction}
\newcommand{\spec}{\mathrm{Spec}}
\newtheorem{definition}[theorem]{Definition}
\newtheorem{example}[theorem]{Example}
\newtheorem{remark}[theorem]{Remark}

\newcommand{\map}{\mathrm{Hom}}
\newcommand{\TC}{\mathrm{TC}}

\newcommand\bcdot{\ensuremath{%
  \mathchoice%
   {\mskip\thinmuskip\lower0.2ex\hbox{\scalebox{1.5}{$\cdot$}}\mskip\thinmuskip}}%
   {\mskip\thinmuskip\lower0.2ex\hbox{\scalebox{1.5}{$\cdot$}}\mskip\thinmuskip}%
   {\lower0.3ex\hbox{\scalebox{1.2}{$\cdot$}}}%
   {\lower0.3ex\hbox{\scalebox{1.2}{$\cdot$}}}%
   }

\begin{document}

\title{Hyperdescent and \'etale $K$-theory}

\author{Dustin Clausen and Akhil Mathew}

\maketitle

\begin{abstract}
We study the \'etale sheafification of algebraic $K$-theory, called \'etale
$K$-theory. 
Our main results show that \'etale $K$-theory is very close to a noncommutative invariant
called Selmer
$K$-theory,
which is defined at the level of categories. 
Consequently, we show that \'etale $K$-theory has surprisingly well-behaved
properties, integrally and without finiteness assumptions.  
A key theoretical ingredient is the distinction, which we investigate in detail, between  sheaves and hypersheaves of spectra on \'etale sites.\end{abstract}

\tableofcontents
\section{Introduction}

Let $K$ denote the algebraic $K$-theory functor defined on quasi-compact
quasi-separated (qcqs) spectral algebraic
spaces.\footnote{Here we mean the non-connective $K$-theory of perfect complexes, as in
\cite{TT90, BGT}.  The reader is free to imagine $X$ a usual algebraic space over
$\mathbb{Z}$, but the results hold more generally for $X$ a spectral algebraic
space as  
in \cite{SAG}. Furthermore, the structure sheaf need only have an $E_2$-ring structure, not
necessarily $E_\infty$. }
After \cite{TT90}, $K$ satisfies
Nisnevich descent, but it does not satisfy the even more useful \'etale descent.
Thus, we let $K^{et}$ (\emph{\'etale $K$-theory}) denote the \'etale sheafification of $K$, viewing the latter
simply as a presheaf of spectra.

The idea of ``\'etale $K$-theory'' has a long history, going back at least to
the work of Soul\'e \cite{Sou79, Sou81}, and there have been many
different approaches. 
Slightly different (but a posteriori equivalent) versions of this construction have been considered by Friedlander \cite{F80,
F82}
and Dwyer-Friedlander \cite{DF85} in the $l$-adic case (where $l$ is invertible
on the base) and in the $p$-adic case by Geisser-Hesselholt \cite{GH99}.  
In this paper, we simply sheafify $K$-theory in the sense of \cite[Ch. 6]{HTT}; this is
different from approaches such as \cite{Ja87, DHI04}, which involve a
stronger process called hypersheafification. Our main results show in
particular that this
yields a functor $K^{et}$ for connective ring spectra which
behaves well in non-noetherian settings and  works integrally.

\subsection{Main results}
For general presheaves of spectra,
sheafification is a difficult operation to access: its only explicit description
is as a transfinite composition of \v{C}ech constructions \cite[Prop. 6.2.2.7
and proof]{HTT}, and thus one might expect $K^{et}$ to be difficult to
describe.  Nonetheless, we prove the following four theorems which give a good handle on $K^{et}$.

 Let $L_1$ denote Bousfield localization at the complex $K$-theory spectrum, and let $\operatorname{TC}$ denote the functor of topological cyclic homology.

\begin{theorem}[\Cref{mainetaleKdesc}(1) below]
\label{KetvsKSel}
Let $X$ be a qcqs spectral algebraic space. Then the natural map
$$K^{et}(X)\rightarrow L_1K(X)\times_{L_1\operatorname{TC}(X)}\operatorname{TC}(X)$$
is an isomorphism on homotopy in degrees $\geq -1$.
\end{theorem}

The comparison map in \Cref{KetvsKSel} comes from the cyclotomic trace
$K\rightarrow\operatorname{TC}$ and the fact that all of the theories $L_1K$,
$L_1\operatorname{TC}$, and $\operatorname{TC}$ satisfy \'etale descent.  For
$L_1K$ and $L_1\operatorname{TC}$ \'etale descent follows from the generalization of (some
of) Thomason's work  \cite{Th85} and \cite[Sec. 11]{TT90} provided by \cite{CMNN}.\footnote{Actually,
\cite{CMNN} assumed $E_\infty$-structure sheaves, and this was crucial to the
method of proof; we will show here how to establish the same result assuming
only $E_2$-structure sheaves.  And for this descent result the structure sheaves are
even allowed to be non-connective, cf.\ section~\ref{secEtwo}.}  For $\operatorname{TC}$ it was proved by
Geisser-Hesselholt \cite{GH99}, cf.\  \cite{WG} for the more classical case of
ordinary Hochschild homology and \cite[Sec. 3]{BMS2} and \Cref{TCissheaf}
below for the current level of generality.

The interest of \Cref{KetvsKSel} is the following. The sheafification procedure
defining $K^{et}$ from $K$ destroys the fundamental property of $K$-theory, which
is that it only depends on the appropriate category of modules.  Thus $K^{et}$
is missing important structure such as proper pushforward functoriality, and
important flexibility such as the ability to define $K^{et}$ of an arbitrary
category (with some kind of exact structure).  \Cref{KetvsKSel} shows that,
miraculously enough, these losses are only apparent.  The theory
$K^{Sel}
\stackrel{\mathrm{def}}{=}L_1K\times_{L_1\operatorname{TC}}\operatorname{TC}$ on
the right, introduced 
in \cite{Artinmaps} and dubbed \emph{Selmer $K$-theory}, by definition
only depends on the category of perfect complexes. Then \Cref{KetvsKSel} says that Selmer $K$-theory is essentially the same as \'etale $K$-theory.

 Selmer $K$-theory $K^{Sel}$ thus combines the best of both worlds: it has \'etale descent and hence can be related to standard cohomology theories, but it only depends on the category of modules and hence has the same flexibility as algebraic $K$-theory.
This theorem can be viewed as a kind of combination and generalization of
Thomason's work \cite{Th85} on $L_1K$ at primes different from the residue
characteristic, and the work of Geisser-Hesselholt \cite{GH99} on $\TC$ at
primes equal to the residue characteristic.  Note that Thomason made crucial use
of this ``best of both worlds" property he established for $L_1K$ in his proof
of Grothendieck's purity conjecture with $\mathbb{Q}_\ell$-coefficients
\cite{Th84}.
The same phenomenon is also crucial to the first author's new proof of the Artin
reciprocity law \cite{Artinmaps}.

\begin{theorem}[\Cref{LQstatementmain}]
\label{KapproxKet}
Let $X$ be a qcqs spectral algebraic space of finite Krull dimension, and $p$ a
prime.  For a field $k$, and let $d_k$ denote the (mod $p$) virtual Galois cohomological
dimension of $k$ if $k$ has characteristic $\neq p$, and $
1 + \operatorname{dim}_k\Omega^1_{k/k^p}$ if $k$ has characteristic $p$.  Then the map
$$K(X)\rightarrow K^{et}(X)$$
is an isomorphism on $p$-local homotopy groups in degrees $\geq
\operatorname{max}(\operatorname{sup}_{x\in X} d_{k(x)} - 2, 0)$.
\end{theorem}

\Cref{KapproxKet} is a Lichtenbaum-Quillen-type statement; it says that in high enough
degrees relative to the dimension, algebraic $K$-theory does satisfy \'etale
descent.  With coefficients prime to the residue characteristics it is
well-known  (cf.\ \cite{RO06}) that such statements follow from the Gabber-Suslin rigidity
theorem 
\cite{Gabber92}
and the norm residue isomorphism theorem of Voevodsky-Rost \cite{Voe03, Voe11} in the form
of the Beilinson-Lichtenbaum conjectures (see \cite{HW19} for a textbook
account of the norm residue isomorphism, and \cite{FS02, Lev08} for the relation to
$K$-theory).  We handle the general case by
inputting the generalization of Gabber-Suslin rigidity proved in \cite{CMM}.  This transfers the problem from $K$-theory to $\operatorname{TC}$, which satisfies \'etale descent by the result of Geisser-Hesselholt mentioned above.
\begin{theorem}[\Cref{mainetaleKdesc}(2) and \Cref{hypsmashingfinitecd}]
\label{KetPostnikov}
Let $X$ be a qcqs spectral algebraic space of finite Krull dimension, with a uniform bound on the virtual Galois cohomological dimension of its residue fields.  Then over $X$, the \'etale sheaf of spectra $K^{et}$ is a \emph{Postnikov sheaf}: it maps by an equivalence to the inverse limit of its (\'etale-sheafified) Postnikov tower.  In particular, there is a conditionally convergent descent spectral sequence
$$H^p(X_{et};\pi_q^{et}K)\Rightarrow \pi_{q-p}K^{et}(X).$$

In fact, even any \'etale sheaf on $X$ which is a \emph{module} over $K^{et}$ is automatically a Postnikov sheaf.
\end{theorem}

\begin{remark}
The \'etale homotopy group sheaves $\pi_q^{et}(K/n)$ of $K$-theory with finite
coefficients can be explicitly described in simpler terms, see
\Cref{KSelHomotopy}.  On the other hand, rational algebraic $K$-theory is largely unknown.
\end{remark}

\Cref{KetPostnikov} addresses a subtlety in the theory of sheaves of spectra,
whose exploration is the main theme of this paper.  This is that one can have a
non-zero sheaf of spectra all of whose stalks vanish, or all of whose homotopy
group sheaves vanish.  In fact, answering a question of Jardine \cite[p.
197]{Ja10}, we show by example
(\Cref{nonhypcompleteZp}) that this can happen even on a site of cohomological dimension $1$, namely
the usual site of finite continuous $\mathbb{Z}_p$-sets.  Such sheaves of
spectra cannot be studied in terms of sheaves of abelian groups, and they
exhibit exotic behavior.  What \Cref{KetPostnikov} shows is that $K^{et}$, as
well as any sheaf it ``touches," is non-exotic (the technical term is
\emph{hypercomplete}, see  \cite{DHI04}, \cite{Ja87}, \cite[Sec. 6.5.2]{HTT}), and hence can be studied in terms of its homotopy group sheaves, or in terms of its stalks, at least in finite-dimensional situations.

In the $l$-adic context and for Bott-inverted (rather than \'etale) $K$-theory, \Cref{KetPostnikov}  goes back to Thomason \cite{Th85, TT90}, with slight additional
assumptions; see also  \cite{RO05, RO06}, which treat the removal of these
assumptions using the Beilinson-Lichtenbaum conjectures. In the $p$-adic case
and for $\mathrm{TC}$ instead, \'etale Postnikov descent is essentially due to
\cite{GH99}.  
Recent work \cite{ELSO} also  treats analogs of Thomason's results (i.e.,
Bott-inverted \'etale hyperdescent) for modules over algebraic
cobordism replacing algebraic $K$-theory, in the context of motivic stable
homotopy theory.

\begin{theorem}[\Cref{mainetaleKdesc}(3)]
\label{Ketcocont}
The functor $R\mapsto K^{et}(\operatorname{Spec}(R))$, from commutative rings\footnote{or connective $E_2$-rings} to spectra, commutes with filtered colimits.
\end{theorem}

\Cref{Ketcocont} shows (in particular) that the stalks of $K^{et}$ are
indeed accessible, even without any finite dimensionality restrictions  on $X$:
they are simply given by the $K$-theory of strictly henselian local rings (for
which see \Cref{stalksofKSel}).  We emphasize that this is not at all a formal statement, in
this context of sheaves of spectra.  Indeed, the sheaf condition involves
infinite limits, namely homotopy fixed points for finite group actions, so even
though $K$-theory commutes with filtered colimits this property is a priori
destroyed by \'etale sheafification.  The analogous commutation with filtered
colimits for $K^{Sel}$ follows from that of $\operatorname{TC}/p$; this is also
neither formal nor obvious, but it was proved in \cite{CMM}.  Another use of
\Cref{Ketcocont} is in reducing the study of $K^{et}$ to the case of finite type
$\mathbb{Z}$-spaces, which are finite dimensional and therefore fall into the
realm of results like \Cref{KetPostnikov}.

\subsection{Technical ingredients}
We have already indicated that the norm residue isomorphism theorem, plus
Gabber-Suslin rigidity and its generalization \cite{CMM}, are key to proving the
above theorems.  The general strategy is the one pioneered by Thomason
\cite[Sec. 11]{TT90}:
use Nisnevich descent to reduce to the henselian local case, then use rigidity
to reduce to the case of fields, then use the norm residue isomorphism theorem
(and its characteristic $p$ analog \cite{GL00}) to handle fields.  (This last step is much more straightforward than in Thomason's days, where only the degree $\leq 2$ part of the norm residue isomorphism theorem was known, the Merkurjev-Suslin theorem.)

But in addition, there is another set of technical theorems lying in the
background of our arguments.  These have to do with the question of
hypercompleteness for sheaves of spectra, which was touched on in discussing
\Cref{KetPostnikov}.  In fact, much of our work in this article centers around
the question of to what extent the hypercompleteness property is automatic in
the setting of \'etale sheaves of spectra.  This is important, because
hypercomplete sheaves (also called \emph{hypersheaves}) are much easier to
access and study, both for computational and theoretical purposes.  In fact, the
proofs of the above four theorems often involve switching back and forth between the hypercomplete and non-hypercomplete settings.  (The non-hypercomplete setting is still useful because it is closer to being finitary and has better permanence properties.)

Our theoretical work on hypercompleteness holds under finite-dimensionality hypotheses.  The upshot is that hyperdescent is automatic for those sheaves of spectra which arise in practice in the study of algebraic $K$-theory, even though there are counterexamples showing that it is not automatic in general, even in dimension 1.  But let us describe the precise results.

The preliminary observation, again due to Thomason \cite[Sec. 11]{TT90}, is that one can reduce to considering separately the Nisnevich setting and the \'etale setting over a field:

\begin{theorem}[\Cref{hypetalehypNis}]
Let $X$ be a qcqs algebraic space, and let $\mathcal{F}$ be a presheaf of
spectra on $et_X$, the category of qcqs algebraic spaces \'etale
over $X$.  Suppose there is a uniform bound on the \'etale cohomological
dimension of $U$ for all $U\rightarrow X$ in $et_X$ (cf.\ 
\Cref{etalecohdim}).

Then $\mathcal{F}$ is an \'etale hypersheaf if and only if the following two conditions hold:
\begin{enumerate}
\item $\mathcal{F}$ is a Nisnevich hypersheaf;
\item For all $x\in X$, the presheaf $x^\ast\mathcal{F}$ on $et_x$, formed by the stalks $\mathcal{F}_y$ at all finite \'etale extensions $y\rightarrow x$, is an \'etale hypersheaf.
\end{enumerate}
\end{theorem}

It is well-known that the Nisnevich setting behaves similarly to the Zariski setting, so it is natural to consider them together.  Our main result in the Zariski/Nisnevich setting shows that under finite-dimensionality assumptions, hyperdescent is automatic:

\begin{theorem}[\Cref{spectralspacehtpy} and \Cref{Nisfinitehtpy}]
\label{hyperdescent}
Let $X$ be a qcqs algebraic space of finite Krull dimension.  Then on either the Zariski site $X_{Zar}$ or the Nisnevich site $X_{Nis}$, every sheaf of spectra (or even spaces) is a hypersheaf.
\end{theorem}

In fact, we show that $X_{Zar}$ and $X_{Nis}$ have \emph{homotopy dimension
$\leq d$} in the sense of \cite[Sec. 7.2.1]{HTT}, where $d$ is the Krull dimension of $X$.  This
statement was known in the noetherian case (cf.\ \cite[Theorem 3.7.7.1]{SAG},
\cite{Nis89}), and for the weaker
\emph{cohomological dimension } was known in general (at least in the Zariski
setting, \cite{Sch92}).  \Cref{hyperdescent} is a version of the \emph{Brown-Gersten property}
\cite{BG73}; in those terms, what we have done is removed the noetherian hypotheses from the familiar Brown-Gersten-style statements.

The next setting to consider is the \'etale site of a field, or more generally
the site of finite continuous $G$-sets for a profinite group $G$ (with the usual
topology, where a cover is a jointly surjective collection).  An important idea,
philosophically present in Thomason's work, is that a sheaf of spectra is a
hypersheaf if and only if it satisfies descent in a ``uniform way".  We make
this into a theorem using the notion of exponents of nilpotence from \cite[Def.
6.36]{MNN17}:

\begin{theorem}[\Cref{Gcompletecrit}]
Let $G$ be a profinite group of finite cohomological
dimension $d$, and let $\mathcal{F}$ be a sheaf of  spectra on the site of finite continuous $G$-sets.

Then $\mathcal{F}$ is hypercomplete if and only if for every open normal
subgroup $N\subset G$, the spectrum with $G/N$-action $\mathcal{F}(G/N)$ is
nilpotent of exponent $\leq d+1$. (Roughly, this means $\mathcal{F}(G/N)$ can be
built from free $G/N$-spectra in $\leq d$ steps.)
\end{theorem}

This strengthens the theorem of Tate-Thomason (cf.\ \cite[Annex 1, Ch.
1]{SerreCG} and \cite[Remark 2.28]{Th85}) on vanishing of Tate constructions,
which says that if $\mathcal{F}$ is hypercomplete then the Tate construction
$\mathcal{F}(G/N)^{t(G/N)}$ vanishes for all $N$.  The converse of the
Tate-Thomason theorem itself fails, as we show by example
(\Cref{nonhypcompleteZp}).

Putting these three theorems together, we get a good handle on what it means for
a sheaf to be a hypersheaf on $X_{et}$, when $X$ is a qcqs algebraic space
satisfying reasonable finite dimensionality hypotheses.  A surprising corollary
is that \emph{hypercompletion is smashing} (in the sense of \cite{Rav84}).  To explain
what this means, recall that hypersheaves of spectra are a Bousfield localization of sheaves of spectra, via a localization functor called hypercompletion.  Then:

\begin{theorem}[\Cref{hypsmashingfinitecd}] \label{hypissmashing}
Let $X$ be a qcqs algebraic space of finite Krull dimension, and suppose that there is a uniform bound on the virtual Galois cohomological dimension of every residue field of $X$.
Then the hypercompletion functor on \'etale sheaves of spectra over $X$ is given by tensoring with some fixed \'etale sheaf $S^h$. (Necessarily, $S^h$ is the hypercompletion of the constant sheaf on the sphere spectrum.)
\end{theorem}

This means that the collection of hypersheaves is closed under all colimits and under tensoring with any sheaf.  Adding to this the obvious fact that hypersheaves are closed under all limits, we see that the property of being a hypersheaf has very strong permanence properties.  An interesting example is the following:
\begin{corollary}
Let $X$ be as in \Cref{hypissmashing}.  Then any \'etale sheaf on $X$ which is a module over an \'etale hypersheaf is itself an \'etale hypersheaf.
In particular, any sheaf of $H\mathbb{Z}$-modules is automatically a hypersheaf.
\end{corollary}

In fact, there is a refinement of \Cref{hypissmashing} in the case where $X$ admits a uniform bound on the honest (not virtual) cohomological dimensions of the residue fields, namely \'etale hypercompletion is smashing even just in the setting of Nisnevich sheaves with finite \'etale transfers (Corollary \ref{transfers}).  Note that all manner of motivic invariants give Nisnevich sheaves with finite \'etale transfers.
Indeed, in the context of motivic homotopy theory and for modules over algebraic
cobordism, an analogous result appears as \cite[Theorem 6.27]{ELSO}. 

It is generally much easier to prove 
that an object is a sheaf than a hypersheaf (or Postnikov sheaf). 
For instance, in \cite{CMNN} it was shown that telescopically localized localizing
invariants all satisfy \'etale descent on  
qcqs spectral algebraic spaces. The argument was entirely homotopy-theoretic,
relying on the May nilpotence conjecture \cite{MNN15}. By contrast, proving hyperdescent
seems to require significantly more, and as far as we know requires some
version of the norm residue isomorphism theorem. (Note that the work of Thomason
\cite{Th84}  relies on the Merkurjev-Suslin theorem.) 
Using the above ingredients, here we prove  the following result. 
For a qcqs spectral algebraic space $Z$, we write $\perf(Z)$ for the monoidal
stable $\infty$-category of perfect modules on $Z$. 

\begin{theorem}[\Cref{mainhypLn}] 
Let $X$ be a qcqs spectral algebraic space of finite Krull dimension and with a
global bound for the virtual mod $p$ cohomological dimensions of its residue fields.
Let $\mathcal{A}$ be a
localizing invariant of $\perf(X)$-linear $\infty$-categories
which takes values in $L_n^f$-local spectra. Then the construction $Y\to X
\mapsto \mathcal{A}( \perf(Y))$ defines an \'etale hypersheaf on $X$. 
\end{theorem}

In particular, the results of the present paper strengthen those of \cite{CMNN}.  However, \cite{CMNN} also treated the case of non-\'etale extensions of
ring spectra; here we have nothing to say about them. 

\subsection*{Notation}

For a qcqs algebraic space $X$, we let $et_X$ denote the category of
\'etale, qcqs $X$-schemes. We let $X_{Nis}, X_{et}$ denote the Nisnevich
and \'etale sites, respectively; these have underlying categories $et_X$ but
different structures as sites. 

Given a Grothendieck site $\mathcal{T}$, we let $\sh(\mathcal{T})$ denote the
$\infty$-topos of sheaves of spaces on $\mathcal{T}$, and $\sh(\mathcal{T},
\sp)$ the $\infty$-category of sheaves of spectra on $\mathcal{T}$. 
We will let $\psh(\mathcal{T})$ denote presheaves on $\mathcal{T}$ and let $h$
denote 
the Yoneda embedding (either into $\psh(\mathcal{T})$ or the sheafified one into
$\sh(\mathcal{T})$). 

For a spectrum $X$ and a prime number $p$, we will typically write $X_{\hat{p}}$
for the $p$-completion of $X$. We let $L_1$ denote Bousfield localization at
complex $K$-theory $KU$ and $L_{K(1)}$ localization at the first Morava
$K$-theory (at the implicit prime $p$).  

For a field $k$ and a prime number $p$, we write $\mathrm{cd}_p(k)$ for the mod
$p$ Galois cohomological dimension of $k$ and $\mathrm{vcd}_p(k)$ for the mod
$p$ \emph{virtual} Galois cohomological dimension. Recall also that these can only
differ at $p = 2$, and $\mathrm{vcd}_p(k) = \mathrm{cd}_p(k)$ if $k$ contains
$\sqrt{-1}$ or $\sqrt[3]{-1}$.

For a set of primes $\mathcal{P}$, we say that an abelian group (or a spectrum)
$X$ 
is $\mathcal{P}$-local if every prime number outside $\mathcal{P}$ acts
invertibly on $X$.  
We let $X_{\mathcal{P}}$ denote the $\mathcal{P}$-localization of $X$ (i.e., one
inverts every prime number outside $\mathcal{P}$). 
The most important case is when $\mathcal{P}$ is the set of all primes, in which
case $\mathcal{P}$-locality is no condition at all. 

We write $K$ for non-connective $K$-theory. At times we will also use connective
$K$-theory, which we write as $K_{\geq 0}$. 

\subsection*{Acknowledgments}
We would like to thank Bhargav Bhatt, Lars Hesselholt, and Jacob Lurie for
helpful conversations related to this project. 
We also thank Matthew Morrow, Niko Naumann, and Justin Noel for the related
collaborations
\cite{CMNN, CMM} and many consequent discussions. 
We are grateful to Adriano C\'ordova, Elden Elmanto, and the two
anonymous referees for many
helpful comments on this paper. 

This work was done while the second author was a Clay Research Fellow, and is
based upon work supported by the National Science Foundation
under Grant No. DMS-1440140 while the second author was in residence at the
Mathematical Sciences Research Institute in Berkeley, California, during the
Spring 2019 semester.  
The second author also thanks the University of Copenhagen for its hospitality
during multiple visits.

\section{Generalities}

In this section we collect and review some general results on sheaves of spectra, 
hypercompletion, and Postnikov completion. 

\subsection{Prestable $\infty$-categories}
Throughout this section, we fix a Grothendieck prestable $\infty$-category
$\mathcal{C}_{\geq 0}$ \cite[Appendix C]{SAG}.  Denote by
$\mathcal{C}=\operatorname{Sp}(\mathcal{C}_{\geq 0})$ its stabilization and by
$\mathcal{C}^\heartsuit\subset \mathcal{C}_{\geq 0}$ its full subcategory of
discrete objects.  The natural functor $\mathcal{C}_{\geq 0}\rightarrow \mathcal{C}$ is fully faithful, and its essential image gives the connective part of a t-structure on $\mathcal{C}$ for which $\mathcal{C}^\heartsuit$ is the heart.   Furthermore, $\mathcal{C}^\heartsuit$ is a Grothendieck abelian category.  The t-truncation functors $\mathcal{C}\rightarrow \mathcal{C}$ will be denoted $X\mapsto X_{\leq n}$, and the homotopy object functors $\mathcal{C}\rightarrow \mathcal{C}^\heartsuit$ will be denoted $X\mapsto \pi_n X \in \mathcal{C}^\heartsuit$.  These functors are indexed by $n\in\mathbb{Z}$.

\begin{example} The primal example is where $\mathcal{C}_{\geq 0}$ is the $\infty$-category of connective spectra.  Then $\mathcal{C}$ is the $\infty$-category of spectra, $\mathcal{C}^\heartsuit$ is the category of Eilenberg-Maclane spectra in degree $0$ and thus identifies with the category of abelian groups, the truncation functors $X\mapsto X_{\leq n}$ are the usual Postnikov truncations, and the homotopy object functors $X\mapsto \pi_n X$ are the the usual homotopy group functors.
\end{example}

\begin{example}[Sheaves of spectra]
More generally, and all of our examples will essentially be of this form, let $\mathcal{T}$
be a Grothendieck site.  Then the $\infty$-category of sheaves of connective
spectra on $\mathcal{T}$ (cf.\ \cite[Sec. 1.3]{SAG}) 
is a Grothendieck prestable $\infty$-category $\mathcal{C}_{\geq 0}$.  In this case $\mathcal{C}$ identifies with the $\infty$-category of sheaves of spectra on $\mathcal{T}$ and $\mathcal{C}^\heartsuit$ identifies with the category of sheaves of abelian groups on $\mathcal{T}$.  The fully faithful inclusions $\mathcal{C}^\heartsuit\rightarrow \mathcal{C}_{\geq 0}\rightarrow \mathcal{C}$, the truncation functors $(-)_{\leq n}:\mathcal{C}\rightarrow \mathcal{C}$ and the homotopy object functors $\pi_n:\mathcal{C}\rightarrow \mathcal{C}^\heartsuit$ can all be obtained from a two-step process: first apply the corresponding functors from the previous example object-wise to the underlying presheaves, then sheafify the result.  
Recall also that in this context sheafification (from
presheaves on $\mathcal{T}$ to sheaves on $\mathcal{T}$) is t-exact; see
\cite[Rem. 1.3.2.8]{SAG}.
\end{example}

\begin{example}[Sheaves of module spectra] 
We will also need a slight variant of the above example. If $\mathcal{T}$ is a
Grothendieck site and $R$ is a connective $E_1$-ring, then 
$\sh(\mathcal{T}, \md(R)_{\geq 0})$, i.e., sheaves on $\mathcal{T}$ of
connective $R$-module spectra, is a Grothendieck prestable $\infty$-category
with heart the category of sheaves of $\pi_0(R)$-modules on $\mathcal{T}$. We
will especially consider this when $R = \mathbb{S}_{\mathcal{P}}$ for a set of primes
$\mathcal{P}$, i.e., we are considering sheaves of $\mathcal{P}$-local
connective spectra. 
\end{example}

\begin{definition}[{Compare \cite[Def. C.1.2.12]{SAG}}]
Let $X\in \mathcal{C}$.  We say that $X$ is:
\begin{enumerate}
\item \emph{acyclic} if $\pi_nX=0$ for all $n\in\mathbb{Z}$;
\item \emph{hypercomplete} if $\operatorname{Hom}_{\mathcal{C}}(U,X)=\ast$ for all acyclic
$U$.
\item \emph{Postnikov complete} if $X\overset{\sim}{\rightarrow}\varprojlim_n
X_{\leq n}$. 
\end{enumerate}

The Grothendieck prestable $\infty$-category $\mathcal{C}_{\geq 0}$ is called
\emph{separated} if all objects are hypercomplete, and \emph{complete} (or
\emph{left-complete}) if 
the natural map $\mathcal{C} \to \varprojlim_n \mathcal{C}_{\leq n}$ sending $X
\mapsto \left\{X_{\leq n}\right\}$ is an equivalence. 
\end{definition}

\begin{example}
Let $\mathcal{C}$ be the $\infty$-category of sheaves of spectra on a Grothendieck site $\mathcal{T}$. 
For $Y \in \mathcal{T}$, 
let $\Sigma^\infty_+ h_Y \in \mathcal{C}$ denote the sheafification of the
presheaf $\mathcal{T}^{op}\rightarrow\operatorname{Sp}$ given by $Z \mapsto \Sigma^\infty_+
\hom_{\mathcal{T}}(Z, Y)$. 
If $U_\bullet \rightarrow X$ is a hypercover in $\mathcal{T}$ (where each $U_i$
is a coproduct of objects in $\mathcal{T}$), then
$$\varinjlim_{n\in \Delta^{op}} \Sigma^\infty_+ h_{U_n} \rightarrow
\Sigma^\infty_+ h_X$$
is a $\pi_\ast$-isomorphism in $\mathcal{C}$ \cite[Lemma 6.5.3.11]{HTT}.  Thus
if $\mathcal{F}$ is a sheaf of spectra which is hypercomplete, then
$\mathcal{F}$ is a \emph{hypersheaf}, meaning 
$\mathcal{F}(X)\overset{\sim}{\rightarrow}
\varprojlim_{n\in\Delta}\mathcal{F}(U_n)$ for all hypercovers as above.  The
converse also holds, cf.\ \cite{DHI04}, \cite{TV03}, \cite[Cor. 6.5.3.13]{HTT}
for sheaves of spaces.
Note also that a sheaf $\mathcal{F}$ of spectra is hypercomplete if and only if the
underlying sheaf of spaces is hypercomplete (or equivalently a hypersheaf), cf.\ \cite[Prop. 1.3.3.3]{SAG}. 
Note that \cite[Prop. 1.3.3.3]{SAG} shows that the subcategory of
hypercomplete sheaves of spectra on $\mathcal{T}$ is also intrinsically described as
sheaves of spectra on the hypercompletion of the $\infty$-topos of sheaves of
spaces on $\mathcal{T}$. 

We will not use this remark; we simply work directly with the notion of hypercompleteness as in the above definition.\end{example}

\begin{lemma}
\label{hypvsPostcompl}
\begin{enumerate}
\item The collection of hypercomplete objects of $\mathcal{C}$ is closed under all limits.
\item Every t-bounded above object of $\mathcal{C}$ is hypercomplete.
\item Postnikov complete objects are hypercomplete.  More generally, for all $X\in\mathcal{C}$, the object $\varprojlim_n X_{\leq n}$ is hypercomplete.
\end{enumerate}
\end{lemma}
\begin{proof}
Claim 1 is clear from the definition.  For claim 2, it suffices to see that if $X$ is acyclic, then $X_{\leq n}=0$ for all $n\in\mathbb{Z}$.  Since the $t$-structure on $\mathcal{C}=\operatorname{Sp}(\mathcal{C}_{\geq 0})$ is right complete by construction, we can assume $X$ is t-bounded below.  Then $X_{\leq n}$ is t-bounded with vanishing homotopy objects, and hence is $0$, as desired.  Claim 3 follows immediately from claims 1 and 2.
\end{proof}

\begin{example} 
Let $\mathcal{A}$ be a Grothendieck abelian category, so then
the connective part $\mathcal{D}(\mathcal{A})_{\geq 0}$ of the derived $\infty$-category $\mathcal{D}(\mathcal{A})$ is a Grothendieck
prestable $\infty$-category \cite[Example C.1.4.5]{SAG}. 
In this case, $\mathcal{D}(\mathcal{A})$ has no nonzero acyclic objects, or equivalently
all objects are hypercomplete. However, 
objects of $\mathcal{D}(\mathcal{A})$ need not be Postnikov complete, cf.\
\cite{Neeman} for examples. 
\end{example}

\begin{definition}[Cohomological dimension]
\begin{enumerate}
\item Let $U\in \mathcal{C}_{\geq 0}$ and $A\in \mathcal{C}^\heartsuit$.  For $i\geq 0$, define the \emph{$i^{th}$ cohomology of $U$ with coefficients in $A$} to be the abelian group
$$H^i(U;A) = [U,\Sigma^i A]:=\pi_0 \map_{\mathcal{C}}(U,\Sigma^iA).$$

\item Let $\mathcal{A}$ be a collection of objects of $\mathcal{C}^\heartsuit$.  For $U\in\mathcal{C}_{\geq 0}$ and $d\in\mathbb{N}$, we say that $U$ has \emph{cohomological dimension $\leq d$ with $\mathcal{A}$-coefficients} if $H^i(U;A)=0$ for all $A\in\mathcal{A}$ and $i>d$.

\item Let $\mathcal{A}$ be a collection of objects of $\mathcal{C}^\heartsuit$.  For $d\in\mathbb{N}$, we say that $\mathcal{C}_{\geq 0}$ \emph{has enough objects of cohomological dimension $\leq d$ with $\mathcal{A}$-coefficients} if for any $X\in\mathcal{C}_{\geq 0}$, there exists a map $f:U\rightarrow X$ in $\mathcal{C}_{\geq 0}$ such that $\pi_0f$ is an epimorphism and $U$ has cohomological dimension $\leq d$ with $\mathcal{A}$-coefficients.\end{enumerate}
\label{cohdimensiondef}
\end{definition}

If we leave out ``with $\mathcal{A}$ coefficients", we implicitly mean to take $\mathcal{A}=\mathcal{C}^\heartsuit$.

\begin{remark}
An arbitrary coproduct of objects of $\mathcal{C}_{\geq 0}$ of cohomological
dimension $\leq d$ also has cohomological dimension $\leq d$.  From this one
sees that if $\mathcal{C}_{\geq 0}$ is generated under colimits by objects of
cohomological dimension $\leq d$, then $\mathcal{C}_{\geq 0}$ has enough objects
of cohomological dimension $\leq d$ (compare \cite[Def. C.2.1.1]{SAG}).
\end{remark}

\begin{proposition}
\label{hypforfinitecd}
Let $X\in\mathcal{C}$.  Suppose there exists a $d\geq 0$ such that $\mathcal{C}_{\geq 0}$ has enough objects of cohomological dimension $\leq d$ with $\{\pi_nX\}_{n\in\mathbb{Z}}$-coefficients.  Then:
\begin{enumerate}
\item The map $\alpha:X\rightarrow \varprojlim_n X_{\leq n}$ is a $\pi_\ast$-isomorphism.
\item $X$ is hypercomplete if and only if $X$ is Postnikov complete.
\end{enumerate}
\end{proposition}

\begin{proof}
First we note 1 $\Rightarrow$ 2.  Indeed, assuming 1, the fiber of $\alpha$ is
acyclic; but if $X$ is hypercomplete the fiber will also be hypercomplete, hence
zero, so $\alpha$ is an equivalence and $X$ is Postnikov complete.  The converse
is true in complete generality by \Cref{hypvsPostcompl}.

Now we prove 1.  Replacing $X$ by its shifts, it suffices to show that $\alpha$ is an isomorphism on $\pi_0$.  Suppose $U\in\mathcal{C}_{\geq 0}$ is of cohomological dimension $\leq d$ with $\pi_\ast X$-coefficients.  Then for all $n\in\mathbb{Z}$, the fiber of
$$\operatorname{Hom}_{\mathcal{C}}(U,X_{\leq n+1})\rightarrow
\operatorname{Hom}_{\mathcal{C}}(U,X_{\leq n})$$
lies in $\operatorname{Sp}_{\geq n-d}$.  Taking the inverse limit along these
maps for all $n\geq d+1$, it follows that if we set
$F$ to be the fiber of $\beta:\varprojlim_n X_{\leq n}\rightarrow X_{\leq d+1}$,
then $[U,F]=[U, \Sigma F ] = 0$.

On the other hand, by hypothesis we can choose $U \in \mathcal{C}_{\geq 0}$ of
cohomological dimension $\leq d$ (with $\pi_* X$-coefficients) so that there is a map
$U\rightarrow F_{\geq 0}$ which is an epimorphism on $\pi_0$.  It follows that
$\pi_0F=0$; similarly $\pi_{-1} F =0 $.  Hence $\pi_0\beta$ is an isomorphism.
But $\pi_0\beta$ is isomorphic to   $\pi_0\alpha$, and therefore the latter is an isomorphism, as desired.
\end{proof}

\begin{example}[{Compare \cite[Th.~1.37]{MV99}}]
\label{cdsite}
Let $\mathcal{T}$ be a Grothendieck site, and take $\mathcal{C}$ to be 
sheaves of spectra on $\mathcal{T}$. 
Suppose $x\in\mathcal{T}$ is of cohomological dimension $\leq d$ in the sense that for every sheaf of abelian groups $A$ on $\mathcal{T}$, the abelian groups $H^i(x;A)=R^i\Gamma(x;A)$ vanish for $i>d$.  
By \cite[Cor. 2.1.2.3]{SAG}, the derived $\infty$-category 
of the category of abelian sheaves on $\mathcal{T}$ is identified with the
$\infty$-category of hypercomplete sheaves of $H\mathbb{Z}$-modules on $\mathcal{T}$. 
In particular, $\infty$-topos cohomology is identified with derived functor
cohomology. 
It follows that the object $h_x\in\mathcal{C}_{\geq 0}$ corepresenting sections over $x$ has cohomological dimension $\leq d$ in the above sense.

Now assume every object of $\mathcal{T}$ admits a covering by objects $x$ of
cohomological dimension $\leq d$.  It follows that $\mathcal{C}_{\geq 0}$ has
enough objects of cohomological dimension $\leq d$.  Thus the previous
proposition applies, and we conclude that Postnikov complete is equivalent to hypercomplete for sheaves of spectra on $\mathcal{T}$.
We have an analogous statement for $\mathcal{P}$-local cohomological dimension, where
$\mathcal{P}$ is some set of primes.  
\end{example}

\begin{example} 
\label{modpdimensionex}
As a special case of the above, 
let $R$ be a commutative $\mathbb{F}_p$-algebra, and consider the \'etale site
$\spec(R)_{et}$. Recall that for any $p$-torsion \'etale sheaf $\sF$ of abelian groups on
$\spec(R)$, one has
$H^*(\spec(R)_{et}, \sF) = 0$ for $\ast > 1$. 
Compare \cite[Exp. X, Theorem 5.1]{SGA4} in the noetherian case and the general case
follows from compatibility with filtered colimits.
It follows that for a sheaf $\sF$ of $p$-complete spectra on $\spec(R)_{et}$,
then $\sF$ is hypercomplete if and only if $\sF$ is Postnikov complete.

\end{example}

The previous proposition will cover all of our cases of interest.  Thus
practically speaking, hypercomplete and Postnikov complete are equivalent
notions.  Nonetheless, they serve different purposes.  Postnikov completeness in
some sense reduces the study of sheaves of spectra to that of sheaf cohomology,
i.e., one obtains the following (standard) descent spectral sequence:

\begin{proposition}
\label{descentss}
Let $U\in\mathcal{C}_{\geq 0}$ and $X\in \mathcal{C}$.  Then there is a conditionally convergent (cohomoligcally indexed) spectral sequence
$$E^2_{p,q} = H^{p}(U;\pi_q X)\Rightarrow \pi_{q-p} \map_{\mathcal{C}}(U,\varprojlim_n X_{\leq n})$$
\end{proposition}
\begin{proof}
Consider the filtered spectrum
$$\ldots \rightarrow \hom_{\mathcal{C}}(U,X_{\leq q})\rightarrow
\hom_{\mathcal{C}}(U,X_{\leq q-1})\rightarrow\ldots$$
indexed by $q\in\mathbb{Z}$.  The colimit is $0$, because the homotopy $\pi_\ast
\hom_{\mathcal{C}}(U,X_{\leq q})$ vanishes in the range $\ast>q$ which covers
everything as $q\rightarrow -\infty$.  Thus the associated spectral sequence
converges conditionally to the homotopy of the inverse limit, which is exactly
$\hom_{\mathcal{C}}(U,\varprojlim_nX_{\leq n})$.  The $E^1$-term is given by the homotopy of the fibers of the maps constituting the above filtered spectrum, and is therefore as claimed because the fiber of $X_{\leq q}\rightarrow X_{\leq q-1}$ is $\Sigma^{q} \pi_qX$.  We reindex $E^1$ to $E^2$ to fit the general convention.
\end{proof}

On the other hand, hypercompleteness can be studied using the formally
convenient machinery of Bousfield localizations, cf.\ \cite[Sec. 5.2.7]{HTT}:

\begin{proposition}[{Cf.\ \cite[Sec. C.3.6]{SAG}}]
\label{hypcompletion}
The fully faithful inclusion $\mathcal{C}^h\rightarrow \mathcal{C}$ of the full subcategory $\mathcal{C}^h\subset\mathcal{C}$ spanned by the hypercomplete objects has an accessible left adjoint.  We denote this left adjoint by $X\mapsto X^h$ and call it hypercompletion.
\end{proposition}
\begin{proof}
By the general machinery \cite[Sec. 5.5.4]{HTT}, it suffices to see that the
acyclic objects form an accessible full subcategory of $\mathcal{C}$ which is
closed under colimits.  It is accessible because it is the kernel of the
collected truncation functors $\mathcal{C}\rightarrow
\prod_{n\in\mathbb{Z}}\mathcal{C}$ and \cite[Prop. 5.4.7.3]{HTT}.  To see that it is closed under colimits, we
need to check closure under cofibers and direct sums.  The first follows from
the 5-lemma on homotopy groups.  The second follows from the fact that direct sums
are t-exact in a Grothendieck prestable $\infty$-category, since filtered
colimits are t-exact. 
\end{proof}

In particular we see that hypercompletion is an idempotent exact functor; neither idempotency nor exactness is clear for the Postnikov completion (inverse limit over the Postnikov tower) in general.

\begin{remark}
\label{hypisPostnikov}
By the general theory, the hypercompletion functor comes with a
natural transformation $X\rightarrow X^h$.  Moreover this pair of functor and
natural transformation can be characterized objectwise, as the unique
$\pi_\ast$-isomorphism from $X$ to a hypercomplete object.  In particular, if
$\mathcal{C}_{\geq 0}$ has enough objects of cohomological dimension $\leq d$
for some $d\in\mathbb{N}$, then by \Cref{hypforfinitecd} the hypercompletion
of $X \in \mathcal{C}$ is given by the Postnikov completion $X\mapsto \varprojlim_n X_{\leq n}$.\end{remark}

The construction of hypercompletion also has a universal property at the
categorical level; for this, we
note that the hypercomplete objects inherit a $t$-structure. 

\begin{construction}[{The universal property of $\mathcal{C}^h$,
\cite[Sec.~C.3.6]{SAG}}] 
For a Grothendieck prestable $\infty$-category $\mathcal{C}_{\geq 0}$, the
the hypercompletion functor $\mathcal{C}_{\geq 0} \to \mathcal{C}_{\geq 0}^h$
has a universal property: it is the initial exact, left adjoint functor to a
separated Grothendieck prestable $\infty$-category. 
\end{construction} 

In general, one can always form a general Postnikov completion construction on
$\mathcal{C}$, but it need not be given by any type of Bousfield localization.
This makes the hypercompletion slightly easier to work with in practice. 

\begin{construction}[{Cf.\ \cite[Prop. C.3.6.3]{SAG}}]
\label{completionofprestable} 
Let $\mathcal{C}_{\geq 0}$ be a Grothendieck prestable $\infty$-category. 
Then the homotopy limit $\widehat{\mathcal{C}_{\geq 0}} = \varprojlim_n
\mathcal{C}_{[0, n]}$ is a Grothendieck prestable $\infty$-category too  which
is complete. The functor $\mathcal{C}_{\geq 0}\to
\widehat{\mathcal{C}_{\geq 0}}$ is the universal cocontinuous exact functor out of
$\mathcal{C}_{\geq 0}$ into a  complete Grothendieck prestable $\infty$-category, and $\widehat{\mathcal{C}_{\geq 0}}$ is called the
completion (or left completion or Postnikov completion) of $\mathcal{C}_{\geq 0}$. 
\end{construction}

Explicitly, an object of $\widehat{\mathcal{C}_{\geq 0}}$ is a compatible (under
truncation) collection of objects $X_n \in \mathcal{C}_{[0, n]}$. 
The functor $\mathcal{C}_{\geq 0} \to \widehat{\mathcal{C}_{\geq 0}}$ 
has a right adjoint, which carries the compatible collection
$\left\{X_n\right\}$ to $\varprojlim_n X_n$. 
In particular, the composite of the right and left adjoints implements the
functor which carries $ Y \in \mathcal{C}_{\geq 0}$ to the limit of its Postnikov tower. 
The condition that $\mathcal{C}_{\geq 0}$ should be complete 
in particular implies that every object of $\mathcal{C}_{\geq 0}$ is Postnikov
complete. 

\begin{example} 
It is possible for a Grothendieck prestable $\infty$-category 
$\mathcal{C}_{\geq 0}$ 
to have the property that every object is Postnikov complete but 
such that $\mathcal{C}_{\geq 0}$ is not complete in the above sense. 
For example, we can take $\mathcal{C}_{\geq 0} \subset \sp_{\geq 0}$ to be the
subcategory of $K(1)$-acyclic, connective spectra. Since the subcategory of
$K(1)$-acyclic spectra inside all spectra is closed under truncations and
colimits, one sees that $\mathcal{C}_{\geq 0}$ is Grothendieck prestable;
however, its left completion is simply $\sp_{\geq 0}$, as the functor
$\mathcal{C} \to \sp_{\geq 0}$ induces an equivalence on truncated objects. 
\end{example} 

Often, one can compare the Postnikov completion and hypercompletion at the
level of Grothendieck prestable $\infty$-categories; we will
describe such comparisons either in cases of finite cohomological dimension or
in certain infinitary situations. 
For the next result, compare also \cite[Prop.~1.2.1.19]{HA} (which essentially
treats the case $d =0 $; the general case is similar). 
\begin{proposition} 
\label{countablepropcond}
Let $\mathcal{C}_{\geq 0}$ be a Grothendieck prestable $\infty$-category with
stabilization $\mathcal{C}$.
Suppose that there is a $d\geq 0$ such that a countable product of truncated objects of $\mathcal{C}_{\geq 0}$ belongs to
$\mathcal{C}_{\geq -d}$.
Then the hypercompletion and Postnikov completion of $\mathcal{C}_{\geq
0}$ agree. 
\end{proposition} 
\begin{proof} 
Without loss of generality, we can assume that $\mathcal{C}_{\geq 0}$ itself is
hypercomplete. 
Let $\left\{X_n\right\}_{n \geq 0}$ be a tower of objects in
$\mathcal{C}_{\geq 0}$ such
that 
for each $n$, 
$X_n$ is $n$-truncated, and the map $(X_n)_{\leq m} \to X_m$ is an equivalence
for $m \leq n$. It suffices to show that the inverse limit 
$X \stackrel{\mathrm{def}}{=}\varprojlim_n X_n$ (computed in $\mathcal{C}$) has
the property that the map 
$X \to X_n$ is given by $n$-truncation for any $n \geq 0$ (in particular, $X \in
\mathcal{C}_{\geq 0}$). 

In fact, consider the homotopy fiber $F_{n+d+1}$ of $X \to X_{n + d +1}$. 
This is the homotopy limit in $\mathcal{C}$ of a countable tower of truncated 
objects of $\mathcal{C}_{\geq n + d + 2}$ and therefore (by our
hypotheses) belongs to $\mathcal{C}_{\geq n + 2}$. 
It follows from the fiber sequence $F_{n+d+1} \to X \to X_{n+d+1}$ that 
the map $X \to X_n$ induces an equivalence on $n$-truncations, as desired. 
\end{proof}

\begin{corollary} 
\label{enoughcdPostnikovcomplete}
The hypotheses of \Cref{countablepropcond} are satisfied if $\mathcal{C}_{\geq
0}$ has enough objects of cohomological dimension $\leq d$ in the sense of
\Cref{cohdimensiondef}.  Thus, the hypercompletion and Postnikov completion of
$\mathcal{C}_{\geq 0}$ agree. 
\end{corollary}
\begin{proof} 
Let $U$ be an object of cohomological dimension $d$. 
Let $\left\{Y_i\right\}_{i \in \mathbb{Z}_{\geq 0}}$ be a countable family of
truncated objects of $\mathcal{C}_{\geq 0}$. 
Then our assumptions show that 
the mapping spectra $\hom_{\mathcal{C}}(U, Y_i)$ belong to $\sp_{\geq -d}$;
taking products, we find that $\hom_{\mathcal{C}}( U, \prod_i Y_i) \in
\sp_{\geq -d}. $
This in particular implies that any map from $U$ to $\Sigma^{d+j}\prod_i Y_i$ is
nullhomotopic for any $j \geq 1$, whence the desired connectivity assertion.  The last claim then follows from \Cref{countablepropcond}. 
\end{proof} 

\begin{example} 
Let $\mathcal{A}$ be a Grothendieck abelian category. 
Suppose that (for some fixed $d \geq 0$) for every object $X \in \mathcal{A}$,
there exists a surjection $Y \twoheadrightarrow X$ such that $Y$ has
cohomological dimension $\leq d$. Then $\mathcal{D}(\mathcal{A})_{\geq 0}$ is Postnikov
complete. 
This follows from \Cref{enoughcdPostnikovcomplete}, noting that for any object
$Z \in \mathcal{D}(\mathcal{A})_{\geq 0}$, there exists an object $X \in
\mathcal{A}$ (embedded as the heart) with a map $X \to Z$ inducing a surjection
on $\pi_0$, i.e., $\mathcal{D}(\mathcal{A})_{\geq 0}$ is $0$-complicial 
\cite[Def.~C.5.3.1]{SAG}. 
\end{example}

\subsection{Smashing hypercompletion}

Fix a Grothendieck prestable $\infty$-category $\mathcal{C}_{\geq 0}$ as in the previous
subsection. 
Now suppose that $\mathcal{C}_{\geq 0}$ has a symmetric monoidal structure
$\otimes$ which commutes with colimits in each variable.  For example, if
$\mathcal{C}_{\geq 0}$ is given as sheaves of connective spectra on a
Grothendieck site, then there is such a symmetric monoidal structure $\otimes$.
It can be produced by sheafifying the usual section-wise smash product, or
(equivalently) by stabilizing the cartesian product symmetric monoidal structure
on the associated $\infty$-topos  \cite[Sec. 1.3.4]{SAG}.
We also denote by $\otimes$ the unique extension of this symmetric monoidal structure to $\mathcal{C}=\mathcal{C}_{\geq 0}\otimes\operatorname{Sp}$ having the same colimit preserving property.

\begin{lemma}
\label{tensorproductCh}
\begin{enumerate}
\item Let $X,Y\in\mathcal{C}$.  If $X$ is acyclic, then so is $X\otimes Y$.
\item There is a unique symmetric monoidal structure $\otimes^h$ on $\mathcal{C}^h$ making hypercompletion $\mathcal{C}\rightarrow\mathcal{C}^h$ into a symmetric monoidal functor, given on objects by $X\otimes^h Y = (X\otimes Y)^h$.
\item Every $X\in\mathcal{C}^h$ has the unique and functorial structure of a $1^h$-module, where $1$ is the unit of $\mathcal{C}$.
\end{enumerate}
\end{lemma}
\begin{proof}
Claims 2 and 3 follow formally from claim 1, cf.\ \cite[Prop. 2.2.1.9]{HA}.  To prove claim 1, since the t-structure is right complete, we can assume $Y$ is t-bounded below, say $Y\in\mathcal{C}_{\geq m}$.  Since $X$ is acyclic, it lies in $\mathcal{C}_{\geq n}$ for any $n$.  Thus $X\otimes Y\in \mathcal{C}_{\geq n+m}$ for any $n$, hence is acyclic.
\end{proof}

\begin{lemma}
\label{whenishypsmashing}
The following conditions are equivalent:
\begin{enumerate}
\item The full subcategory $\mathcal{C}^h$ is closed under colimits and tensoring with any $X\in\mathcal{C}$.
\item For any $X\in\mathcal{C}$, the object $1^h\otimes X$ is hypercomplete.
\item For any $X\in \mathcal{C}$, the map $X\rightarrow 1^h\otimes X$ of tensoring with $1\rightarrow 1^h$ gives the hypercompletion of $X$.
\item The forgetful functor $\operatorname{Mod}_{1^h}(\mathcal{C})\rightarrow\mathcal{C}$ is fully faithful with essential image $\mathcal{C}^h$.
\item Every $X\in\mathcal{C}$ which is a module over an algebra $A\in\operatorname{Alg}(\mathcal{C}^h)$ also lies in $\mathcal{C}^h$.
\end{enumerate}
\label{lem:smashinghyp}
\end{lemma}
\begin{proof}
1 $\Rightarrow$ 2 is trivial.

Suppose 2.  Then $1^h\otimes X$ is hypercomplete.  On the other hand the fiber
of $X\rightarrow 1^h\otimes X$ is acyclic by \Cref{tensorproductCh}, whence 3.

Suppose 3.  Full faithfulness of the forgetful functor is equivalent to
$1^h\otimes 1^h = 1^h$.  This is the special case $X=1^h$ of 3.  For the
essential image claim, suppose $X$ admits a $1^h$-module structure.  Then $X$ is
a retract of $X\otimes 1^h=X^h$, hence $X$ is hypercomplete.  Conversely if $X$
is hypercomplete, then $X$ is a module over $1^h$ by \Cref{tensorproductCh}.

Suppose 4.  Then 1 follows because the forgetful functor in 4 commutes with colimits and tensoring with $X\in \mathcal{C}$.

Thus all of 1 through 4 are equivalent, and it suffices to show they are also equivalent to 5.

Since a module over $A$ is a fortiori a module over $1^h$, we see that even just the essential image claim of 4 implies 5.  Conversely if 5 holds, then since $1^h\otimes X$ is a module over $1^h$ we clearly have 2.
\end{proof}

\begin{definition}[Smashing hypercompletions]
Suppose that $\mathcal{C}$ satisfies the conditions of \Cref{lem:smashinghyp}.
Then we say that \emph{hypercompletion is smashing} for $\mathcal{C}$ (compare
\cite{Rav84}).
\end{definition}

This represents a desirable formal situation, only one step away from the ideal situation where $\mathcal{C}^h=\mathcal{C}$.

Next we provide a local-global criterion, essentially that of
\cite[Prop. 1.3.3.6]{SAG}.  To give the setup, let $\mathcal{X}$ be an
$\infty$-topos.  Recall that a sheaf on $\mathcal{X}$ with values in a
presentable $\infty$-category $\mathcal{Y}$ can be defined simply as a
limit-preserving functor $\mathcal{X}^{op}\rightarrow\mathcal{Y}$ \cite[Sec.
1.3.1]{SAG}.  In particular, we have the $\infty$-category of sheaves of connective spectra on $\mathcal{X}$, which is a Grothendieck prestable $\infty$-category with stabilization the $\infty$-category of sheaves of spectra on $\mathcal{X}$ and heart the category of sheaves of abelian groups on $\mathcal{X}$.

For $x\in\mathcal{X}$, we note by $x^\ast$ the pullback functor on sheaves of spectra from $\mathcal{X}$ to the slice topos $\mathcal{X}_{/x}$, defined by
$$(x^\ast\mathcal{F})(y\rightarrow x) = \mathcal{F}(y).$$

\begin{proposition}
\label{locglobhyp}
Let $\mathcal{X}$ be an $\infty$-topos, and $S$ a collection of objects which covers $\mathcal{X}$, i.e.\ the terminal object $\ast$ lies in the smallest subcategory of $\mathcal{X}$ containing $S$ and closed under colimits.
\begin{enumerate}
\item If $\mathcal{F}$ is a sheaf of spectra on $\mathcal{X}$, then $\mathcal{F}$ is hypercomplete if and only if $x^\ast\mathcal{F}$ is hypercomplete for all $x\in S$.
\item If $\mathcal{F}$ is a sheaf of spectra on $\mathcal{X}$, then $\mathcal{F}$ is Postnikov complete if and only if $x^\ast\mathcal{F}$ is Postnikov complete for all $x\in S$.
\item If hypercompletion is smashing for sheaves of spectra on $\mathcal{X}_{/x}$ for all $x\in S$, then it is also smashing for sheaves of spectra on $\mathcal{X}$.
\end{enumerate}
\end{proposition}
\begin{proof}
First we show that $x^\ast$ preserves hypercompleteness for any $x\in\mathcal{X}$.  For this, note that $x^\ast$ is t-exact and preserves limits.  Its left adjoint $x_!$ then necessarily sends $n$-connective objects to $n$-connective objects for all $n$, and in particular preserves acyclics.  It follows that $x^\ast$ does indeed preserve hypercompleteness.

Now suppose that $\mathcal{F}$ is such that $x^\ast\mathcal{F}$ is hypercomplete for all $x\in S$. Let $A$ be acyclic; we will check that the mapping sheaf $\underline{\operatorname{Map}}(A,\mathcal{F})$ is terminal.  It is enough to check this on pullback to any $x\in S$, because $S$ covers.  However, pullback is t-exact and hence preserves acyclic objects, whence the claim.

For 2, since $x^\ast$ preserves limits and is t-exact, it preserves Postnikov towers and their limits.  On the other hand, since $S$ covers, the map $\mathcal{F}\rightarrow\varprojlim_n\mathcal{F}_{\leq n}$ is an equivalence if and only if it is an equivalence on $x^\ast$ for all $x\in S$. Combining gives the claim.

Finally we prove 3.  We need to see that $\mathcal{F}\otimes \mathcal{G}$ is
hypercomplete for all sheaves of spectra $\mathcal{F},\mathcal{G}$ on
$\mathcal{X}$ such that $\mathcal{F}$ is hypercomplete, by part 2 of
\Cref{whenishypsmashing}.  By the first claim this can be checked after applying $x^\ast$.  But $x^\ast$ is symmetric monoidal and preserves hypercompleteness, so we do indeed reduce to hypercompletion being smashing on each $\mathcal{X}_{/x}$, as claimed.
\end{proof}

\begin{remark}
We do not know the whether the converse to 3 holds, that is whether having smashing hypercompletion passes to slice topoi.  If it does not, as seems likely, then the notion of \emph{locally} having smashing hypercompletion, where one requires all slice topoi to have smashing hypercompletion, is probably more amenable in general than the notion of just having smashing localization.  In all the cases in this paper where we prove that hypercompletion is smashing, it follows immediately from the statement that the same is true locally, so this issue does not concern us much here.
\end{remark}

Recall from \Cref{whenishypsmashing} that hypercompletion is smashing for
$\mathcal{C}$ if and only if the full subcategory
$\mathcal{C}^h\subset\mathcal{C}$ of hypercomplete objects is closed under
colimits and $-\otimes X$ for all $X\in\mathcal{C}$.  Using the connection with
Postnikov completion given by \Cref{hypisPostnikov}, one can often check the first condition using abelian cohomology (compare \cite[Prop. 1.39]{Th85}).  Here we no longer need the symmetric monoidal structure on $\mathcal{C}$.
First we need an auxiliary notion.

\begin{definition} 
Let $F: \mathcal{C} \to \mathcal{D}$ be  a left adjoint functor between presentable 
$\infty$-categories with right adjoint $G: \mathcal{D} \to \mathcal{C}$. We say that 
this adjunction (or simply the left adjoint $F$) is \emph{strongly generating} if given a diagram 
$f: K^{\rhd} \to \mathcal{D}$ (the right cone on a simplicial set $K$, cf.\ \cite[Notation 1.2.8.4]{HTT}) such that $G(f)$ is a colimit diagram in
$\mathcal{C}$, then $f$ is a
colimit diagram. 
\end{definition} 

Consider a strongly generating adjunction $(F, G): \mathcal{C} \rightleftarrows
\mathcal{D}$ as above. In this case, 
the right adjoint $G$ is conservative since one can apply the definition to $K =
\ast$. It follows that the image of $F$ generates $\mathcal{D}$ under colimits. 
An important example of strong generation, though not the one relevant for our purposes, arises from the theory of 
Grothendieck prestable $\infty$-categories. 
Let $\mathcal{D}_{\geq 0}$ be a separated Grothendieck prestable
$\infty$-category, and let $X \in \mathcal{D}_{\geq 0}$ be an object such that
for any $Y \in \mathcal{D}_{\geq 0}$, there exists a set $I$ and a map
$\bigoplus_I X \to Y$ which induces a surjection on $\pi_0$. 
Then the functor $\mathrm{Sp}_{\geq 0} \to \mathcal{D}_{\geq 0}$ given by
tensoring with $X$ is strongly generating, by the $\infty$-categorical
Gabriel-Popescu theorem \cite[Theorem C.2.1.6]{SAG}.   In the following, we will use strong generation to deduce 
hypercompleteness. 

\begin{proposition}
\label{critforhypcolimits}
Let $\mathcal{C}_{\geq 0}$ be a Grothendieck prestable $\infty$-category with stabilization $\mathcal{C}$. Suppose there exist a $d\geq 0$ and a functor $h_{(-)}:\mathcal{T}\rightarrow\mathcal{C}_{\geq 0}$ from a small $\infty$-category $\mathcal{T}$ such that:
\begin{enumerate}
\item The induced functor $\operatorname{PSh}(\mathcal{T};\operatorname{Sp})\rightarrow\mathcal{C}$ is strongly generating;
\item For every $U\in\mathcal{T}$, the cohomology functor $H^n(U;-)=[h_U,\Sigma^n(-)]:\mathcal{C}^\heartsuit\rightarrow\operatorname{Ab}$ commutes with filtered colimits for all $n$ and vanishes for $n>d$.
\end{enumerate}

Then the collection of hypercomplete objects $\mathcal{C}^h\subset\mathcal{C}$ is closed under all colimits.
Moreover, the hypercompletions of the objects $h_t, t \in \mathcal{T}$ are
compact in $\mathcal{C}^h$. 
\end{proposition}

For instance, if $\mathcal{T}$ is a Grothendieck site and $\mathcal{C}_{\geq
0}$ is given by sheaves of connective spectra, then 1 is automatic;  if 2
is satisfied,
it follows that  hypercomplete sheaves of spectra on $\mathcal{T}$ form a full 
subcategory of $\psh(\mathcal{T}, \sp)$ which is closed under all colimits.

\begin{proof}
Since $\mathcal{C}^h\subset\mathcal{C}$ is a stable subcategory, it is closed under finite colimits; thus we need only check closure under filtered colimits.
Since each $h_U$ is of cohomological dimension $\leq d$ by hypothesis 2, the
fact that the $h_U$ generate $\mathcal{C}_{\geq 0}$ under colimits (a weak form
of condition 1) implies that $\mathcal{C}_{\geq 0}$ has enough objects of
cohomological dimension $\leq d$.  Thus an object of $\mathcal{C}$ is
hypercomplete if and only if it is Postnikov complete (\Cref{hypforfinitecd}).
We write $(-)^{\pi}$ for the inverse limit over the Postnikov tower, or
equivalently the hypercompletion (by \Cref{hypforfinitecd} again). 
It suffices to show that if $i\mapsto X_i$ is a filtered system of Postnikov complete objects of $\mathcal{C}$, then the map
$$\varinjlim_iX_i\rightarrow(\varinjlim_iX_i)^\pi,$$
from $\varinjlim_iX_i$ to the inverse limit over its Postnikov tower, is an equivalence.

For this, note that for every $U\in\mathcal{T}$ the spectral sequence  
of \Cref{descentss}
(applied to each $X_i$ as well as their colimit) and the hypothesis 2 imply that the natural map
$$\varinjlim_i X_i(U)\rightarrow \left( (\varinjlim_iX_i)^\pi\right)(U)$$
is an equivalence.  This shows that the restriction of $(\varinjlim_iX_i)^\pi$ to $\operatorname{PSh}(\mathcal{T};\operatorname{Sp})$ identifies with the filtered colimit of the restrictions of the $X_i$ to $\operatorname{PSh}(\mathcal{T};\operatorname{Sp})$.  
Since the adjunction was strongly generating, we can conclude that the
hypercomplete objects are closed under colimits. 
The same spectral sequence argument also shows that the hypercompletion of
$h_U$ is compact in $\mathcal{C}^h$ for each $U \in
\mathcal{T}$, using again the convergence of Postnikov towers in
$\mathcal{C}^h$. 
\end{proof}

\subsection{Finite cohomological dimension; nilpotence criteria}

Here we specialize to the following setting, instances of which will be the
focus of the remainder of the paper. 
Recall the following definition as in \cite[A.3.1]{SAG}; in this paper, by
convention, we will always assume in addition the existence of a terminal object in our
sites. 

\begin{definition}[Finitary sites] \label{finitarysitedef} 
A \emph{finitary site} is a small $\infty$-category $\mathcal{T}$ with all finite limits equipped with a Grothendieck topology such that every covering sieve admits a refinement which is generated by a finite number of elements. 
\end{definition}

Given a finitary site $\mathcal{T}$, we have a well-behaved theory of sheaves
on it, forming an $\infty$-topos $\mathrm{Sh}(\mathcal{T})$; 
similarly we can form an $\infty$-category of sheaves of spectra
$\mathrm{Sh}(\mathcal{T}, \sp)$. 
We are interested in questions of hyperdescent or Postnikov completeness in
such $\infty$-categories; these two notions are the same under finite
cohomological dimension assumptions.

\begin{definition}
Given a finitary site $\mathcal{T}$, we say that $\mathcal{T}$ has
\emph{cohomological dimension $\leq d$} if, for
each $x \in \mathcal{T}$, the object $\Sigma^\infty_+ h_x \in
\mathrm{Sh}(\mathcal{T}, \sp_{\geq 0})$ (i.e., sheaves of connective spectra on
$\mathcal{T}$) has cohomological dimension $\leq d$ (cf.\ \Cref{cdsite}).  
For a set of primes $\mathcal{P}$, we similarly have a notion of
\emph{$\mathcal{P}$-local cohomological dimension} of a site $\mathcal{T}$.
\end{definition} 

\newcommand{\pre}{\mathrm{pre}}
Our discussion of hypercompleteness will take place primarily in the setting of
finitary sites of bounded cohomological dimension.
In this case, 
hypercompleteness is equivalent to Postnikov completeness by
\Cref{hypforfinitecd}, cf.\ also \cite[Prop. 3.3]{Ja87} or \cite[Prop.
3.20]{Mi97} for equivalent results.   Note also that by the local-global principle \Cref{locglobhyp}, this and several of the results below also hold in the more general setting of \emph{local } finite cohomological dimension.

Even under this finite-dimensionality hypothesis, the condition
that a sheaf of spectra should be hypercomplete seems slightly subtle, and we do
not know whether hypercompletion is smashing in this generality (although it
will be in the cases of interest below). 
We first observe that
hypercompleteness is automatic for $H\mathbb{Z}$-modules under a mild assumption
on $\mathcal{T}$. 

\begin{proposition} 
\label{Zmodulesarehypercomplete}
Let $\mathcal{T}$ be a 
finitary site of $\mathcal{P}$-local cohomological dimension $\leq d$ which is
an ordinary category. 
Then $\mathrm{Sh}(\mathcal{T}, \md_{H\mathbb{Z}_{\mathcal{P}}})$ is hypercomplete. 
\end{proposition} 
\begin{proof} 
We apply the criterion of \Cref{critforhypcolimits}. Since sheafification 
is $t$-exact, the functor 
$H\mathbb{Z}_{\mathcal{P}} \otimes \Sigma^\infty_+ h_{(-)}: \mathcal{T} \to
\mathrm{Sh}(\mathcal{T}, \md_{H\mathbb{Z}_{\mathcal{P}}})$ takes values in truncated,
connective 
objects and strongly generates the target. Our assumption implies that the
functor also takes values in connective objects of cohomological dimension
$\leq d$.  
The hypothesis of finitary implies that cohomology commutes with filtered
colimits on $\mathcal{T}$, cf.\ \cite[Exp. VI.5]{SGA4} for a detailed treatment. 
Then, \Cref{critforhypcolimits} shows that the subcategory of hypercomplete objects 
in $\mathrm{Sh}(\mathcal{T}, \md_{H\mathbb{Z}_{\mathcal{P}}})$ is
closed
under all colimits. 
Since the objects $H\mathbb{Z}_{\mathcal{P}} \otimes \Sigma^\infty_+ h_x$ are hypercomplete
(indeed, truncated) and generate everything, the claim follows. 
\end{proof}

\begin{proposition} 
Let $\mathcal{T}$ be a 
finitary site of $\mathcal{P}$-local cohomological dimension $\leq d$ which is an
 ordinary category. 
Let $\mathcal{F} \in \mathrm{Sh}(\mathcal{T}, \sp_{\mathcal{P},\geq 0})$. Then
the map $\mathcal{F} \to \varprojlim (\mathcal{F} \otimes \tau_{\leq
n} \mathbb{S})$ exhibits the target as the hypercompletion (equivalently,
Postnikov completion) of the source. 
\end{proposition} 
\begin{proof} 
Indeed, the map $\mathcal{F} \to \varprojlim (\mathcal{F} \otimes \tau_{\leq
n} \mathbb{S})$ is a $\pi_*$-isomorphism. Moreover, each 
of the terms in the limit on the right-hand-side 
is hypercomplete
(\Cref{Zmodulesarehypercomplete}). The result follows. \end{proof}

We now give a criterion for hypercompleteness via the notion of nilpotence, cf.,
e.g., \cite{Ma15}. 

\begin{definition}[Nilpotence and weak nilpotence]
Fix an integer $m \geq 0$. 
\begin{enumerate}
\item  
A filtered object $\dots \to X_{-1} \to X_0 \to X_1 \to \dots $  in a stable $\infty$-category
$\mathcal{C}$ 
(resp.~$\mathrm{Sp}$)
is called \emph{$m$-nilpotent} (resp.~\emph{weakly $m$-nilpotent})
if for each $i$, the map $X_i \to X_{i+m+1}$ is nullhomotopic (resp.~induces the
zero map on $\pi_*$). 
Either condition implies that $\varprojlim X_i \simeq \varinjlim X_i = 0$. 
\item
We say that an augmented cosimplicial diagram 
$X^\bullet \in \mathrm{Fun}(\Delta^+, \mathcal{C})$
(resp.~in $\mathrm{Fun}(\Delta^+, \mathrm{Sp})$)
is \emph{$m$-rapidly converging} (resp. \emph{weakly $m$-rapidly converging}) if 
the tower
$\left\{\mathrm{cofib}(X^{-1} \to \mathrm{Tot}_n(X^\bullet)\right\}_{n \geq 0}$ is 
$m$-nilpotent (resp. weakly $m$-nilpotent); this in particular implies that $X^\bullet$ is a limit diagram. 

\end{enumerate}

\end{definition}

The main use of the above notion is that they will enable us to commute
totalizations and filtered colimits, cf.\ also \cite[Sec. 3.1.3]{Mi97}. 
We will only lightly use the following lemma here but it will play a crucial
role in the sequel. 

\begin{lemma}[Commuting totalizations and filtered colimits]
\label{commutetotandcolimit}
Let $X_{\alpha}^\bullet, \alpha \in A$ be a system of augmented cosimplicial spectra
indexed over a filtered partially ordered set $A$. Suppose that there exists
an $m \in \mathbb{Z}_{\geq 0}$ such that each
$X_\alpha^\bullet$ is weakly $m$-rapidly converging for $\alpha \in A$.
Then the colimit $\varinjlim_{\alpha \in A} X_\alpha^\bullet$ is weakly
$m$-rapidly converging and in particular a limit diagram. 
In particular, the map 
\[ \varinjlim_{\alpha \in A} \mathrm{Tot}(X_\alpha^\bullet) 
\to \mathrm{Tot}( \varinjlim_{\alpha \in A} X_\alpha^\bullet)
\]
is an equivalence. 
\end{lemma} 
\begin{proof} 
For each $n$, the map 
\[ \varinjlim_{\alpha \in A} \mathrm{Tot}_n(X_\alpha^\bullet) 
\to \mathrm{Tot}_n( \varinjlim_{\alpha \in A} X_\alpha^\bullet)
\]
is an equivalence, since finite homotopy limits commute with filtered (indeed
all) colimits in spectra. 
Now the tower
\begin{equation} 
\label{colimittower}
\left\{Z_n\right\}_{n \geq 0}
\stackrel{\mathrm{def}}{=}\left\{\mathrm{cofib}(\varinjlim_{\alpha \in A} 
X_\alpha^{-1} \to \mathrm{Tot}_n(\varinjlim_{\alpha \in A} X_\alpha^\bullet))
\right\}_{n \geq 0}
\end{equation}
is a filtered colimit of 
the towers 
$\left\{\mathrm{cofib}(X_\alpha^{-1} \to 
\mathrm{Tot}_n (X_\alpha^{\bullet}))
\right\}$ which are all weakly $m$-nilpotent. Thus the tower
$\left\{Z_n\right\}_{n \geq 0}$ of
\eqref{colimittower}
has the property that the maps $Z_{n+m+1} \to Z_n$ are filtered colimits of
maps which are zero on homotopy, and hence are zero on homotopy. Therefore 
$\left\{Z_n\right\}$ is weakly $m$-nilpotent, 
$\varprojlim_n Z_n =
0$, and 
$\varinjlim_{\alpha \in A} X_\alpha^\bullet$ is a limit diagram as desired. 
\end{proof} 

\begin{proposition}[Nilpotence and hyperdescent] 
\label{nilpandhyp}
Let $\mathcal{T}$ be a 
finitary site of $\mathcal{P}$-local cohomological dimension $\leq d$. 
Let $\sF $ be a sheaf of $\mathcal{P}$-local spectra on $\mathcal{T}$. Then the
following are equivalent: 
\begin{enumerate}
\item $\sF$ is hypercomplete (or Postnikov complete). 
\item 
For every truncated hypercover $y_\bullet$ in $\mathcal{T}$ of an object $x \in
\mathcal{T}$, the augmented cosimplicial spectrum
\[ \sF(x) \to \sF(y)^\bullet  \]
is $d$-rapidly converging (equivalently, weakly $d'$-rapidly converging for some uniform constant $d'$
independent of the hypercover). 
It suffices to consider truncated hypercovers such that each $y_i$ is a finite coproduct
of objects in $\mathcal{T}$. 
\end{enumerate}
\end{proposition} 
\begin{proof} 
For ease of notation we will omit the $\mathcal{P}$'s in the following proof,
so everything is implicitly $\mathcal{P}$-localized. 
Suppose first that $\sF$ is hypercomplete. 
Consider the augmented simplicial object $(\Sigma^\infty_+ h_{y_\bullet})^h$. 
By applying the functor $\hom_{\mathrm{Sh}(\mathcal{T}, \sp)}(\cdot, \sF)$, 
it suffices to show that 
$(\Sigma^\infty_+ h_{y_\bullet})^h$ is $d$-rapidly converging in 
the opposite category of 
the hypercompletion of 
$\mathrm{Sh}(\mathcal{T}, \sp)$, in order to prove 2. 
For this, consider the map 
\[ |\sk_n  (\Sigma^\infty_+ h_{y_\bullet}) |^h \to 
|\sk_{n+d+1}  (\Sigma^\infty_+ h_{y_\bullet}) |^h \to (\Sigma^\infty_+ h_x)^h
.\]
The
object
$|\sk_n  (\Sigma^\infty_+ h_{y_\bullet})^h |$
has cohomological dimension $\leq d + n$ as an $n$-truncated geometric
realization of objects of cohomological dimension $\leq d$,\footnote{Note that
hypercompletion preserves cohomological dimension, which is tested in terms of
maps into truncated objects.} while the second map 
has homotopy fiber in $\mathrm{Sh}(\mathcal{T}, \sp)_{\geq n+d+1} $
(cf.\ \cite[Lemma 6.5.3.11]{HTT}). This shows that 
$$\mathrm{fib}(|\sk_n  (\Sigma^\infty_+ h_{y_\bullet}) | \to \Sigma^\infty_+
h_x)^h   \to 
\mathrm{fib}(|\sk_{n+d+1}  (\Sigma^\infty_+ h_{y_\bullet}) | \to
\Sigma^\infty_+ h_x)^h$$ 
is a map from an object of cohomological dimension $\leq d+n$ to an object 
concentrated in degrees $\geq d+ n+1$ and therefore vanishes since all objects
are Postnikov complete. 

Now suppose 2 holds (with possibly any $d'$ replacing $d$). Note that the class
of sheaves that satisfy 2 (for some $d'$, which may depend on the sheaf but
which is independent of the truncated hypercover) forms a thick subcategory. 
Replacing $\sF$ with the fiber of the map from $\sF$ to its hypercompletion, we
may assume that the homotopy group sheaves of $\sF$ vanish; then we want $\sF =
0$. 
Then for every hypercover $y_\bullet \to x$, we have $\sF(x) \simeq
\sF(y^\bullet)$: in fact, this is true for truncated hypercovers with a uniform
weak nilpotence, by assumption, so we can pass to the limit 
(approximating a hypercover by its skeleta and using
\Cref{commutetotandcolimit})
and obtain a
statement for all hypercovers. 
By \Cref{DHIprop} below, this implies that all sections of $\sF$ vanish, as
desired. 
\end{proof} 

In practice, the advantage of the above result is that in explicit examples, we
will be able to work not with all hypercovers but certain specific ones,
especially those arising from Galois covers. 

For the above result, we needed the following crucial construction of sufficiently many hypercovers. 
For ease of notation, we drop the $h$ from the Yoneda embedding, so identify
$\mathcal{T}$ as a full subcategory of $\psh(\mathcal{T})$. 
\begin{proposition}[Cf.\ Dugger-Hollander-Isaksen \cite{DHI04}] 
\label{DHIprop}
Let $\mathcal{T}$ be a finitary site. 
Suppose $\mathcal{F} \to \mathcal{G}$ is a map in $\psh(\mathcal{T})$ which induces an 
isomorphism on homotopy group sheaves. 
Suppose given an object $x \in \mathcal{T}$ and a map $x \to \mathcal{G}$. 
Then there exists a hypercover $y_\bullet$ of $x$, by finite coproducts of
representables, and a commutative diagram in $\psh(\mathcal{T})$,
\[ \xymatrix{
|{y_\bullet}| \ar[d]  \ar[r] & \mathcal{F} \ar[d]  \\
x \ar[r] &  \mathcal{G}
}.\]
\end{proposition} 
\begin{proof} 
We can assume that $x \to \mathcal{G}$ is an equivalence by forming the
pullback, so that we have a map 
$\mathcal{F} \to x$ which induces an 
isomorphism on homotopy group sheaves. 
We build the hypercover inductively. 

For $n =0$, we let $y_0$ be a finite coproduct of objects in $\mathcal{T} $ such
that 
$y_0 \to x$ is a cover and 
such that $y_0 \to x \simeq \mathcal{G}$ lifts to $\sF$. 
Suppose that we have constructed the $(n-1)$-skeleton $y^{\leq n-1}_\bullet$ as an object
of $\psh(\mathcal{T})_{/\sF}$, and then we need to build the $n$-skeleton as an
object of $\psh(\mathcal{T})_{/\sF}$. 

We let $L_n y^{\leq n-1}, M_n y^{\leq n-1}$ denote the $n$th latching and matching objects of the
$(n-1)$-truncated simplicial object $y^{\leq n-1}_\bullet$, considered as 
objects in $\psh(\mathcal{T})_{/\sF}$. 
To this end, by \cite[Prop. A.2.9.14]{HTT}, we need to find an object $y_n \in
\psh(\mathcal{T})_{/\mathcal{F}}$, which is a
finite disjoint union of representables, 
and a factorization $L_n y^{\leq n-1} \to y_n \to M_n y^{\leq n-1}$ over $\mathcal{F}, $ such that
$y_n \to M_n y^{\leq n-1}$ is a surjection on $\pi_0$ (the hypercover condition).  
To this end, we observe that $M_n y^{\leq n-1}$ has the same homotopy groups as an object
which is a finite coproduct of objects in $\mathcal{T}$, since $\mathcal{F} \to
x$ is an isomorphism on homotopy groups. 
In particular, there exists an object $z \in \psh(\mathcal{T})_{/\sF}$ which is a
finite coproduct of objects in $\mathcal{T} $ and a map $z \to M_n y^{\leq n-1}$
which is a surjection on homotopy groups. 
Thus, we can take $y_n$ to be the  coproduct of $L_n y^{\leq n-1} \sqcup z$. This defines
the $n$-truncated simplicial object $y^{\leq n}_\bullet$. Continuing as $n \to
\infty$, we build the desired hypercover. 
\end{proof} 

\section{The Zariski and Nisnevich topoi}

The purpose of this section is to prove some results about the \emph{homotopy
dimension} of certain $\infty$-topoi. 
The notion of homotopy dimension, introduced in \cite[Sec. 7.2.1]{HTT}, gives an
effective
criterion which guarantees hypercompleteness and the convergence of Postnikov
towers, so that one can extract descent spectral sequences.  

\begin{definition}[Homotopy dimension $\leq n$] 
Let $\mathfrak{X}$ be an $\infty$-topos.
\begin{enumerate}
\item  
 We have internal notions of
\emph{$n$-truncated} and \emph{$n$-connective} objects in $\mathfrak{X}$ (cf.\ 
\cite[Sec. 5.5.6]{HTT} and 
 \cite[Def. 6.5.1.10]{HTT}).  
For example, if $\mathfrak{X}$ is the $\infty$-category $\mathrm{Sh}(\mathcal{C})$ of
sheaves of spaces on a Grothendieck site $\mathcal{C}$, then $\sF \in
\mathrm{Sh}(\mathcal{C})$ is $n$-truncated if 
the underlying presheaf of $\sF$ takes values in $n$-truncated spaces. 
Suppose $n \geq 0$; we say that all objects are $(-1)$-connective. 
Then $\sF \in \mathrm{Sh}(\mathcal{C})$ is $n$-connective if $\sF$ is
locally non-empty and for any $U \in
\mathcal{C}$, basepoint $\ast \in \sF(U)$, and class $y \in \pi_i(\sF(U),
\ast)$ for $i < n$,  
there exists a covering sieve $\left\{U_\alpha \to U\right\}$ of $U$ such that $y$
pulls back to the unit (or basepoint for $i = 0$) in $\pi_i(\sF(U_\alpha),
\ast)$ for each
$\alpha$. 
\item  We say that 
$\mathfrak{X}$ has \emph{homotopy dimension $\leq n$} if every $n$-connective
object $\sF \in \mathfrak{X}$ admits a section, i.e., a map $\ast \to \sF$ from
the terminal object. We say that $\mathfrak{X}$ is \emph{locally of homotopy
dimension $\leq n$} if there exists a set of objects $U_\alpha$ which
generate $\mathfrak{X}$ under colimits such that $\mathfrak{X}_{/U_\alpha}$ is
of homotopy dimension $\leq n$.
\end{enumerate}
\end{definition} 

For us, the main relevance of the notion of homotopy dimension arises from the
following theorem. 
\begin{theorem}[{Cf.\ \cite[Cor. 7.2.1.12]{HTT} and \cite[Cor. 1.3.3.11]{SAG}}]
\label{htpyimplieshypcomplete}
Let $\mathfrak{X}$ be an $\infty$-topos which is locally of homotopy dimension
$\leq n$. Then the $\infty$-category $\sh(\mathfrak{X}, \sp)$ of sheaves of spectra on
$\mathfrak{X}$  is Postnikov complete. 
\end{theorem}

In this section, we will show that the Zariski and Nisnevich topoi of qcqs
algebraic spaces of finite Krull dimension have finite homotopy dimension. 
The Zariski topos only depends on the underlying topological space, which is a
spectral space (cf.~\cite[Sec.~3.6]{SAG}), so the statement about Zariski topoi
is equivalently one about spectral spaces of finite Krull dimension. 
Since
these results (in various forms) are well-known in the noetherian case, 
the reader interested primarily in noetherian rings may skip this section
without loss of generality. 

In fact, we will give two arguments for the results. The first argument will be based on a
general criterion for passing finite homotopy dimension through a limiting
process, which could be useful in other settings. 
The second argument (which in the case of spectral spaces is based on \cite[Tag
0A3G]{stacks-project}) is based on a passage to pro-objects, and gives a
slightly stronger connectivity assertion. 

\subsection{Finitary excisive sites and finite homotopy dimension}

We continue the discussion of finitary sites from \Cref{finitarysitedef}. 
 A \emph{morphism} of finitary sites is a functor $F$ which preserves all finite
 limits such that the image under $F$ of a covering sieve generates a covering
 sieve.  This defines the $\infty$-category of finitary sites.  Given a morphism
 of finitary sites $F:\mathcal{C}\rightarrow\mathcal{D}$ we get a geometric
 morphism $\operatorname{Sh}(\mathcal{D})\rightarrow
 \operatorname{Sh}(\mathcal{C})$, see \cite[Prop. 6.2.3.20]{HTT}: the corresponding ``pullback" functor $\operatorname{Sh}(\mathcal{C})\rightarrow\operatorname{Sh}(\mathcal{D})$ can be characterized as the unique colimit-preserving functor which sends the representable sheaf $h_U$ to $h_{F(U)}$ for $U\in\mathcal{C}$, and the right adjoint ``pushforward" functor $\operatorname{Sh}(\mathcal{D})\rightarrow \operatorname{Sh}(\mathcal{C})$ is given by composition with $F$ on the level of presheaves (which sends sheaves to sheaves).

\begin{lemma}
\label{filtdiagfinitarysite}
\begin{enumerate}
\item The $\infty$-category of finitary sites admits all filtered colimits, and these commute with the forgetful functor to the $\infty$-category of small $\infty$-categories.  Explicitly, if $\{\mathcal{C}_i\}_{i\in I}$ is a filtered diagram of finitary sites, then one equips the filtered colimit category $\mathcal{C}:=\varinjlim_{i\in I}\mathcal{C}_i$ with the Grothendieck topology described by: a sieve is covering if and only if it admits a refinement which is generated by the image of a covering sieve in one of the $\mathcal{C}_i$.
\item If $\{\mathcal{C}_i\}_{i\in I}$ is a filtered diagram of finitary sites with colimit $\mathcal{C}$, then the induced pushforward functor
$$\operatorname{Sh}(\mathcal{C})\rightarrow\varprojlim_{i\in I}\operatorname{Sh}(\mathcal{C}_i)$$
is an equivalence.
\end{enumerate}
\end{lemma}
\begin{proof}
First we prove claim 1.  Since the colimit is filtered, every finite diagram
$d:\mathcal{D}\rightarrow \mathcal{C}$ lifts to some
$d':\mathcal{D}\rightarrow\mathcal{C}_i$.  Furthermore, the limit of $d'$ will
map to a limit of $d$, because again this is a claim about a finite diagram and
we know it at every stage of the filtered system.  From this we see that
$\mathcal{C}$ has all finite limits, each functor
$\mathcal{C}_i\rightarrow\mathcal{C}$ preserves finite limits, and $\mathcal{C}$
is the filtered colimit of the $\mathcal{C}_i$ in the $\infty$-category of
finitely complete small $\infty$-categories with finite-limit preserving
functors.  The remaining claims concern the Grothendieck topology; these depend
only on the homotopy category (\cite[Rem. 6.2.2.3]{HTT}) and therefore reduce to
the classical case, treated in \cite[Exp. VI, Sec. 7-8]{SGA4}.

For claim 2, since presheaves are functors out, we have
$\operatorname{PSh}(\mathcal{C})\simeq \varprojlim_{i\in
I}\operatorname{PSh}(\mathcal{C}_i)$.  Thus we only need to check that an
$\mathcal{F}\in\operatorname{PSh}(\mathcal{C})$ is a sheaf if its restriction to
each $\mathcal{C}_i$ is a sheaf.  (The converse follows from functoriality of
the $\infty$-topos associated to a finitary site, recalled above.)  But by
refinement, $\mathcal{F}$ is a sheaf if and only if it has descent with respect
to covering sieves generated by finitely many objects $X_1,\ldots X_n$ covering
some $X$.  By a standard cofinality argument, this descent property is
equivalent to requiring that $\mathcal{F}(X)$ be the limit of a diagram built out of iterated fiber products of the $X_\alpha$ over $X$.  As each of the $X_\alpha \rightarrow X$ are realized at some common stage $\mathcal{C}_i$ and the functor $\mathcal{C}_i\rightarrow\mathcal{C}$ preserves finite limits, the descent condition only depends on the restriction to $\mathcal{C}_i$, verifying the claim.
\end{proof}

Further nice properties of these $\infty$-categories of sheaves are available
under a hypothesis.  To motivate what follows, recall that when $\mathcal{C}$ is
a finitary site, the representable sheaves
$h_U\in\operatorname{Sh}(\mathcal{C})$ for $U\in\mathcal{C}$, while
\emph{coherent} in the sense of $\infty$-topos theory, are not necessarily
\emph{compact } in the categorical sense of mapping spaces out of them commuting
with filtered colimits.  (However, they have compact image in the $d+1$-topos
$\operatorname{Sh}(\mathcal{C})_{\leq d}$ of $d$-truncated sheaves for any
$d\geq 0$: see \cite[A.2.3]{SAG}.)  An example to have in mind is the $\infty$-topos of spaces $\mathcal{S}$: the coherent objects are the spaces all of whose homotopy sets $\pi_n$ are finite (at any basepoint, when $n>0$), and the compact objects are the retracts of the spaces homotopy equivalent to finite CW-complexes.

\begin{definition}
Let $\mathcal{C}$ be a finitary site.  We say $\mathcal{C}$ is \emph{excisive} if the full subcategory $\operatorname{Sh}(\mathcal{C})\subset\operatorname{PSh}(\mathcal{C})$ is closed under filtered colimits, or equivalently if the generating coherent objects $h_U\in\operatorname{Sh}(\mathcal{C})$ for $U\in\mathcal{C}$ are also compact in the $\infty$-categorical sense that $\mathcal{F}\mapsto \operatorname{Map}(h_U,\mathcal{F})=\mathcal{F}(U)$ commutes with filtered colimits.
\end{definition}

\begin{proposition}
Let $\mathcal{C}$ be a finitary site.  Denote by $\operatorname{Sh}^f(\mathcal{C})$ the smallest full subcategory of $\operatorname{Sh}(\mathcal{C})$ closed under finite colimits and containing the representable sheaf $h_U$ for every $U\in\mathcal{C}$.

If $\mathcal{C}$ is excisive, then
$$\operatorname{Sh}(\mathcal{C})=\operatorname{Ind}(\operatorname{Sh}^f(\mathcal{C})).$$
\label{exccompactgen}
\end{proposition}

\begin{proof}
If $\mathcal{C}$ is excisive then each $h_U$ is compact, whence so is every
object in $\operatorname{Sh}^f(\mathcal{C})$.  Thus the functor
$\operatorname{Ind}(\operatorname{Sh}^f(\mathcal{C}))\rightarrow\operatorname{Sh}(\mathcal{C})$
is fully faithful.  It is then essentially surjective because the $h_U$
clearly generate $\operatorname{Sh}(\mathcal{C})$ under colimits. Compare
\cite[Proposition 5.3.5.11]{HTT}. 
\end{proof}

\begin{corollary}
\label{filteredcolimitexc}
Let $\mathcal{C}_i$ be a filtered system of finitary excisive sites.  Then:
\begin{enumerate}
\item  The colimit $\mathcal{C}:=\varinjlim_{i\in I}\mathcal{C}_i$ is excisive.
\item $\operatorname{Sh}^f(\mathcal{C})=\varinjlim_{i\in I}\operatorname{Sh}^f(\mathcal{C}_i)$ via the pullback functors.
\end{enumerate}
\end{corollary}
\begin{proof}
Claim 1 follows from \Cref{filtdiagfinitarysite} part 2, which implies that a
sheaf on $\mathcal{C}$ is a compatible system of sheaves on
$\{\mathcal{C}_i\}$; the same holds for presheaves, and thus implies that
$\sh(\mathcal{C}) \subset \psh(\mathcal{C})$ commutes with filtered colimits,
since this is true at each finite level. 

For claim 2, the description 
\cite[Corollary 5.3.5.4]{HTT} 
of $\operatorname{Ind}$ as finite limit-preserving presheaves gives an identification
$$\operatorname{Ind}(\varinjlim_{i\in
I}\operatorname{Sh}^f(\mathcal{C}_i))=\varprojlim_{i
\in I}\operatorname{Ind}(\operatorname{Sh}^f(\mathcal{C}_i)).$$
Combining \Cref{exccompactgen} and \Cref{filtdiagfinitarysite} part 2 shows that the $\infty$-category on the right identifies with $\operatorname{Sh}(\mathcal{C})$.  In particular we can see that $\varinjlim_{i\in I}\operatorname{Sh}^f(\mathcal{C}_i)\rightarrow\operatorname{Sh}^f(\mathcal{C})$ is fully faithful.  On the other hand it is essentially surjective because pullback functors preserve finite colimits (indeed, all colimits) and the $h_U$ for $U\in\mathcal{C}$ are clearly hit.\end{proof}

\begin{remark}
Let $d\in\mathbb{Z}_{\geq 0}$.  By \cite[Sec. A.2.3]{SAG}, the analog of
\Cref{exccompactgen} and claim 2 of \Cref{filteredcolimitexc} hold for the $d+1$-category of $d$-truncated sheaves $\operatorname{Sh}(-)_{\leq d}$ without any excisive hypotheses.
\label{colimitworksfornexc}
\end{remark}

Next we will give an analog of the previous proposition and corollary also in the setting of $d$-connective sheaves.  First, some notation: if $\mathcal{C}$ is a site, we write $\operatorname{Sh}(\mathcal{C})_{\geq d}\subset \operatorname{Sh}(\mathcal{C})$ for the full subcategory of $d$-connective sheaves, and if $\mathcal{C}$ is finitary we further set
$$\operatorname{Sh}^f(\mathcal{C})_{\geq d} := \operatorname{Sh}(\mathcal{C})_{\geq d}\cap \operatorname{Sh}^f(\mathcal{C}),$$
the full subcategory of those sheaves which both are generated by $h_U$'s under finite colimits and are $d$-connnective.

A key technical lemma for us will be that any $\mathcal{F}\in\operatorname{Sh}^f(\mathcal{C})_{\geq d}$ has \emph{locally finitely generated} $\pi_d$ in the appropriate sense.  For motivation one can think of the case of the $\infty$-topos $\mathcal{S}$, where it is a standard fact from homotopy theory that a $d$-connective finite CW-complex $X$ has finitely generated $\pi_d$.  In fact, $X$ need not itself be finite: it's enough for it to have finite $d$-skeleton.  In general, we articulate this as follows:

\begin{lemma}
\label{finitegrouplemma}
Let $\mathcal{C}$ be a finitary site, let $d\geq 0$, and let $\mathcal{F}\in\operatorname{Sh}(\mathcal{C})$ be such that the Postnikov truncation $\mathcal{F}_{\leq d}\in \operatorname{Sh}(\mathcal{C})_{\leq d}$ lies in the smallest full subcategory $(\operatorname{Sh}(\mathcal{C})_{\leq d})^f$ of $\operatorname{Sh}(\mathcal{C})_{\leq d}$ containing the $(h_U)_{\leq d}$ for $U\in\mathcal{C}$ and closed under finite colimits.

Then there are finitely many $\{U_i\}_{i\in I}$ in $\mathcal{C}$ and points $x_i \in \mathcal{F}(U_i)$ such that every section of $\pi_0\mathcal{F}$ is locally equal to the class of some $x_i$.  Furthermore, if $d\geq 1$ and $\mathcal{F}$ is $d$-connective then this data can be chosen so that for each $i\in I$ the homotopy group sheaf $\pi_d(\mathcal{F};x_i)$ over $U_i$ is generated by finitely many global sections.

\end{lemma}
\begin{proof}
For any such $\mathcal{F}$ we have that $\pi_0(\mathcal{F})=\mathcal{F}_{\leq 0}$ lies in $(\operatorname{Sh}(\mathcal{C})_{\leq 0})^f$, so that $\pi_0(\mathcal{F})$ is a compact object in the category of sheaves of sets on $\mathcal{C}$.  On the other hand every sheaf of sets is tautologically the filtered colimit of its locally finitely generated subsheaves.  It follows that $\pi_0(\mathcal{F})$ is a retract of a locally finitely generated sheaf of sets, hence is locally finitely generated.  This gives $\{U_i\}$ and $\{x_i\}$ as required.

Now take $d\geq 1$, and assume
$\mathcal{F}\in\operatorname{Sh}(\mathcal{C})_{\geq d}$.  Since we allowed
ourselves to modify the $\{U_i\}$ and $\{x_i\}$, we are free to work locally,
hence without loss of generality we can equip $\mathcal{F}$ with a basepoint.
Since $\mathcal{F}_{\leq d}\in\operatorname{Sh}(\mathcal{C})_{\leq d}$ is a
compact object, and a point is also compact because the site is finitary,
$\mathcal{F}_{\leq d}$ is also compact in the $\infty$-category of pointed
objects of $\operatorname{Sh}(\mathcal{C})_{\leq d}$.  However, \cite[Prop.
7.2.2.12]{HTT} shows that $\pi_d$ establishes an equivalence of categories between pointed, $d$-truncated, $d$-connective sheaves and group sheaves (if $d=1$) or abelian group sheaves (if $d\geq 2$).  Since $d$-connectivity is preserved by filtered colimits (indeed, all colimits), we deduce that $\pi_d\mathcal{F}$ is compact as a sheaf of groups ($d=1$) or abelian groups ($d\geq 2$).  But every such sheaf is a filtered colimit of its locally finitely generated subsheaves, so from compactness we deduce that $\pi_d\mathcal{F}$ itself is locally finitely generated, whence the claim.
\end{proof}

\begin{proposition}
\label{approxbydconncompact}
Let $\mathcal{C}$ be a finitary site.  Then for every $d\geq 0$ we have:
\begin{enumerate}
\item If $\mathcal{F}\rightarrow\mathcal{G}$ is a map in $\operatorname{Sh}(\mathcal{C})$ with $\mathcal{F}\in \operatorname{Sh}^f(\mathcal{C})$ and $\mathcal{G}\in\operatorname{Sh}(\mathcal{C})_{\geq d}$, then there is a factorization $\mathcal{F}\rightarrow\mathcal{F}'\rightarrow\mathcal{G}$ with $\mathcal{F}'\in\operatorname{Sh}^f(\mathcal{C})_{\geq d}$.
\item If $\mathcal{C}$ is excisive, then $\operatorname{Sh}(\mathcal{C})_{\geq d} = \operatorname{Ind}(\operatorname{Sh}^f(\mathcal{C})_{\geq d}).$
\end{enumerate}
\end{proposition}
\begin{proof}
Claim 2 follows formally from Claim 1 and \Cref{exccompactgen}, by a cofinality argument.  Thus we need only prove Claim 1.

When $d=0$, we can find finitely many $\{U_i\}$ covering $\ast$ and points $y_i\in\mathcal{G}(U_i)$.  Then set $\mathcal{F}'=\mathcal{F}\sqcup \sqcup_{i\in I} h_{U_i}$ with the map to $\mathcal{G}$ being the given map on $\mathcal{F}$ and the map classifying the $y_i$ on the $i^{th}$ summand of the coproduct.  This $\mathcal{F}'$ is clearly $0$-connective and finite, giving the claim for $d=0$.  Then we proceed by induction on $d$.  Thus we can assume that $\mathcal{F}$ itself lies in $\operatorname{Sh}^f(\mathcal{C})_{\geq d}$, and try to find $\mathcal{F}'$ in $\operatorname{Sh}^f(\mathcal{C})_{\geq d+1}$ with a factoring $\mathcal{F}\rightarrow\mathcal{F}'\rightarrow\mathcal{G}$ assuming $\mathcal{G}\in\operatorname{Sh}^f(\mathcal{C})_{\geq d+1}$.

Consider the $x_i$ associated to $\mathcal{F}$ as in \Cref{finitegrouplemma}, and classify them by a map of sheaves
$$\sqcup_{i\in I}h_{U_i}\rightarrow\mathcal{F}.$$
If $d=0$, then since $\mathcal{G}$ is $1$-connective, for every $(i,j)\in
I\times I$ there is a finite cover $\{U_{ijk}\}_k$ of $U_i\times U_j$ such that
the images of $x_i$ and $x_j$ in $\mathcal{G}$ are homotopic over $U_{ijk}$.  In
other words, defining $\mathcal{F}'$ to be the coequalizer of the two different
natural maps $\sqcup_{i,j,k}h_{U_{ijk}} \rightarrow\mathcal{F}$, our map
$\mathcal{F}\rightarrow\mathcal{G}$ factors through $\mathcal{F}'$.  By
construction $\mathcal{F}'\in\operatorname{Sh}^f(\mathcal{C})$, so we need to
see that $\mathcal{F}'$ is $1$-connective, meaning $\pi_0\mathcal{F}'=\ast$, the
terminal sheaf of sets.  But $\mathcal{F}\rightarrow\mathcal{F}'$ is epimorphic
on $\pi_0$ because a coequalizer is a quotient, and furthermore every section of
$\pi_0\mathcal{F}$ is locally equal to some $x_i$ by \Cref{finitegrouplemma}, so it suffices to show that the $x_i$ become equal in $\pi_0\mathcal{F}'$, or equivalently that they become locally homotopic in $\mathcal{F}'$.  But this has exactly been arranged by the definition of $\mathcal{F}'$.

If $d\geq 1$, then for each $i\in I$ choose generators for
$\pi_d(\mathcal{F};x_i)$ as in \Cref{finitegrouplemma}, and lift them to maps
$$\sqcup_{j\in J_i} S^d\times h_{U_i}\rightarrow \mathcal{F}.$$
Collecting together all the $i\in I$, this gives a map
$$\sqcup_{i,j} S^d\times h_{U_i}\rightarrow\mathcal{F}.$$
Since $\mathcal{G}$ is $d+1$-connective, by refining the $U_i$ if necessary we can assume that each composition $S^d\times h_{U_i}\rightarrow\mathcal{F}\rightarrow \mathcal{G}$ factors through $\ast\times h_{U_i}$.  In other words, if we define $\mathcal{F}'$ to be the pushout of the above map along the projection $\sqcup_{i,j} S^d\times h_{U_i}\rightarrow \sqcup_{i,j}\ast\times h_{U_i}$, then $\mathcal{F}\rightarrow\mathcal{G}$ factors through $\mathcal{F}'$.  Clearly $\mathcal{F}'\in\operatorname{Sh}^f(\mathcal{C})$.  Further, since the projection $S^d\rightarrow \ast$ is $d$-connective and $d$-connective maps are closed under pushouts, we deduce that $\mathcal{F}\rightarrow \mathcal{F}'$ is $d$-connective.  Since $\mathcal{F}$ is $d$-connective, it follows that $\mathcal{F}'$ is as well, and that the map of group sheaves $\pi_d(\mathcal{F};x_i)\rightarrow\pi_d(\mathcal{F}';x_i)$ over $U_i$ is epimorphic for all $i$.  But by construction the generators of $\pi_d(\mathcal{F};x_i)$ become nullhomotopic in $\pi_d(\mathcal{F}';x_i)$, so the latter group sheaf vanishes, hence $\mathcal{F}'$ is $d+1$-connective as desired.
\end{proof}

\begin{proposition}
\label{filtcolimitdconnectivecompact}
Let $\mathcal{C}_i$ be a filtered system of finitary sites.  Let $\mathcal{C}$ denote the filtered colimit site, and let $d\geq 0$. Then:
\begin{enumerate}
\item If $\mathcal{F}\in\operatorname{Sh}^f(\mathcal{C}_i)$ has $d$-connective pullback to $\mathcal{C}$, then it has $d$-connective pullback to $\mathcal{C}_j$ for some $j$. 
\item If each $\mathcal{C}_i$ is excisive, then $\operatorname{Sh}^f(\mathcal{C})_{\geq d}=\varinjlim_{i\in I}\operatorname{Sh}^f(\mathcal{C}_i)_{\geq d}$ via the pullback functors.
\end{enumerate}
\end{proposition}
\begin{proof}
For claim 1, since Postnikov truncations commute with pullback maps of
$\infty$-topoi and a sheaf is $d$-connective if and only if its truncation to
degrees $<d$ is $\ast$, this follows directly from \Cref{colimitworksfornexc}.

For claim 2, the functor $\varinjlim_{i\in
I}\operatorname{Sh}^f(\mathcal{C}_i)_{\geq
d}\rightarrow\operatorname{Sh}^f(\mathcal{C})_{\geq d}$ is fully faithful by
\Cref{filteredcolimitexc}, and essentially surjective by
\Cref{colimitworksfornexc} and claim 1.\end{proof}

The important principle for us is the following corollary:

\begin{corollary}
\label{filtcolimitdimbound}
Suppose $\mathcal{C}_i$ is a filtered system of finitary excisive sites, with colimit $\mathcal{C}$.  If each $\operatorname{Sh}(\mathcal{C}_i)$ has homotopy dimension $\leq d$, then so does $\operatorname{Sh}(\mathcal{C})$.
\end{corollary}
\begin{proof}
Let $\mathcal{F}\in\operatorname{Sh}(\mathcal{C})_{\geq d}$.  From
\Cref{approxbydconncompact} we see that $\mathcal{F}$ admits a map from an
object $\mathcal{G}\in\operatorname{Sh}^f(\mathcal{C})_{\geq d}$.  By claim 2 of
\Cref{filtcolimitdconnectivecompact}, $\mathcal{G}$ is pulled back from some object $\mathcal{G}'\in\operatorname{Sh}^f(\mathcal{C}_i)_{\geq d}$.  This $\mathcal{G}'$ has a global section by hypothesis, hence its pullback $\mathcal{G}$ has a global section, hence $\mathcal{F}$ has a global section, as desired.
\end{proof}

\subsection{Spectral spaces}
Let $X$ be a spectral space\footnote{Recall this means that the quasi-compact
open subsets of $X$ form a basis closed under finite intersection, and $X$ is
$T_0$ and sober; equivalently, by \cite{Ho69}, $X=\operatorname{Spec}(R)$ for
some commutative ring $R$, or $X$ is a filtered inverse limit of finite
$T_0$-spaces. 
A continuous map between spectral spaces $ X \to Y$ is called \emph{spectral} if the preimage of a
quasi-compact open subset is a quasi-compact open subset.}, and let $\operatorname{pt}_X$ denote the category of points of
the topos of sheaves on $X$.  Concretely, the objects of $\operatorname{pt}_X$
are the points of $X$, there is a map $x\rightarrow y$ iff every open subset
containing $y$ also contains $x$, and in this case the map is
unique.\footnote{For instance, this follows using \cite[Theorem 5.1.5.6 and
Proposition 6.1.5.2]{HTT}.}  Recall that the \emph{Krull dimension} of $X$ is the supremum of the lengths $n$ of the chains $x_0\rightarrow x_1\rightarrow\ldots \rightarrow x_n$ of non-identity morphisms in $\operatorname{pt}_X$.

We will denote by $\operatorname{Sh}(X)$ the $\infty$-category of sheaves of
spaces ($\infty$-groupoids) on $X$.  As always with an $\infty$-topos, for
$d\in\mathbb{Z}_{\geq 0}$ there is an internal notion of a sheaf
$\mathcal{F}\in\operatorname{Sh}(X)$ being \emph{$d$-connective}: it means the
(sheafified) Postnikov truncation $\mathcal{F}_{<d}$ is $\ast$ \cite[Sec.
6.5.1.12]{HTT}.  Thus, for example, $0$-connective means locally non-empty, and
$1$-connective means every two sections can locally be connected by a path.
Since the usual topos of sheaves of sets on $X$ has enough points given by the
points of $X$, a sheaf of spaces $\mathcal{F}$ is $d$-connective if and only if
the stalk $\mathcal{F}_x$ is a $d$-connective (equivalently $(d-1)$-connected) space for all $x\in X$.

The main result is:

\begin{theorem}
\label{spectralspacehtpy}
Let $X$ be a spectral space of Krull dimension $d\in\mathbb{Z}_{\geq 0}$, and let $\mathcal{F}\in\operatorname{Sh}(X)$ be $d$-connective.  Then $\mathcal{F}(X)$ is nonempty.
That is, $\mathrm{Sh}(X)$ has homotopy dimension $\leq d$. 
\end{theorem}

Note that by applying this statement to iterated loop (or path) spaces on $\mathcal{F}$, we deduce the more general statement that for any $n\geq 0$, if $\mathcal{F}$ is $(d+n)$-connective then $\mathcal{F}(X)$ is $n$-connective.

In other words, $X$ has \emph{homotopy dimension} $\leq d$.  Since the
hypothesis clearly passes to any quasi-compact open subspace of $X$, the space $X$ is also
locally of homotopy dimension $\leq d$.  Thus $\operatorname{Sh}(X)$ is
Postnikov complete, hence hypercomplete (\Cref{htpyimplieshypcomplete}).  The
useful consequence is that a map of sheaves of spaces (or spectra) on $X$ is an
equivalence if and only if it is an equivalence on every stalk.  Note that this
consequence is \emph{not} ensured by the weaker property of locally finite
cohomological dimension, even in the setting of sheaves of spectra; see
\Cref{nonhypcompleteZp}.

The claim that $X$ as in \Cref{spectralspacehtpy} has cohomological dimension $\leq d$ was
proved in \cite{Sch92}, and in the noetherian case \Cref{spectralspacehtpy} was
proved in \cite[Cor. 7.2.4.17]{HTT}.  
The convergence of Postnikov towers and the descent spectral sequence in the
noetherian case appears in
\cite{BG73}. 
Together these results suggested that
\Cref{spectralspacehtpy} should be true.

The idea of the proof is to reduce to the case of \emph{finite} spectral spaces
(which are the same as finite $T_0$-spaces), via the result that every spectral
space $X$ is a filtered inverse limit of finite spectral spaces
(\cite[Prop. 10]{Ho69}) --- and
if $X$ has Krull dimension $\leq d$, these finite approximations can also be
taken of Krull dimension $\leq d$ \cite{Sal91}.  On the other hand, the case of finite spectral spaces is fairly elementary:

\begin{lemma}
\label{htpydimforfinitespectral}
Let $X$ be a finite spectral space and $\mathcal{F}\in\operatorname{Sh}(X)$.  If $\mathcal{F}$ is $\operatorname{dim}(X)$-connective, then $\mathcal{F}(X)$ is non-empty.
\end{lemma}
\begin{proof}
We proceed by induction on the number of points of $X$.   If $X$ is empty, the result is trivial, as $\mathcal{F}(X)=\ast$.  If $X$ is nonempty, it has a closed point $x$.  Let $X_x$ denote the intersection of all open neighborhoods of $x$; it is the set of all points specializing to $x$, and is itself open.  Then $X$ admits an open cover by $X_x$ and $X-x$ with intersection $X_x-x$, whence $\mathcal{F}(X) = \mathcal{F}(X_x)\times_{\mathcal{F}(X_x-x)}\mathcal{F}(X-x)$.  Thus it suffices to show that $\mathcal{F}(X_x)$ and $\mathcal{F}(X-x)$ are non-empty, and that $\mathcal{F}(X_x-x)$ is connected.  However, $\mathcal{F}(X_x)=\mathcal{F}_x$ is nonempty by assumption, $\mathcal{F}(X-x)$ is non-empty by the inductive hypothesis, and $\mathcal{F}(X_x-x)$ is even \emph{connected } by the inductive hypothesis, since $\operatorname{dim}(X_x-x)\leq \operatorname{dim}(X)-1$.
\end{proof}

\begin{proof}[Proof of \Cref{spectralspacehtpy}]
First, a general remark.  A spectral space $Y$ can equivalently be encoded by its category (poset) $\mathcal{C}_Y$ of quasicompact open subsets under inclusion.  Since the quasicompact open subsets of $Y$ form a basis of the topology of $Y$ closed under finite intersections,
the $\infty$-topos of sheaves of spaces on $Y$ 
 can equivalently be described as sheaves of spaces on the site $\mathcal{C}_Y$
 of quasicompact open subsets equipped with the induced Grothendieck topology,
 which is the obvious one of open coverings, cf.\ 
\cite[Prop. 1.1.4.4]{SAG}. 
Since $Y$ is spectral, $\mathcal{C}_Y$ is a finitary site.  It is also excisive
(cf.\ also \cite[Prop. 6.5.4.4]{HTT}):
indeed, induction on the number of quasicompact opens generating a covering sieve shows
that a presheaf $\mathcal{F}$ on $\mathcal{C}_Y$ is a sheaf if and only if
$\mathcal{F}(\emptyset)=\ast$ and $\mathcal{F}(U\cup
V)\overset{\sim}{\rightarrow}\mathcal{F}(U)\times_{\mathcal{F}(U\cap
V)}\mathcal{F}(V)$ for all $U,V\in\mathcal{C}_Y$, and this condition is
preserved under filtered colimit of presheaves. 

Furthermore, it is elementary to see that if $Y$ is a filtered inverse limit $\varprojlim_{i\in I}Y_i$ in the category of spectral spaces and spectral maps, then $\varinjlim_{i\in I}\mathcal{C}_{Y_i} = \mathcal{C}_Y$ via the pullback maps.

Returning to our $X$ of Krull dimension $\leq d$, the results of \cite{Sal91}
show that $X $ can be expressed as  
a filtered inverse limit 
$\varprojlim_{i \in I} X_i$
of finite spectral spaces $X_i,  i \in I$ of Krull dimension $\leq d$.
More precisely, using the Stone duality between 
spectral spaces and distributive lattices and 
\cite[Theorem 3.2]{Sal91} (which as explained there is essentially due to
\cite{Isbell}), we obtain an expression for $X$ as
$X \simeq \varprojlim_I X_i$ as desired. 
Each $X_i$ has homotopy dimension $\leq d$ by
\Cref{htpydimforfinitespectral}.  Thus 
\Cref{filtcolimitdimbound}
and the above remarks let us conclude that $X$ has homotopy dimension $\leq d$, as desired.
\end{proof}

We also 
include an alternative argument for 
\Cref{spectralspacehtpy}, which is inspired by the argument in 
\cite[Tag 0A3G]{stacks-project} for the cohomological dimension; it avoids the
use of \Cref{filtcolimitdimbound} and the reduction to finite spectral spaces. 
In fact, it yields a slightly stronger 
statement, in that we can relax the connectivity assumptions on the stalks
depending on the point. 

\begin{theorem} 
\label{htpydimspectralbyprosite}
Let $X$ be a spectral space of finite Krull dimension, and let $\sF \in \sh(X)$. 
Suppose that for each $x \in X$, we have that the stalk $\sF_x $ is $\dim(
\overline{\{x\}})$-connective. Then $\sF(X) \neq \emptyset$. 
\end{theorem} 
\begin{proof} 
We will prove the result by induction on the Krull dimension of $X$. When $X =
\emptyset$, the result is evident, so the induction starts. 
Note that (by taking loop or path spaces) the statement of the theorem (for a given $X$) is equivalent to the statement
that for any $\sF$ such that each stalk $\sF_x$ is $c +
\dim(\overline{\left\{x\right\}})$-connective, then $\sF(X)$ is $c$-connective.

For convenience and ease of notation (although this is not necessary), we will
assume that $X$ is the spectrum of a ring $A$ of finite Krull dimension. 
Then $\sF$ defines a sheaf of spaces on the Zariski site of $A$; in particular,
we can evaluate $\sF$ on finite products of $A$-algebras of the form $A[1/f], f
\in A$. 
By extending $\sF$ by forcing $\sF$ to preserve filtered colimits, 
we obtain a Zariski sheaf $\sF$ on all $A$-algebras which are filtered colimits
of algebras of the above form; here we use implicitly that the Zariski site is
finitary excisive. 
(Moreover, the Zariski site of a filtered colimit of rings
$\left\{A_i\right\}$ is the filtered
colimit of the Zariski sites of the individual $\left\{A_i\right\}$.)
Our assumption is that if $\mathfrak{p} \in \spec(A)$, then
$\sF(A_{\mathfrak{p}})$ is $\dim( A/\mathfrak{p})$-connective. 

Now we consider the collection $\mathcal{P}$ of all localizations $A[S^{-1}]$ of $A$ such that
$\sF(A[S^{-1}]) = \emptyset$. We want to show that $\mathcal{P}$ is empty, so
suppose $\mathcal{P}$ is nonempty.  
By construction, 
 $\mathcal{P}$ (considered as a subcategory of $A$-algebras) is a partially
 ordered set, and it admits filtered colimits. By Zorn's lemma, $\mathcal{P}$
 admits a maximal element $A'$: in particular, $\sF(A') = \emptyset$ but $\sF$
 takes nonempty values on any proper localization of $A'$. 
 Now $A'$ is not local, or $\sF(A')$ would be nonempty by our assumption, and
 $A' \neq 0$.
 Therefore, there exists an element $f \in A'$ which is neither a unit nor 
 in the Jacobson radical of $A'$. 
 We can form the localization $A''$ of $A'$ at all elements $1+gf, g \in A'$ and
 then we have a pullback square
 of spaces
 \[ \xymatrix{
 \sF(A') \ar[d]  \ar[r] &  \sF(A'[1/f]) \ar[d]  \\
 \sF(A'') \ar[r] &  \sF(A''[1/f])
 }.\]
 Our assumption is that $A'[1/f]$ and $A''$ are both nontrivial localizations of
 $A'$, so the terms 
 $\sF(A'[1/f]),  \sF(A'')$ are nonempty. 
 Finally, the Krull dimension of $A''[1/f]$ is less than that of $A$: in fact,
 no maximal ideals belong to the image of 
 the injective map $\spec(A''[1/f]) \to \spec(A)$. 
In fact, this shows that for any $x \in \spec(A''[1/f])$, the dimension
of the closure of its
image in $\spec(A''[1/f])$ is strictly less
than 
the closure of its image in $\spec(A)$. 
By induction on the Krull dimension, we can conclude (using the statement of the
theorem applied to $\spec(A''[1/f])$) that $\sF( A''[1/f])$ is
1-connective (i.e., connected). 
It follows that $\sF(A') \neq \emptyset$, a contradiction which proves the
theorem. 
\end{proof} 

\subsection{The Nisnevich topos}
Next, we would like to prove the finite homotopy dimension assertion
for Nisnevich sites of algebraic spaces.
Let us first fix the definitions.

\begin{definition}[The Nisnevich and \'etale sites]
\label{def:Nisandetale}
Let $X$ be a qcqs algebraic space (over $\mathbb{Z}$).  Define the \emph{\'etale site} $X_{et}$ and the Nisnevich site $X_{Nis}$ as follows:
\begin{enumerate}
\item The underlying category in both cases is the category $et_X$ of \'etale maps of algebraic spaces $U\rightarrow X$ such that $U$ is qcqs.
\item A sieve over $U\in et_X$ is covering for $X_{et}$ iff it contains finitely many $U_1,\ldots U_n$ mapping to $U$ such that $\sqcup_{i} U_i\rightarrow U$ has nonempty pullback to every point of $U$.
\item A sieve over $U\in et_X$ is covering for $X_{Nis}$ iff it contains finitely many $U_1,\ldots U_n$ mapping to $U$ such that $\sqcup_{i} U_i\rightarrow U$ admits a section after pullback to every point of $U$.
\end{enumerate}
\end{definition}

The Nisnevich topos was introduced in \cite{Nis89} in the 
case of a noetherian scheme. See \cite[Sec. 3.7]{SAG} for a detailed treatment in 
the present general setting. Note that the definition given there is a priori slightly
stronger (an \'etale map $p: Y \to X$ is a cover if there is a stratification of
$X$ such that $p$ admits a section along each stratum). It was shown in
\cite[Appendix A]{BH}
that it is actually enough to demand lifting of field-valued points; 
the treatment in \emph{loc.~cit.} is for qcqs schemes, but qcqs algebraic spaces
admit Nisnevich covers by affine schemes, 
cf.\ \cite{Rydh15},
\cite[Sec. 3.4.2]{SAG}.

\begin{remark}One could more generally allow $X$ to be an algebraic space over the sphere spectrum, but the extra generality is spurious for present purposes, because every algebraic space over the sphere spectrum has an underlying ordinary algebraic space with the same \'etale and Nisnevich sites.\end{remark}
A basic example is the following. 

\begin{example}[The Nisnevich site of a field] 
\label{Nisfield}
Let $x = \spec (k)$ for $k$ a field. Then the Nisnevich site is equivalent to the
site of \'etale $x$-schemes (i.e., the opposite category to products of finite
separable extensions of $k$) and the topology is that of finite disjoint unions, i.e.\ the sheaves are exactly those presheaves
which preserve finite products. 
Equivalently, we can define this to be the site  of  
finite continuous $\mathrm{Gal}(k)$-sets with the topology of finite disjoint unions.  We denote this category by $\mathcal{T}_x$ so that Nisnevich
sheaves are finitely product-preserving presheaves on $\mathcal{T}_x$. 
\end{example} 

Our main result is the following. 

\begin{theorem}
\label{Nisfinitehtpy}
Let $X$ be a qcqs algebraic space, and suppose that the underlying topological
space $\vert X\vert$ (which is a spectral space, cf.\ \cite[Prop.~3.6.3.3]{SAG}) has Krull dimension $\leq d$.  Then $\operatorname{Sh}(X_{Nis})$ has homotopy dimension $\leq d$.
\end{theorem}

 The descent spectral sequence for sheaves of spectra is constructed in 
 \cite{Nis89} 
 under the
assumption that $X$ is a noetherian scheme of Krull dimension $\leq d$. 
Under the same assumptions, the result that the Nisnevich topos has cohomological dimension $\leq d$ 
appears in \cite{KS83}, and   
the homotopy dimension assertion (which implies all the others)
appears in \cite[Theorem 3.7.7.1]{SAG}. 
The purpose of this subsection is thus to remove the noetherian assumptions on all of these results. 

The idea of the proof of \Cref{Nisfinitehtpy} will be the same as in the proof of \Cref{spectralspacehtpy}: to find finite approximations to the Nisnevich site.  For this we need a technical lemma.  In its statement we will refer to the concept of a \emph{spectral stratification} of a qcqs algebraic space $X$.  This just means a spectral map $p:\vert X\vert\rightarrow S$ with $S$ a finite spectral space.  The \emph{strata} are the locally closed subspaces $\{s\}\times_S X\subset X$.  For specificity these can be equipped with their reduced structure sheaves, but in fact every statement we make about them will be independent of the choice of structure sheaf.

\begin{lemma}
\label{colimofstratification}
Let $X$ be a qcqs algebraic space and let $f:Y'\rightarrow Y$ be a Nisnevich covering map of qcqs \'etale $X$-spaces.  Then there is a spectral stratification $p:\vert X\vert\rightarrow S$ such that for all $s\in S$, the map $f$ admits a section over the stratum $\{s\}\times_S Y\subset Y$.
\end{lemma}

\begin{proof}
First we claim that if the lemma is known for all the strata of a spectral
stratification of $X$, then it can be deduced for $X$.  Indeed, since every
spectral space is a filtered inverse limit of finite spectral spaces \cite[Prop.
10]{Ho69}, every partition of a spectral space into finitely many constructible subspaces can be refined by a partition given by the strata of a spectral stratification.  Thus, if we have a stratification of each stratum of $X$, then possibly after refinement we can collect all of the strata of strata together to a single stratification of $X$, verifying the reduction claim.

By 
\cite[Prop. 4.4]{Rydh15}, 
there exists a spectral stratification of $X$ such that $Y$ is finite \'etale
over each stratum; thus we can reduce to the case where $Y$ is finite \'etale
over $X$.  By the ``\'etale d\'evissage" result of \cite{Rydh15} (or the
``scallop decomposition" of \cite[Theorem 3.4.2.1]{SAG}), we can further reduce to the case where $X$ is affine.  
In fact, the theory of scallop decompositions \cite[Definition
2.5.3.1]{SAG} shows that 
$X$ admits a finite filtration by quasi-compact open
subsets $U_i \subset X$
such that each successive difference $U_{i+1} \setminus U_i$ (e.g., endowed with the
reduced algebraic subspace structure) is affine. 
By noetherian approximation, we can then assume $X$ is noetherian (and forget that it's affine if we like).

In that case, proceeding by noetherian induction, it suffices to show that there is a non-empty open subset $U\subset X$ such that $f$ has a section over $U\times_XY$.  Let $x\in X$ be a minimal point (equivalently, a generic point of an irreducible component of $X$).  Then $\{x\}\times_X Y$ consists of a finite set $y_1,\ldots, y_n$ of minimal points of $Y$, which can therefore be separated by disjoint open neighborhoods $V_1,\ldots, V_n$ of $Y$.  By the Nisnevich property and spreading out, we can assume that $f$ has a section over each $V_i$, hence over $V=\cup_i V_i$.  Let $U$ denote the complement of the image of $Y - V$ in $X$.  Then $U$ is open as $f$ is finite, $U$ contains $x$ and hence is non-empty, and $f$ has a section over $U\times_XY$ since $U\times_XY \subset V$.  Thus $U$ satisfies the desired conditions, finishing the proof.
\end{proof}

Now we define the desired finite approximations to the Nisnevich site.

\begin{definition}
Let $X$ be a qcqs algebraic space, and $p:\vert X\vert\rightarrow S$ a spectral
stratification of $X$.  Define the \emph{$p$-Nisnevich site of $X$}, denoted
$X_{Nis,p}$, to be the category  of qcqs \'etale
$X$-spaces, equipped with the Grothendieck topology where a sieve over $Y$ is
covering if and only if it contains finitely many $\{Y_i\rightarrow Y\}_{i\in
I}$ such that $\sqcup_{i\in I} Y_i\rightarrow Y$ admits a section over $\{s\}\times_S Y$ for all $s\in S$.
\end{definition}

The axioms of a finitary Grothendieck topology follow from the definition. The
following is then immediate from \Cref{colimofstratification}:

\begin{lemma}
\label{filteredcolimitNispequivalence}
As $p$ runs over the filtered system of all spectral stratifications of $X$, we
have an equivalence of finitary sites
$$\varinjlim_p X_{Nis,p} = X_{Nis}$$
where all the transition maps, and the identification, are given by the identity
functor on the category of qcqs \'etale $X$-schemes.
\end{lemma}

We also have the following, recovering the well-known result 
of Morel-Voevodsky
(\cite[Prop. 1.4]{MV99} and \cite[Theorem 3.7.5.1]{SAG})
in the limit over all spectral stratifications:

\begin{lemma}
\label{Nisnevichexcision}
Let $p:\vert X\vert\rightarrow S$ be a spectral stratification of a qcqs
algebraic space $X$.  A presheaf $\mathcal{F}\in\operatorname{PSh}(et_X)$ is a sheaf for $X_{Nis,p}$ if and only if the following conditions are satsified:
\begin{enumerate}
\item $\mathcal{F}$ sends finite coproducts in $et_X$ to finite products of spaces;
\item For every closed subset $Z\subset S$ and every map $Y'\rightarrow Y$ in $\operatorname{et}_X$ which is an isomorphism over $Z\times_S Y$, the map $\mathcal{F}(Y)\rightarrow \mathcal{F}(Y')\times_{\mathcal{F}(U')}\mathcal{F}(U)$ is an equivalence, where $U=(S-Z)\times_S Y$ and $U'=(S-Z)\times_S Y'$.
\end{enumerate}
In particular, $X_{Nis,p}$ is excisive.
\end{lemma}

\begin{proof}
If $S$ is empty, then so is $X$ and thus the claim is trivial.  We can therefore
assume $S$ is nonempty and proceed by induction on the number of elements of
$S$.  To make use of the inductive hypothesis, it is useful to note that for
every open subset $U\subset S$, if a presheaf $\mathcal{F}$ on $et_X$ satisfies
conditions 1 and 2 above, then its restriction to the full subcategory $\operatorname{et}_{U\times_S X}\subset \operatorname{et}_X$ also satisfies conditions 1 and 2 with respect to the restricted stratification $U\times_SX\rightarrow U$.  (This is elementary: for a closed subset $Z\subset U$, consider the closed subset $Z\cup (S-U)$ of $S$.)

First assume conditions 1 and 2 hold.  By \cite[Prop. A.3.3.1]{SAG} and condition 1, to show that
$\mathcal{F}$ is a sheaf on $X_{Nis,p}$ it suffices to show \v{C}ech descent with respect to an arbitrary covering map $Y'\rightarrow Y$ in $X_{Nis,p}$.  Let $s\in S$ be a closed point.  By hypothesis, $Y'\rightarrow Y$ admits a section $\sigma$ over $\{s\}\times_S Y$.  Let
$$Y'':= Y' - (\{s\}\times_S Y' - \sigma(\{s\}\times_S Y)).$$
Note that $Y''$ is an open subspace of $Y'$, as a section of an \'etale map is an open immersion.  Furthermore, we have arranged it so that $Y''\rightarrow Y$ is an isomorphism over $\{s\}\times_SY$, so from condition 2 it follows that to verify \v{C}ech descent for any map $Y_1\rightarrow Y$ it suffices to verify \v{C}ech descent for its pullback to $Y''$, to $(S-\{s\})\times_S Y$, and to $(S-\{s\})\times_S Y''$.  In our situation, where $Y_1\rightarrow Y$ is $f:Y'\rightarrow Y$, the latter two follow from the inductive hypothesis, and the former follows because $Y''\rightarrow Y$ factors through $f$, so that $f$ acquires a section on pullback along $Y''\rightarrow Y$.

To show the converse, assume that $\mathcal{F}$ is a sheaf for $X_{Nis,p}$.  To check that condition 2 holds, it suffices to show that in the situation of condition 2, the induced map on pushouts of representable sheaves $h_U\sqcup_{h_{U'}}h_{Y'}\rightarrow h_{Y}$ is an equivalence.  Since all representables are sheaves of sets (the site is a $1$-category) and $h_{U'}\rightarrow h_{Y'}$ is a monomorphism, $h_U\sqcup_{h_{U'}}h_{Y'}$ is also a sheaf of sets, and therefore it suffices to show that if $\mathcal{F}$ is a Nisnevich sheaf of sets, then $\mathcal{F}(Y)\overset{\sim}{\rightarrow}\mathcal{F}(U)\times_{\mathcal{F}(U')}\mathcal{F}(Y')$.  For this, by Nisnevich descent for the cover $\{U\rightarrow Y,Y'\rightarrow Y\}$, it suffices to show that if a section $s\in \mathcal{F}(Y')$ is such that the image of $s$ in $\mathcal{F}(U')$ comes from $\mathcal{F}(U)$, then $s$ has the same image on pullback along the two different projections $Y'\times_ Y Y'\rightarrow Y'$.  But $Y'\times_Y Y'$ is covered by the diagonal $Y'\rightarrow Y'\times_YY'$ (which is an open immersion since $Y'\rightarrow Y$ is \'etale) together with pullback of $Y'\times_YY'\rightarrow Y$ to $U\subset Y$, which leads to the conclusion.
\end{proof}

\begin{lemma}
\label{pointsofNisp}
Let $p:\vert X\vert\rightarrow S$ be a spectral stratification of a qcqs algebraic space $X$, let $Y\in\operatorname{et}_X$, and let $s\in S$.  For $\mathcal{F}\in\operatorname{Sh}(X_{Nis,p})$, define
$$\mathcal{F}_{Y,s}:=\varinjlim\mathcal{F}(U),$$
the colimit being over the co-filtered category of all $U\rightarrow Y$ in $\operatorname{et}_Y$ which are an isomorphism over $\{s\}\times_SY$.

Then this functor $\mathcal{F}\mapsto\mathcal{F}_{Y,s}:\operatorname{Sh}(X_{Nis,p})\rightarrow\mathcal{S}$ is the pullback functor associated to a point of the $\infty$-topos $\operatorname{Sh}(X_{Nis,p})$.\end{lemma}
\begin{proof}
Consider the functor $\operatorname{et}_X\rightarrow\mathcal{S}$ which sends an
$f:X'\rightarrow X$ in $\operatorname{et}_X$ to the set of sections of $f$ over
$\{s\}\times_SY$.  This preserves finite limits, and the definition of the
topology on $X_{Nis,p}$ implies that the image of a covering sieve for
$X_{Nis,p}$ is epimorphic.  Thus, the criterion of \cite[Prop.
6.2.3.20]{HTT} shows that there is a unique functor $\operatorname{Sh}(X_{Nis,p})\rightarrow\mathcal{S}$ which restricts to our original functor $\operatorname{et}_X\rightarrow\mathcal{S}$ on representables and is the pullback functor associated to a geometric morphism of $\infty$-topoi.  Writing a sheaf in the tautological manner as a colimit of representable sheaves, we deduce that this pullback functor sends
$$\mathcal{F}\mapsto \varinjlim_{\{s\}\times_SY\rightarrow U\rightarrow Y}\mathcal{F}(U),$$
where the indexing category is all \'etale neighborhoods of $\{s\}\times_SY$ in $Y$.  This indexing category is co-filtered by fiber products, so by a cofinality argument, to prove the lemma it suffices to see that every \'etale neighborhood $\{s\}\times_SY\overset{\sigma}{\rightarrow} U\rightarrow Y$ admits a refinement $U'\rightarrow U$ for which the map $U'\rightarrow Y$ is an isomorphism over $\{s\}\times_SY$ (in which case the section $\{s\}\times_SY\rightarrow U'$ is unique). For this, let $S_s$ denote the smallest open neighborhood of $s$ in $S$.  Noting that $\{s\}$ is closed in $S_s$, we can take
$$U' = S_s\times_S U - (\{s\}\times_S U -\sigma(\{s\}\times_SY)).$$\end{proof}

\begin{proposition}
Let $p:\vert X\vert\rightarrow S$ be a spectral stratification of a qcqs algebraic space $X$ and let $\mathcal{F}\in\operatorname{Sh}(X_{Nis,p})$.  Adopting the notation of the previous lemma, if $\mathcal{F}_{Y,s}$ is $\operatorname{dim}(S)$-connective for all $Y\in\operatorname{et}_X$ and $s\in S$, then $\mathcal{F}(X)\neq\emptyset$.
\label{Nispexistsections}
\end{proposition}
\begin{proof}
The case $S=\emptyset$ is trivial, so assume $S$ nonempty and proceed by induction on the number of points of $S$.

Let $s$ be a closed point of $S$, and let $S_s$ be the smallest open
neighborhood of $s$ in $S$.  There are two cases: $S_s\neq S$, and $S_s=S$.
First consider the first case.  By Zariski descent (or \Cref{Nisnevichexcision}) we have
$$\mathcal{F}(X)\overset{\sim}{\longrightarrow}\mathcal{F}(S_s\times_S X)\times_{\mathcal{F}((S_s-s)\times_S X)}\mathcal{F}((S-s)\times_S X).$$
We have $\mathcal{F}((S-s)\times_S X)\neq\emptyset$ and $\mathcal{F}(S_s\times_S X)\neq\emptyset$ by the inductive hypothesis, and $\mathcal{F}((S_s-s)\times_S X)$ is even \emph{connected} by the inductive hypothesis, as $\operatorname{dim}(S_s-s)\leq\operatorname{dim}(S)-1$.  Hence we deduce $\mathcal{F}(X)\neq\emptyset$, as claimed.

To finish, assume we are in the second case, so $S_s=S$.  Define
$$\mathcal{F}'_{X,s}:= \varinjlim \mathcal{F}((S-s)\times_SU),$$
where the indexing category is the same as the one which defined
$\mathcal{F}_{X,s}=\varinjlim\mathcal{F}(U)$, namely all the $U\rightarrow X$ in
$\operatorname{et}_X$ which are an isomorphism over $\{s\}\times_SX$.  Then
passing to filtered colimits in \Cref{Nisnevichexcision} shows that
$$\mathcal{F}(X)\overset{\sim}{\longrightarrow}\mathcal{F}((S-s)\times_S X)\times_{\mathcal{F}'_{X,s}}\mathcal{F}_{X,s}.$$
By the inductive hypothesis, $\mathcal{F}((S-s)\times_S X)$ is non-empty (and even connected, as $\operatorname{dim}(S-s)\leq\operatorname{dim}(S)-1$), and likewise $\mathcal{F}'_{X,s}$ is a filtered colimit of connected spaces, hence is connected.  Since $\mathcal{F}_{X,s}$ is non-empty by hypothesis, we deduce $\mathcal{F}(X)\neq\emptyset$, as desired.
\end{proof}
\begin{corollary}
\label{htpydimNisp}
Let $p:\vert X\vert\rightarrow S$ be a spectral stratification of a qcqs algebraic space $X$.  Then:
\begin{enumerate}
\item As $Y$ varies over $\operatorname{et}_X$ and $s$ varies over $S$, the
points of $\operatorname{Sh}(X_{Nis,p})$ described in \Cref{pointsofNisp} form a conservative family of points of the $\infty$-topos $\operatorname{Sh}(X_{Nis,p})$.
\item The homotopy dimension of $\operatorname{Sh}(X_{Nis,p})$ is $\leq \operatorname{dim}(S)$, the Krull dimension of $S$.
\end{enumerate}
\end{corollary}
\begin{proof}
For 1, suppose $\mathcal{F}\rightarrow\mathcal{G}$ is a morphism in $\operatorname{Sh}(X_{Nis,p})$ such that $\mathcal{F}_{Y,s}\rightarrow\mathcal{G}_{Y,s}$ is an equivalence for all $Y\in\operatorname{et}_X$ and $s\in S$.  We want to see that $\mathcal{F}\rightarrow\mathcal{G}$ is an equivalence.  By repeating the argument for every $U\in\operatorname{et}_X$, it suffices to see that $\mathcal{F}(X)\rightarrow\mathcal{G}(X)$ is an equivalence.  Considering the homotopy fiber at an arbitrary point of $\mathcal{G}(X)$,  we can further reduce to the case where $\mathcal{G}=\ast$.  Then the hypothesis is that $\mathcal{F}_{Y,s}$ is $\ast$, hence is $d$-connective for all $d\geq 0$; therefore $\mathcal{F}(X)$ is $d-\operatorname{dim}(S)$-connective for all $d$, hence is also $\ast$, as required.

For 2, if $\mathcal{F}$ is $\operatorname{dim}(S)$-connective, then
$\mathcal{F}_{Y,s}$ is $\operatorname{dim}(S)$-connective for all $Y$ and $s$
since pullbacks preserve connectivity.  Thus \Cref{Nispexistsections} implies $\mathcal{F}(X)\neq\emptyset$, as desired.
\end{proof}

Combining everything, we deduce the desired theorem (\Cref{Nisfinitehtpy}) on
finiteness of the homotopy dimension of $X_{Nis}$: from
\Cref{filteredcolimitNispequivalence} and
\Cref{filtcolimitdimbound}  we reduce to the analogous claim for the $X_{Nis,p}$,
which is \Cref{htpydimNisp} part 2 (note also that there is a cofinal collection
of spectral stratifications $p:\vert X\vert\rightarrow S$ with
$\operatorname{dim}(S)\leq \operatorname{dim}(\vert X\vert)$, by \cite{Sal91}).
Inputting the classical fact that the ordinary topos of Nisnevich sheaves on $X$
has enough points given by the finite separable extensions of the residue fields
of $X$ (which can also be recovered in the limit over $p$ from
\Cref{htpydimNisp} part
1), we can rephrase \Cref{Nisfinitehtpy} as follows:

\begin{theorem}
Let $X$ be a qcqs algebraic space of Krull dimension $d<\infty$.  If
$\mathcal{F}\in\operatorname{Sh}(X_{Nis})$ is such that the Nisnevich pullback
$x^\ast\mathcal{F}\in\operatorname{Spec}(k(x))_{Nis}=\operatorname{PSh}_{\Pi}(\mathcal{T}_{\spec
k(x)})$
(cf.\ \Cref{Nisfield}) is $d$-connective for all $x\in X$, then $\mathcal{F}(X)\neq\emptyset$.
\label{connectiveNisdsection}
\end{theorem}

Here are the standard corollaries.
These follow from 
\Cref{connectiveNisdsection} 
and 
\cite[Proposition 7.2.1.10]{HTT}, 
in light of \cite[Remark 6.5.4.7]{HTT} and \cite[Corollary 7.2.2.30]{HTT}. 

\begin{corollary}
\label{Nishypercomplete}
Let $X$ be a qcqs algebraic space of Krull dimension $d<\infty$.  If
$f:\mathcal{F}\rightarrow\mathcal{G}$ is a map of sheaves of spaces (or spectra)
on $X_{Nis}$ such that $x^\ast f:x^\ast\mathcal{F}\rightarrow x^\ast\mathcal{G}$
is an equivalence for all $x\in X$, then $f$ is an equivalence. In particular,
all sheaves are hypercomplete. 
\end{corollary}

\begin{corollary}
Let $X$ be a qcqs algebraic space of Krull dimension $d$.  Then $X_{Nis}$ has cohomological dimension $\leq d$.
\end{corollary}

This bound on cohomological dimension has the following further consequence; at
least in the noetherian case, this is well-known, cf.\ for instance \cite[Theorem
2.8]{Mi97} for the argument. 

\begin{corollary}
\label{etalecohdim}
Let $X$ be a qcqs algebraic space of Krull dimension $d$, and $\mathcal{P}$ a set of primes.  For $x\in X$, denote by $cd_x$ the $\mathcal{P}$-local cohomological dimension of the absolute Galois group of the residue field at $x$.  Then
$$\operatorname{sup}_{x\in X} cd_x \leq \operatorname{sup}_{U\rightarrow X\in\operatorname{et}_X}\operatorname{CohDim}_{\mathcal{P}}(U_{et})\leq d + \operatorname{sup}_{x\in X}cd_x.$$
\end{corollary}
\begin{proof}
First we show the first inequality.  In fact we show something slightly stronger: for all $x\in X$,
$$cd_x\leq \operatorname{sup}_{x\in U\subset X}\operatorname{CohDim}_{\mathcal{P}}(U_{et}),$$
where the supremum is now over all Zariski open neighborhoods of $x$ in $X$.
For this, note that since a point in a spectral space is pro-constructible, $x$
is a filtered inverse limit of finitely presented closed subspaces $Z\subset U$
of quasicompact Zariski open neighborhoods $U$ of $x$.  Standard continuity results in
\'etale cohomology \cite[Tag 03Q4]{stacks-project} then imply that the \'etale cohomological dimension of $x$ is bounded by the supremum of the \'etale cohomological dimensions of such $Z$.  But \'etale pushforwards along closed immersions are exact, so the cohomological dimension of each $Z$ is bounded by that of its corresponding $U$, as desired.

For the second inequality, since the Krull dimension of each
$U\in\operatorname{et}_X$ is bounded by that of $X$, and likewise the Galois
cohomological dimension of each residue field of $U$ is  bounded in terms of those of $X$ (residue fields of $U$ are finite separable extensions of residue fields of $X$), it suffices to show that
$$\operatorname{CohDim}(X_{et})\leq d + \operatorname{sup}_{x\in X}cd_x.$$
For this, let $\mathcal{A}$ be a $\mathcal{P}$-local sheaf of abelian groups on $X_{et}$.
We wish to show that $H^i(X_{et}; \sF) = 0$ for $i> 
d + \operatorname{sup}_{x\in X}cd_x$. 
To this end, we consider the Leray spectral sequence for the pushforward
$\lambda$ from
the \'etale to the Nisnevich site, 
$$ H^i(X_{Nis}, R^j \lambda_*( \mathcal{F})) \implies H^{i+j}(X_{et}, \sF).$$
Since 
the Nisnevich site has cohomological dimension $\leq d$, 
it suffices to show that $R^j \lambda_*( \mathcal{F}) = 0$ for 
$j
> \operatorname{sup}_{x\in X}cd_x$. 
But the Nisnevich stalk of $R^j \lambda_*(\mathcal{F})$ at $y \in X$ is given by 
$H^j(
\spec(\mathcal{O}_{X,y}^h), \sF)$, where $\mathcal{O}_{X,y}^h$ denotes the
henselian local ring at $y$, and by definition this vanishes for
$j>\mathrm{cd}_y$. 
\end{proof}

Finally, as in 
\Cref{htpydimspectralbyprosite}, 
we include an alternative argument
for 
\Cref{Nisfinitehtpy}, which gives a slightly stronger connectivity bound. 
For simplicity, we only treat the case of qcqs schemes. 

\begin{theorem} 
Let $X$ be a qcqs scheme of finite Krull dimension and let $\sF \in \sh(X)$. 
Suppose that for each $x \in X$,
the Nisnevich pullback $x^* \sF \in \sh( (\spec k(x))_{Nis})$
is $\dim( \overline{\left\{x\right\}})$-connective. 
Then $\sF(X) \neq \emptyset$. 
\end{theorem} 
\begin{proof} 
We will prove this by induction on the Krull dimension of $X$, starting with the
(evident) case of dimension $-1$. 
In view of \Cref{htpydimspectralbyprosite}, it suffices to show that for each $x \in X$, the Zariski stalk of $\sF$ at $x$ is 
$\dim( \overline{\left\{x\right\}})$-connective for each $x \in X$. 
In particular, using that the Nisnevich site is finitary excisive to pass to the
limit, we may assume that $X = \spec(A)$ for $A$ a local
ring with closed point $x$. Since the Nisnevich pullback of $X$ to $x$ is
0-connective, it follows that there exists an \'etale neighborhood $A'$ of $x$
(in particular, $\spec(A') \to \spec(A)$ pulls back along $\spec(k(x))$ to an
isomorphism)
such that $\sF(A') \neq \emptyset$. 
There exists a finitely generated ideal $I \subset A$ such that $A \to A'$ is an
isomorphism along $I$. Therefore, we have a pullback square
\[ \xymatrix{
\sF( A) \ar[d]  \ar[r] &  \sF( A') \ar[d]  \\
\sF( \spec A \setminus V(I)) \ar[r] &  \sF( \spec A' \setminus V(I))
}.\]
The observation now is that the maps $\spec (A) \setminus V(I) \to \spec(A)$ and
$\spec A' \setminus V(I) \to \spec A'$ miss the closed point, and hence the
inductive hypothesis applies to show that $\sF$ takes 1-connective values on each
of these schemes. 
Since $\sF(A') \neq \emptyset$, it follows that the above pullback $\sF(A)$ is
nonempty, as desired. 
\end{proof} 

\begin{remark} 
In \cite{RO06}, it is asserted (cf.\ Theorem 4.1 of \emph{loc. cit.}) that
algebraic $K$-theory is a hypersheaf in the Zariski or Nisnevich topology on a
qcqs scheme $X$ of finite Krull dimension. 
The argument relies on the fact that this is well-known (due to \cite{TT90}) in the case where 
$X$ is \emph{noetherian} and of finite Krull dimension, and then one can write 
\emph{any} $X$ as a filtered limit of schemes $X_i$ of finite type over
$\mathbb{Z}$, and appeal to continuity results such as \cite[Theorem 3.8]{Mi97}.
To the best of our knowledge, however, such an argument requires the $X_i$ to have
\emph{uniformly} bounded Krull dimension, which generally cannot be arranged
even if $X$ has finite Krull dimension.
Nonetheless, our arguments here show that there is in fact no distinction
between sheaves and hypersheaves on a qcqs scheme $X$ (or algebraic space) of finite
Krull dimension. Since the results of \cite{TT90} are sufficient to show that
$K$-theory is at least a sheaf on any qcqs algebraic space, it follows that the
assertions of \cite{RO06} do indeed hold without noetherian hypotheses. 
\end{remark}

\section{The \'etale topos}

\newcommand{\pshp}{\mathrm{PSh}_{\Pi}}
In this section, we study the $\infty$-category $\mathrm{Sh}(X_{et}, \sp)$ of 
sheaves of spectra on the \'etale site of a qcqs algebraic space $X$ of finite
Krull dimension and with a global bound on the \'etale cohomological dimension
of the residue fields: by \Cref{etalecohdim}, this implies that the \'etale
topos of $X$ has finite cohomological dimension. 

Contrary to the examples of the previous subsection, the $\infty$-topos
$\mathrm{Sh}(X_{et})$
need not have finite homotopy dimension and 
$\mathrm{Sh}(X_{et}, \sp)$ is generally not hypercomplete. 
Our main result (\Cref{main:hypcriterion}) is that hypercompletion
is 
smashing 
and that we can make explicit the condition of hypercompleteness on a sheaf of
spectra. 

Throughout this section, we will fix a set of primes $\mathcal{P}$ and will be
working with $\mathcal{P}$-local objects. 

\subsection{The classifying topos of a profinite group}
Let $G$ be a profinite group. 
First, we review the  topos of continuous $G$-sets, cf.\ \cite[Sec. III.9]{MLM94}
for a treatment.

\begin{definition} 
Let $\mathcal{T}_G$ denote the Grothendieck site defined as follows:
\begin{enumerate}
\item The underlying category of $\mathcal{T}_G$ is the category of finite
continuous $G$-sets. 
\item A system of maps $\left\{S_i \to S\right\}_{i \in I}$ forms a covering
sieve if it
is jointly surjective. 
\end{enumerate}

Given an $\infty$-category $\mathcal{D}$ with all limits, 
we let $\mathrm{Sh}(\mathcal{T}_G, \mathcal{D})$ denote the $\infty$-category of
$\mathcal{D}$-valued sheaves on $\mathcal{T}_G$, as usual. 
We also write $\pshp(\mathcal{T}_G, \mathcal{D})$
for the $\infty$-category of presheaves on $\mathcal{T}_G$ with values in
$\mathcal{D}$ which carry finite coproducts on $\mathcal{T}_G$ to finite
products; equivalently, these are $\mathcal{D}$-valued presheaves on the orbit
category of $G$. 
\end{definition} 

Clearly, $\mathcal{T}_G$ is a finitary site. 
The category of sheaves of sets on $\mathcal{T}_G$ is the category of continuous
(discrete) $G$-sets. More generally, we can describe sheaves on $\mathcal{T}_G$ in an
arbitrary $\infty$-category with limits. 
The category of finite continuous $G$-sets has pullbacks, and the pullback of a
surjection is a surjection. Moreover, coproducts are disjoint. Consequently, 
we note that this Grothendieck topology is a special case of a
general construction described in \cite[Sec. A.3.2, A.3.3]{SAG}.  
In particular, \cite[Prop. A.3.3.1]{SAG} yields the following criterion for a
presheaf to be a sheaf: 
\begin{proposition} 
\label{Gsheafcrit}
Let $G$ be a profinite group, and $\mathcal{T}_G$ the site as above. 
A presheaf $\sF$ on $\mathcal{T}_G$ with values in an $\infty$-category
$\mathcal{D}$ with limits is a sheaf if and only if: 
\begin{enumerate}
\item  For $X, Y \in \mathcal{T}_G$, the natural map gives an equivalence $\sF(X
\sqcup Y) \simeq \sF(X) \times \sF(Y)$. 
That is, $\sF \in \pshp(\mathcal{T}_G, \mathcal{D})$. 
\item For every 
surjective map of $G$-sets $T \twoheadrightarrow S$, the natural map 
\begin{equation} \label{finiteGsetCech} \sF(S) \to \mathrm{Tot}( \sF(T)
\rightrightarrows \sF(T \times_S T) 
\triplearrows
\dots )  \end{equation}
is an equivalence. 
\end{enumerate}
\end{proposition} 

We now need a basic lemma about Kan extensions. 
\begin{lemma} 
\label{Kanlemma}
Let $\mathcal{C}$ be an $\infty$-category with finite nonempty products and let $X \in
\mathcal{C}. $ Let $\mathcal{D}$ be an $\infty$-category with all limits. 
Let $\sF: \mathcal{C} \to \mathcal{D}$ be a functor and let $\mathcal{C}' \subset
\mathcal{C}$ be the full subcategory of objects $Y$ which admit a map $Y \to X$. 
Let $\widetilde{\sF}$ denote the right Kan extension of $\sF|_{\mathcal{C}'}$ to
$\mathcal{C}$. 
Then $\widetilde{\sF}$ is given by the formula
\begin{equation} \label{Kextformula} \widetilde{\sF}(Z) = \mathrm{Tot}(\sF(Z
\times X) \rightrightarrows \sF(Z \times X
\times X )  \triplearrows \dots ).\end{equation} 
\end{lemma} 
\begin{proof} 
It suffices to see that the functor 
$ \Delta^{op} \to \mathcal{C}'_{/Z}$ 
given by $Z \times X^{\bullet +1}$
is cofinal. 
This is an easy argument with Quillen's Theorem A, cf.\ \cite[Prop. 6.28]{MNN17}. 
\end{proof}

\begin{example}[Finite groups] 
\label{ex:sheavesonBG}
Let $G$ be a finite group. 
In this case, $\mathcal{T}_G$ is the category of finite $G$-sets. 
We observe that a functor $\sF: \mathcal{T}_G^{op} \to \mathcal{D}$ (where
$\mathcal{D}$ is an $\infty$-category with all limits) is 
a sheaf if and only if it is right Kan extended from the full subcategory of
$\mathcal{T}_G$ spanned by the $G$-set $G$. 
Indeed, by \Cref{Kanlemma} 
any sheaf $\sF$ is Kan extended from the finite free $G$-sets. 
Since $\sF$ takes finite coproducts of $G$-sets to finite products, $\sF$ is
actually right Kan extended from the subcategory $\left\{G\right\}$ itself. 
Conversely, 
given 
a functor $\sF_0: BG \to \mathcal{D}$, let $\sF$ be the right Kan extension to
$\mathcal{T}_G$. It is now easy to see that $\sF$ carries finite coproducts to
finite products and (via the expression \eqref{Kextformula}) satisfies the sheaf
property: any surjection $T \twoheadrightarrow S$ of $G$-sets admits a section after 
taking the product with $G$. 

Consequently, one sees that $\mathrm{Sh}(\mathcal{T}_G, \mathcal{D})$ is
equivalent to the $\infty$-category $\fun(BG, \mathcal{D})$ of objects of
$\mathcal{D}$ equipped with a $G$-action. 
One can also see this by reducing to the case $\mathcal{D}  = \mathcal{S}$ and
then observing that both sides are 1-localic $\infty$-topoi and agree on
discrete objects.  
Given a surjection of finite groups $G \twoheadrightarrow G'$ with kernel $N
\leq G$,
we obtain an inclusion $\mathcal{T}_{G'} \subset \mathcal{T}_G$ and thus by
restriction a
functor $\mathrm{Sh}(\mathcal{T}_G, \mathcal{D}) \to
\mathrm{Sh}(\mathcal{T}_{G'}, \mathcal{D})$; 
unwinding the definitions, one sees that under the above identifications 
this becomes the functor 
$\fun(BG, \mathcal{D}) \to \fun(BG', \mathcal{D})$ given by $(-)^{hN}$. 

Let $\mathcal{O}(G) \subset \mathcal{T}_G$ be the full subcategory spanned by
nonempty \emph{transitive} $G$-sets, i.e., those of the form $G/H, H \leq G$;
$\mathcal{O}(G)$ is called the \emph{orbit category} of $G$. 
Similarly, a functor $\sF: \mathcal{T}_G^{op} \to \mathcal{D}$ preserves finite
products if and only if it is right Kan extended from 
the inclusion $\mathcal{O}(G)^{op} \subset \mathcal{T}_G^{op}$. The category
$\pshp(\mathcal{T}_G, \mathcal{D})$ of
product-preserving presheaves $\mathcal{T}_G^{op} \to \mathcal{D}$ is thus
identified with $\fun(\mathcal{O}(G)^{op}, \mathcal{D})$. 

\end{example} 

We can recover any profinite group as a limit of finite groups, and on
Grothendieck sites we obtain a filtered colimit. This leads to the next
construction: 

\begin{construction}[$\mathcal{T}_G$ as a filtered colimit]
\label{explicit:sheavesonprofinite}
Let $G$ be a profinite group. For each open normal subgroup $N \leq G$, we have 
the finitary site $\mathcal{T}_{G/N}$. When $N' \leq N$, we have a natural
functor of finitary sites
$\mathcal{T}_{G/N} \to \mathcal{T}_{G/N'}$ by pulling back along the quotient
$G/N' \to G/N$. Moreover, it is easy to see that
$\mathcal{T}_G = \varinjlim_{N} \mathcal{T}_{G/N}$, the colimit taken over all
open normal subgroups $N \leq G$. In particular, if $\mathcal{D}$ is any
$\infty$-category with limits, we have an equivalence of $\infty$-categories
(cf.\
\Cref{ex:sheavesonBG} and \Cref{filtdiagfinitarysite})
\[ \mathrm{Sh}(\mathcal{T}_G, \mathcal{D}) \simeq \varprojlim_{N \leq G} \fun(
 B(G/N), \mathcal{D}),\]
where for $N' \leq N$, the functor 
$\fun(
 B(G/N'), \mathcal{D}) \to 
\fun( B(G/N), \mathcal{D})
 $ is given by $(\cdot)^{h(N/N')}$. 

 In a similar 
 fashion, we can describe product-preserving presheaves out of $\mathcal{T}_G$. 
 Let 
 $\mathcal{O}(G)$ denote the category of finite continuous nonempty transitive $G$-sets;
 then clearly $\mathcal{O}(G) =\varinjlim_N \mathcal{O}(G/N)$. 
 In view of \Cref{ex:sheavesonBG} again, we find a natural equivalence 
\[ \pshp(\mathcal{T}_G, \mathcal{D}) = 
\fun(\mathcal{O}(G)^{op}, \mathcal{D}) \simeq \varprojlim_N \fun(
\mathcal{O}(G/N)^{op}, \mathcal{D}).
\] 
 \end{construction}

We next describe the Postnikov sheafification following a method motivated by
the pro-\'etale site \cite{BS15}. This can be done without
any extra finiteness assumptions either on $G$ or on the target
$\infty$-category. 
\begin{construction}[Postnikov sheafification]
Let $G$ be a profinite group and let $\sF \in \pshp(\mathcal{T}_G, \sp)$. We can
give an explicit construction of the Postnikov sheafification
$\widetilde{\sF}$ of $\sF$. 
The
construction is analogous for presheaves of spaces.

Since $\sF$ is defined on finite $G$-sets, we can canonically extend it by
continuity to all profinite $G$-sets so that it carries filtered limits of
$G$-sets to
filtered colimits of spectra. 
Given a finite $G$-set $S$, 
we consider the simplicial profinite $G$-set
\begin{equation} \label{simpprofiniteGset}  \dots  \triplearrows G \times G \times S \rightrightarrows  G
\times S  \end{equation}
augmented over $S$ (equivalently, the product of the \v{C}ech nerve of $G \to
\ast$ with $S$), and we claim that
there is a natural identification (functorially in $S$)
\begin{equation} \label{tildeF} \widetilde{\sF}(S) \simeq \mathrm{Tot}( \sF(G
\times S) \rightrightarrows \sF(G \times
G \times S ) \triplearrows \dots ). \end{equation}
Indeed, we observe first that 
the construction \eqref{tildeF} defines a sheaf on $\mathcal{T}_G$. 
This follows in view of \Cref{Gsheafcrit} and the fact that for any surjection
of finite $G$-sets $S \twoheadrightarrow S'$, the map $S \times G \to S' \times
G$ admits a section. 
Next, the construction \eqref{tildeF} clearly commutes with the inverse
limit along Postnikov towers in
$\sF$. Thus, it remains to show that if $\sF$ is coconnective, then
$\widetilde{\sF}(S)$ is the sheafification of $\sF$; indeed, the Postnikov
sheafification of $\sF$ is the inverse limit of the sheafifications of the
presheaf-level truncations of $\sF$. In this case, we can
simply check on
homotopy group sheaves: the map $\sF \to \widetilde{\sF}$ of presheaves on
$\mathcal{T}_G$ induces an isomorphism of stalks (i.e., extending
$\widetilde{\sF}$ to profinite $G$-sets and after evaluating on
the profinite $G$-set $G$), 
because the simplicial diagram 
of profinite $G$-sets
\eqref{simpprofiniteGset} 
for $S = G$ admits a splitting. 
Hence $\sF \to \widetilde{\sF}$ induces an 
isomorphism on sheafification.  
\end{construction}

\begin{construction}[Restriction to open subgroups]
\label{restricttoopen}
Let $G$ be a profinite
group, 
and let $H \leq G$ be an open subgroup. 
We consider the functor $u: \mathcal{T}_H \to \mathcal{T}_G$ 
sending an $H$-set $S$ to the induced $G$-set $G \times_H S$. 
This is a morphism of finitary sites, 
so it induces a morphism $u_*$
(via precomposition) from sheaves (of spaces) on $\mathcal{T}_G$ to sheaves on $\mathcal{T}_H$. 
Using \Cref{resGcommutesheaf} below, it follows that $u_*$ commutes with
all limits and colimits and with taking
truncations, and hence with Postnikov completions. 
\end{construction}

\begin{definition}[Weakly nilpotent group actions]
Let $K$ be a finite group acting on a spectrum $X$. We say that the
$K$-action is \emph{weakly $m$-nilpotent} if the standard cosimplicial spectrum
$X \rightrightarrows \prod_K X \triplearrows \dots $ computing $X^{hK}$ (and
coaugmented over $X^{hK}$)
is weakly $m$-rapidly converging. 
\end{definition}

\begin{proposition}
\label{Postcompletecrit}
Let $G$ be a profinite group, and $\mathcal{F}$ a sheaf of spectra on
$\mathcal{T}_G$.  Suppose there exists a $d\geq 0$ such that for each normal
containment $N\subset H$ of open subgroups of $G$, the spectrum with
$H/N$-action $\mathcal{F}(G/N)$ is weakly $d$-nilpotent.  Then $\mathcal{F}$ is Postnikov complete.
\end{proposition}
\begin{proof}
Let $\widetilde{\sF}$ denote the Postnikov completion of $\sF$. 
It suffices to show that the map $\sF \to \widetilde{ \sF}$ induces an
equivalence after applying to the $G$-set $\ast$; we can then repeat this
argument for every open subgroup by \Cref{restricttoopen}. 
Indeed, by \eqref{tildeF}
\begin{equation} \label{tsFast} \widetilde{\sF}(\ast) \simeq 
\mathrm{Tot}(\sF(G) \rightrightarrows \sF( G \times G)  \triplearrows \dots ). 
\end{equation}
The associated cosimplicial object is the filtered colimit of the cosimplicial
objects
\[ \sF(G/N) \rightrightarrows \sF(G/N \times G/N) \triplearrows \dots, \]
each of which has totalization given by $\sF(\ast)$ since $\sF$ is a sheaf. 
By assumption, there is a uniform bound on the weak nilpotence of each of these
cosimplicial objects (augmented over $\sF(\ast)$), so we can interchange the totalization and the filtered
colimit (\Cref{commutetotandcolimit}). 
Doing so together with \eqref{tsFast} yields $\sF(\ast) \simeq
\widetilde{\sF}(\ast)$, as desired. 
\end{proof}

We now want to describe the sheafification construction for $\mathcal{T}_G$. 
In general, recall that sheafification is a difficult construction 
because it requires a transfinite \v{C}ech construction (cf.\ \cite[Prop.
6.2.2.7]{HTT}), and we cannot expect something as straightforward as the
formula \eqref{tildeF}. 
The primary issue is that 
totalizations need not commute with filtered colimits. 
In this case, this turns out to be essentially the only obstruction. 
To this end, we describe a slightly different method of writing $\mathcal{T}_G$
as a filtered colimit of ``smaller'' finitary sites. Here the
categories stay the same, but only the topologies change. 
\begin{construction}[The Grothendieck site $\mathcal{T}_G^N$]
Let $N \leq G$ be an open normal subgroup. We define a finitary
Grothendieck topology on 
the category of finite continuous $G$-sets
such that a family $\left\{T_i \to
T\right\}$ is a covering if $\bigsqcup_i T_i \times G/N \to T \times G/N$ admits
a section. 
We write $\mathcal{T}_G^N$ for the associated 
site. 

Just as in \Cref{Gsheafcrit}, we see (via \cite[Sec. A.3.2-A.3.3]{SAG}) that a
functor $\sF \in \pshp(\mathcal{T}_G^N,
\mathcal{D})$ is a sheaf if and only if for every surjection $T
\twoheadrightarrow S$ of finite continuous $G$-sets such that $T \times G/N \to
S \times G/N$ admits a section, the natural map \eqref{finiteGsetCech} is an
equivalence (i.e., $\sF$ has \v{C}ech descent for the map). 
Note also that as $N$ ranges over all open normal subgroups, the
$\mathcal{T}_G^N$ form a filtered system of finitary sites whose filtered
colimit is $\mathcal{T}_G$. 
\end{construction}

One advantage of $\mathcal{T}_G^N$ over $\mathcal{T}_G$ is that we can very
explicitly 
describe sheafification. 
\begin{proposition} 
\label{FGNsheafification}
Let $\sF \in \pshp(\mathcal{T}_G^N, \mathcal{D})$. Then the 
sheafification  of $\sF$ in $ \mathrm{Sh}(\mathcal{T}_G^N, \mathcal{D})$ is
given via the formula (for any $S \in \mathcal{T}_G^N$, i.e., any finite
continuous $G$-set)
\begin{equation} \sF_{G/N}(S) 
\simeq \mathrm{Tot}( \sF(G/N \times S) \rightrightarrows \sF( G /N \times G/N \times
S) \triplearrows \dots ) \simeq \sF(G/N \times S)^{h(G/N)}. \label{FGN}
\end{equation}
\end{proposition} 
\begin{proof} 
We claim first that $\sF_{G/N}$ is a sheaf on $\mathcal{T}_G^N$.
If $T' \twoheadrightarrow T$ is a surjection of finite $G$-sets such that
$G/N\times T' \to G/N \times T$ admits a section (e.g., this is automatic if the
$G$-action on $T'$ factors through $G/N$),
then
we claim
\[ \sF_{G/N}(T) \to \mathrm{Tot}(\sF_{G/N}( T' ) \rightrightarrows \sF_{G/N}(T'
\times_T T') \triplearrows \dots ) \]
is an equivalence. This follows because the \v{C}ech nerve of $T' \to T$ admits
a contracting homotopy after taking the product with $G/N$, so each of the
terms in the totalization defining $\sF_{G/N}$ in \eqref{FGN} takes the \v{C}ech
nerve of $T' \to T$ to a limit diagram.  
Note also that if $\sF$ was a sheaf to start with, then $\sF \to \sF_{G/N}$ is an
equivalence. 
However, it
remains to verify that $\sF \to \sF_{G/N}$ actually exhibits $\sF_{G/N}$ as the
sheafification of $\sF$. 
That is, if $\sF, \sF' \in \pshp(\mathcal{T}_G^N, \mathcal{D})$ and 
$\sF'$ is a sheaf, then we need to verify an equivalence
\begin{equation} \label{sheafverify} \hom_{\pshp}(\sF, \sF') \simeq 
\hom_{\pshp}(\sF_{G/N}, \sF'). 
\end{equation}

We observe by \Cref{Kanlemma} that the construction $ \sF \mapsto \sF_{G/N}$ is the right Kan extension 
of $\sF$
from the full subcategory $(\mathcal{T}_{G}^N)' \subset \mathcal{T}_{G}^N$ spanned by those $G$-sets $S$ which
admit a map $S \to G/N$. 
In particular, since $\sF'$ is a sheaf and $\sF' \simeq \sF'_{G/N}$, $\sF'$ is right Kan
extended from this subcategory. 
It follows that
the mapping spaces in \eqref{sheafverify} can be calculated in the
$\infty$-category of presheaves on 
$(\mathcal{T}_{G/N})'$. However, $\sF, \sF_{G/N}$ agree on this subcategory (again
because $\sF_{G/N}$ is a right Kan extension), so the equivalence
\eqref{sheafverify} follows as desired. 
\end{proof}

\begin{proposition} 
Let $\mathcal{D}$ be a presentable $\infty$-category and let $G$ be a profinite
group. 
Suppose that for every open normal subgroup $N \leq G$ and subgroup $K \leq
G/N$, the functor
$(\cdot)^{hK}: \fun(BK, \mathcal{D}) \to \mathcal{D}$ commutes with
filtered colimits. 
Let $\sF \in \pshp(\mathcal{T}_G, \mathcal{D})$ be a product-preserving presheaf
on $\mathcal{T}_G$. 
Then the sheafification $\sF^{sh}$ of $\sF$ is given by the formula
\[ \sF^{sh}(G/H) = \varinjlim \sF( G/H')^{h (H/H')}, \]
as $H' \leq H$ ranges over all open normal subgroups. 
\label{sheafificationbyCech}
\end{proposition} 
\begin{proof} 
For each open normal subgroup $N \leq G$, we consider 
the construction $\sF_{G/N}$ as in \eqref{FGN} and the 
augmentation map $\sF \to \sF_{G/N}$ in $\pshp(\mathcal{T}_G, \mathcal{D})$. 
We claim that the filtered colimit
$\widetilde{\sF} = \varinjlim_N \sF_{G/N}$ (taken over all open normal subgroups of $G$) is the
sheafification of $\sF$. 

To see this, we observe that $\sF_{G/N}$ restricts to a sheaf on 
$\mathcal{T}_G^N$, by \Cref{FGNsheafification}. 
Moreover, our assumptions and the explicit formula for sheafification
\eqref{FGN} implies that the inclusion 
$\mathrm{Sh}(\mathcal{T}_G^N, \mathcal{D}) \subset \pshp(\mathcal{T}_G^N,
\mathcal{D})$ is closed under filtered colimits. 
It follows that the colimit $\varinjlim_N \sF_{G/N}$ is a sheaf on each site
$\mathcal{T}_G^N$, and hence on $\mathcal{T}_G$. It remains to verify that this
colimit is actually the sheafification.
If $\sF' \in \mathrm{Sh}(\mathcal{T}_G, \mathcal{D})$, then we have that (for each
$N$)
$\hom_{\pshp(\mathcal{T}_G^N, \mathcal{D})}(\sF, \sF') \simeq 
\hom_{\pshp(\mathcal{T}_G^N, \mathcal{D})}(\sF_{G/N}, \sF')$ in view of
\Cref{FGNsheafification}. Taking the colimit, we get
$\hom_{\pshp(\mathcal{T}_G, \mathcal{D})}(\sF, \sF') \simeq 
\hom_{\pshp(\mathcal{T}_G, \mathcal{D)}}(\widetilde{\sF}, \sF')$ which is the
desired universal property. 
\end{proof}

The condition that homotopy fixed points for finite group actions in $\mathcal{D}$
should commute with filtered colimits is a strong one. 
Here we note three important cases where this is satisfied: 
\begin{enumerate}
\item If $\mathcal{D}$ is a presentable $\infty$-category where finite limits
and filtered colimits commute, and which is an $n$-category for some $n <
\infty$. Then totalizations (hence homotopy fixed points for finite group
actions) and filtered colimits commute, since totalizations can be computed as
finite limits, cf.~\cite[Lemma 1.3.3.10]{HA} and its proof. 
For instance, one can take the $\infty$-category $\mathcal{S}_{\leq n}$ of
$n$-truncated spaces. 
\item  
If $\mathcal{D}$ is 
the $\infty$-category $\sp_{\leq 0}$ of \emph{coconnective} spectra,
then homotopy fixed points and filtered colimits commute. 
In fact, this can be tested after applying $\tau_{\geq -d}$ for some $d$, i.e.,
in 
$\sp_{[-d, 0]}$ for each $d \geq 0$. 
\item Fix a prime $p$ and $n \geq 0$. If $\mathcal{D}$ is 
the $\infty$-category of $T(n)$-local spectra for $T(n)$ a telescope of height
$n$ (for $n = 0$, we set $T(0) = H\mathbb{Q}$), then Tate spectra for finite
groups vanish \cite{Kuhn}, so homotopy fixed points are identified with homotopy orbits and
thus commute with colimits. 
\end{enumerate}

We will discuss examples of the second two cases now. 
\begin{example}[Continuous group cohomology]
For presheaves with values in coconnective spectra, we can compute
sheafification via the \v{C}ech construction, by \Cref{sheafificationbyCech}. 
Recall also that in this case sheafification and hypersheafification gives the
same answer
since truncated sheaves are hypercomplete.

Let $\mathcal{A}$ be a sheaf of abelian groups on $\mathcal{T}_G$, corresponding
to a continuous $G$-module $M$. 
Let
$K(\mathcal{A},0)$ denote the corresponding sheaf of Eilenberg-Maclane spectra
in degree $0$.  
We obtain
$$K(\mathcal{A},0)(G/H) = \varinjlim_{N\subset H} (M^N)^{hH/N},$$
as the colimit ranges over all normal subgroups $N  \subset H$ and $M^N
\subset M$ denotes the usual $N$-fixed points. 
In particular
$$\pi_\ast \left(K(\mathcal{A},0)(G/G)\right) = \varinjlim_{N\subset G}
H^{-*}(G/N; M^N).$$
This is exactly the continuous $G$-cohomology groups of the corresponding $G$-module in
the usual definition \cite{SerreCG}.  This of course conforms with the general
theory \cite[Cor. 2.1.2.3]{SAG}: both sides compute the derived functors of the functor of sections on $G$ (or $G$-fixed points).
\end{example}

\begin{remark} 
For the next example, we will need the following observation. 
Let $R$ be an $E_1$-ring spectrum, and let $\sF$ be a presheaf in $\md_R$ on a
Grothendieck site $\mathcal{T}$. Then the sheafification of $\sF$ can be
calculated either in $R$-modules or in $\sp$. 
This is an abstract Bousfield localization argument. 
To see this, we note that the map $\sF \to \sF^{sh}$ (from $\sF$ to its
sheafification in $\md_R$) belongs to the strongly
saturated class 
\cite[Sec. 5.5.4]{HTT}
of morphisms generated by $\widetilde{h_t} \otimes M \to h_t
\otimes M$, for each $\widetilde{h_t} \to h_t$ a covering sieve and $M \in
\md_R$. 
Such maps induce equivalences upon sheafification, either for sheaves of $\md_R$
or for sheaves of spectra. Since the forgetful functor from $\md_R$ to $\sp$
preserves limits and colimits, the claim follows. For a similar argument, see
also 
\Cref{sheafificationintwocases} below. 
\end{remark}

\begin{example}[A non-hypercomplete example]
\label{nonhypcompleteZp}
Next we consider an example where sheafification and hypersheafification differ. 
Fix a prime number $p$ and an integer $n \geq 0$, and consider the Morava
$K$-theory spectrum $K(n)$. 
We can use \Cref{sheafificationbyCech} to compute the constant
sheaf $\underline{K(n)}$ on $\mathcal{T}_{\mathbb{Z}_p}$, i.e., the
sheafification of the constant presheaf $\mathcal{O}(\mathbb{Z}_p)^{op} \to \sp$ given by
$K(n)$ (on the orbit category, and then extended to
$\mathcal{T}_{\mathbb{Z}_p}^{op}$ by forcing it to preserve finite products).
Note that this can either be computed as a sheaf of spectra or in sheaves of
$K(n)$-modules; since Tate spectra of finite groups vanish in the latter, we can
apply \Cref{sheafificationbyCech}. 

We obtain $$\underline{K(n)}(\ast) = \varinjlim_d K(n)^{B\mathbb{Z}/p^d\mathbb{Z}},$$
where the transition maps are induced by the quotient maps $\mathbb{Z}/p^{d+1}\mathbb{Z}\rightarrow\mathbb{Z}/p^d\mathbb{Z}$.

When $n=0$ we have $K(n)=H\mathbb{Q}$, and the colimit is constant and just produces $H\mathbb{Q}$ again.  In fact the constant sheaf $\underline{H\mathbb{Q}}$ is just the constant presheaf $H\mathbb{Q}$.  

When $n\geq 1$, let $\mathscr{G}$ be the height $n$ formal group over the
perfect field $k$ of characteristic $p$ used to define $K(n)$.  Then the above
colimit instead gives an even periodic ring spectrum whose $\pi_0$ identifies
with the ring of functions on the affine scheme $T_p(\mathscr{G})=\varprojlim_d
\mathscr{G}[p^d]$ over $k$.  Thus when we take
$\mathscr{G}=\widehat{\mathbb{G}_m}$ over $k=\mathbb{F}_p$, we have
$$\pi_0 \left(\underline{K(1)}(\ast)\right) = \mathbb{F}_p[\mathbb{Q}_p/\mathbb{Z}_p].$$
The value of each $\underline{K(n)}(\mathbb{Z}_p/p^d\mathbb{Z}_p)$ is abstractly
the same, but the restriction map
$\underline{K(n)}(\mathbb{Z}_p/p^{d-1}\mathbb{Z}_p)\rightarrow
\underline{K(n)}(\mathbb{Z}_p/p^d\mathbb{Z}_p)$ corresponds on $\pi_0$ to
pullback along the multiplication by $p$ map $T_p(\mathscr{G})\rightarrow
T_p(\mathscr{G})$.  The
$\mathbb{Z}/p^d\mathbb{Z}$-action on
$\underline{K(n)}(\mathbb{Z}_p/p^d\mathbb{Z}_p)$ is of course trivial.

This provides an example of a sheaf of spectra on $\mathcal{T}_{\mathbb{Z}_p}$
which is not Postnikov complete (or hypercomplete, which is the same by
\Cref{hypforfinitecd}).  Indeed, the homotopy groups of the global sections of
its Postnikov completion can be calculated using the spectral sequence from
\Cref{descentss}:
$$E^2_{i,j} = H^{-i}_{cont}(\mathbb{Z}_p;\pi_j K(n))\Rightarrow \pi_{i+j}
\left(K(n)^h(\ast)\right).$$
The homotopy of $K(n)$ is $k$ in every even degree and $0$ in every odd degree,
so this gives $k$ in every degree for $\pi_\ast\left(K(n)^h(\ast)\right)$ ---
clearly different from the above description of $\pi_\ast\left(
\underline{K(n)}(\ast)\right)$.  Note that 
\cite[Warning 7.2.2.31]{HTT}
gave an example (due to Ben Wieland) of a sheaf of \emph{spaces} on $\mathcal{T}_{\mathbb{Z}_p}$ which is not hypercomplete; here we we even have a sheaf of \emph{spectra} which is not hypercomplete.  It yields another example of sheaf of spaces which is not hypercomplete, simply by taking $\Omega^\infty$.
\end{example}

Next we give a criterion for a sheaf of spectra on $\mathcal{T}_G$
to be Postnikov complete in the case where $G$ has finite cohomological
dimension, strengthening \Cref{Postcompletecrit} to an if and only if
assertion.

\begin{proposition} 
\label{Postcompletecrit2}
Let $G$ be a profinite group of finite $\mathcal{P}$-local cohomological dimension $d$. 
Then the following are equivalent for a $\mathcal{P}$-local sheaf $\sF$ of
spectra on $\mathcal{T}_G$: 
\begin{enumerate}
\item $\sF$ is hypercomplete (equivalently, Postnikov complete).  
\item There exists an integer $M$ such that 
for all normal inclusions $N \subset H$ of open subgroups, the 
$H/N$-action on $\sF(G/N)$ is weakly $M$-nilpotent. 
\item 
For all normal inclusions $N \subset H$ of open subgroups, the 
$H/N$-action on $\sF(G/N)$ is weakly $d$-nilpotent. 
\end{enumerate}
\end{proposition} 
\begin{proof} 
Combine \Cref{Postcompletecrit} and  \Cref{nilpandhyp}. 
\end{proof}

\begin{proposition} 
\label{hypissmashingprofinitegroup2}
Let $G$ be a profinite group of finite virtual $\mathcal{P}$-local cohomological dimension $d$. 
Then hypercompletion is smashing on  $\mathcal{P}$-local sheaves $\sF$ of
spectra on $\mathcal{T}_G$. 
\end{proposition}
\begin{proof} 
Without loss of generality, $G$ has finite 
$\mathcal{P}$-local cohomological dimension. 
Since the cohomology of profinite groups commutes with filtered
colimits, \Cref{critforhypcolimits} implies that hypercompleteness is preserved
under colimits, so all that remains is to check that it is also preserved under
tensoring with an arbitary sheaf $\mathcal{F}$.  Resolving $\mathcal{F}$ by
representables and again using preservation of hypercompleteness under colimits,
we can reduce to the case where $\mathcal{F}=h_{G/H}$ is the representable
associated to some  $G/H$ in $\mathcal{T}_G$.  But
$$\mathcal{G}\otimes h_{G/H} = x_!x^\ast\mathcal{G},$$
where $x^\ast$ the pullback functor from sheaves of spectra on $\mathcal{T}_G$
to sheaves of spectra on $\mathcal{T}_H$ and $x_!$ is the left adjoint to $x^*$.
Thus it suffices to show that $x_!$ preserves hypercompleteness.  This can be
checked after pullback along any covering map by \Cref{locglobhyp}; but after such a pullback $x$ can be split, and so $x_!$ is just a finite direct sum and hence certainly preserves hypercompleteness.
\end{proof}

In the remainder of this subsection, we will give a precise description of the
hypercomplete sheaves, which will give a more explicit proof of 
\Cref{hypissmashingprofinitegroup2}. 
First we recall some material from \cite{MNN17, MNN15}, though for convenience we abbreviate and reindex the terminology, so that instead of saying ``nilpotent of exponent $\leq d+1$ with respect to the trivial family'' we say ``$d$-nilpotent".

\begin{definition}
Let $\mathcal{C}$ be a stable $\infty$-category. 
Let $G$ be a finite group and $X\in\operatorname{Fun}(BG,\mathcal{C})$ an
object with $G$-action.  Then we say $X$ is:
\begin{enumerate}
\item \emph{$0$-nilpotent} if it is a retract of an induced object $\oplus_G Y$,
some $Y\in\mathcal{C}$.
\item \emph{$d$-nilpotent} if it is a retract of an extension a $0$-nilpotent
object by an $(d-1)$-nilpotent object.
\item \emph{nilpotent} if it is $d$-nilpotent for some $d$. 
\end{enumerate}

More generally, given a family $\mathscr{F}$ of subgroups of $G$ (i.e., $\mathscr{F}$ is closed
under subconjugation), we can similarly define a notion of $\mathscr{F}$-nilpotence: $X$ is 
$\mathscr{F}$-nilpotent if it belongs to the thick subcategory of $\fun(BG,
\mathcal{C})$ generated by objects induced from subgroups in $\mathscr{F}$.\footnote{One
can also define exponents in this context, but we will not need them here.}
\end{definition}

We now prove some basic properties of the notion of nilpotence. 
To begin with, we will take $\mathcal{C} = \mathrm{Sp}$; we will observe below
(\Cref{nilpendocrit}) that the general case can be reduced to this. 

\begin{lemma}\label{basicnilplemma}
\begin{enumerate}
\item If $X$ is a $d$-nilpotent object of $\fun(BG, \sp)$, it is also
$d$-nilpotent viewed as an object of $\fun(BH, \sp)$ for all subgroups $H\subset G$.
\item For $d\geq 0$, the collection of $d$-nilpotent objects of $\fun(BG, \sp)$ is closed under shifts, retracts, and tensoring with any $Z\in\operatorname{Fun}(BG,\operatorname{Sp})$.
\item If $A$ is a $d$-nilpotent spectrum with $G$-action with an algebra structure and $M$ is a module over $A$ in $G$-spectra, then $M$ is $d$-nilpotent.
\item If $X \in \fun(BG, \sp)$ is $d$-nilpotent and $N \leq G$ is a normal
subgroup, then $X^{hN} \in \fun(B (G/N), \sp)$ is $d$-nilpotent. 
\end{enumerate}
\label{general:dnilpotentobj}
\end{lemma}
\begin{proof}
Claim 1 is clear by induction, the base case being the observation that $G$ is
free as an $H$-set.  For claim 2, the case of shifts and retracts is clear.  For
tensoring, an induction reduces us to the case $d=0$, where the claim follows
from the usual remark that the spectrum $\oplus_G Z$ with $G$ acting both on $G$
and $Z$ is equivalent to $\oplus_G Z$ with $G$ acting only on $Z$.  Claim 3
follows from claim 2, since $M$ is then a retract of $A\otimes M$. Claim 4 is
easy to check by induction, starting with $d = 0$. 
\end{proof}

For future reference we note also the following lemma. 
\begin{lemma} 
Let $G$ be a finite group, and let $X \in \fun(BG, \sp)$. 
For each $n$, 
let $\mathrm{sk}_n(EG)$ denote the $n$-skeleton of the standard simplicial model for
$EG$. Then $X$ is $d$-nilpotent if and only if the map 
\[ \Sigma^\infty_+ \mathrm{sk}_d(EG) \otimes X \to X   \]
admits a section in $\fun(BG, \sp)$. 
\label{nilpskEG}
\end{lemma} 
\begin{proof} 
Since
$\mathrm{sk}_d(EG)$ is a $d$-dimensional $G$-CW complex all of whose cells are
indexed by free $G$-sets, it follows that
$\Sigma^\infty_+\mathrm{sk}_d(EG) \otimes X \in \fun(BG, \sp)$ is always $d$-nilpotent. 

Conversely, suppose that $X$ is $d$-nilpotent. 
Let $A = F(G_+, \mathbb{S}) \in \fun(BG, \sp)$; this is a commutative algebra
object, and we can form the cobar construction
$\mathrm{CB}^\bullet(A)$, a cosimplicial object in spectra. 
As in \cite[Prop. 4.9]{MNN17} the $d$-nilpotence means that 
the map $X \to \mathrm{Tot}_{d}(X \otimes \mathrm{CB}^\bullet(A))$ admits 
a splitting. However, this is equivalent to the map $X \to
F(\Sigma^\infty_+\mathrm{sk}_dEG, X)$;
adjointing over gives the section desired. 
\end{proof}

Next, we give the promised reduction of nilpotence in any stable
$\infty$-category to 
the case of $\mathrm{Sp}$; this will be used in later sections. 
\begin{proposition} 
\label{nilpendocrit}
Let $\mathcal{C}$ be a stable $\infty$-category and let $X \in \fun(BG,
\mathcal{C})$. Then for any $d\geq 0$
the following are equivalent: 
\begin{enumerate}
\item  $X$ is $d$-nilpotent. 
\item $\mathrm{End}_{\mathcal{C}}(X) \in \fun(B(G \times G), \sp)$ is
$d$-nilpotent. 
\item 
$\mathrm{End}_{\mathcal{C}}(X) \in \fun(BG, \sp)$ is $d$-nilpotent (here the
$G$-action is the diagonal one). 
\end{enumerate}
\end{proposition} 
\begin{proof} 
The fact that 1 implies 2 and 3 follows from a thick subcategory argument, so
we prove that 3 (which is implied by 2) implies 1. 
Suppose $\mathrm{End}_{\mathcal{C}}(X) \in \fun(BG, \sp)$ is $d$-nilpotent. 
For each $n$, let $\mathrm{sk}_n (EG)$ be the $n$-skeleton of $EG$; we have a
map 
\( \mathrm{sk}_n (EG)_+ \otimes X \to X .  \)
Note that 
$\hom_{\fun(BG, \mathcal{C})}(X, X) \simeq \mathrm{End}_{\mathcal{C}}(X)^{hG}$
while 
$$\hom_{\fun(BG, \mathcal{C})}(X, \mathrm{sk}_n(EG)_+ \otimes X)
\simeq (\sk_n(EG)_+ \otimes \mathrm{End}_{\mathcal{C}}(X, X))^{hG}
.$$
By assumption, $\mathrm{End}_{\mathcal{C}}(X, X)$ is $d$-nilpotent; therefore, 
by \Cref{nilpskEG}, the map 
\[ ( \mathrm{sk}_d(EG)_+ \otimes \mathrm{End}_{\mathcal{C}}(X, X))^{hG} \to 
(  \mathrm{End}_{\mathcal{C}}(X, X))^{hG} 
\]
has image including the identity.
Unwinding the above, it follows that there exists a map $f: X \to
\sk_d(EG)_+\otimes
X$ in $\fun(BG, \mathcal{C})$ such that the composite
$$X \stackrel{f}{\to} \sk_d(EG)_+ \otimes X \to X $$
is the identity. 
This in particular implies that $X$ is $d$-nilpotent as desired. 
\end{proof}

We now observe that for algebra objects in $\mathrm{Sp}$ with $G$-action,
nilpotence can be tested on the Tate construction. 
Note that the Tate construction $(-)^{tG}$ vanishes on nilpotent objects of
$\mathrm{Fun}(BG, \mathrm{Sp})$, but the converse is false in general.

\begin{lemma}
\label{nilpexpbounded}
Let $G$ be a finite group and $R$ an algebra object of $\operatorname{Fun}(BG,\operatorname{Sp})$.  
Suppose $R^{tG} = 0$.  
Then: 
\begin{enumerate}
\item $R$ is nilpotent.  
\item
If the underlying spectrum of $R$ belongs to $\sp_{\geq -d}$, then $R$ is
$d$-nilpotent. 
\end{enumerate}
\end{lemma}
\begin{proof}
As before, choose the standard simplicial model $EG_\bullet$, which is a simplicial free
$G$-set such that $|EG_\bullet|$ is contractible. 
This gives  an ascending filtration sequence of spaces $\{\mathrm{sk}_i
EG\}_{i \geq 0}$.   
For each $i$, we consider the 
map $\Sigma^\infty_+ \mathrm{sk}_i (EG) \otimes R \to R$ in $\fun(BG, \sp)$. 
Note that the norm map from homotopy orbits to homotopy fixed points is an equivalence on 
$\Sigma^\infty_+ \mathrm{sk}_i (EG) \otimes R$ since it is finitely built from
induced objects of $\fun(BG, \sp)$. 
Our assumption that $R^{tG} = 0$ thus implies that the natural maps induce an
equivalence 
\begin{equation} \label{normEG}
R_{hG} = 
\varinjlim_i 
(\Sigma^\infty_+ \mathrm{sk}_i (EG) \otimes R)_{hG} 
\simeq
\varinjlim_i 
(\Sigma^\infty_+ \mathrm{sk}_i (EG) \otimes R)^{hG} 
\simeq R^{hG}.
\end{equation}

Suppose we are the situation of 2. 
Since $R$ belongs to $\sp_{\geq -d}$, it follows that the unit in $R^{hG}$ belongs to
the image of the map 
$(\Sigma^\infty_+ \mathrm{sk}_d (EG) \otimes R)^{hG} \to R^{hG}$, via
\eqref{normEG}, since both homotopy fixed points are identified with homotopy
orbits. Unwinding the definitions, it follows that we can find maps in
$\fun(BG, \sp)$,
\[ R \to \Sigma^\infty_+ \mathrm{sk}_d(EG) \otimes R \to R  \]
such that the composite is the identity. 
Thus $R$ is $d$-nilpotent by \Cref{nilpskEG}. 
The situation of 1 is analogous but we do not have a specific $d$; one just
chooses $d$ large enough such that the map $(\Sigma^\infty_+ \mathrm{sk}_d(EG)
\otimes R)^{hG} \to R^{hG}$ has image including the unit. 
\end{proof}
\begin{proposition} 
Let $G$ be a finite group. Let $R \in \mathrm{Alg}( \fun(BG,
\md_{H\mathbb{Z}}))$ be an $H\mathbb{Z}$-algebra equipped with a $G$-action. 
Suppose that, for each $H \leq G$, the underlying spectrum of $R^{hH}$ belongs
to $\sp_{\geq -d}$ for some $d \geq 0$. 
Then $R$ is $d$-nilpotent. 
\label{Gnilpcriterion}
\end{proposition} 
\begin{proof} 
By \Cref{nilpexpbounded}, it suffices to show that $R \in \fun(BG, \sp)$ is nilpotent. 
We use the more general theory of nilpotence with respect to a family of
subgroups here, cf.\ \cite{MNN17, MNN15}. By induction, for every proper subgroup
$H < G$, $R|_H \in \fun(BH, \sp)$ is 
$H$-nilpotent. It thus suffices to show that $R$ is nilpotent for the family of
proper subgroups. 
Let $\widetilde{\rho_G}$ be the reduced complex regular representation of $G$,
and let $S^{\widetilde{\rho_G}}$ be the one-point compactification of $G$,
considered as a spectrum equipped with a $G$-action. 
We have the Euler class $e: \mathbb{S} \to S^{\widetilde{\rho_G}}$. 
We recall from \cite{MNN15} that 
an object $M \in \fun(BG, \sp)$
is nilpotent for the family of proper subgroups if there exists 
$n \gg 0$ such that the map 
\[ \mathrm{id}_M \otimes e^n: M \to M \otimes S^{n\widetilde{ \rho_G}}  \]
is nullhomotopic. 
Since we are working over $H\mathbb{Z}$, we observe that 
there is an equivalence $H \mathbb{Z} \otimes S^{n \widetilde{\rho_G}} \simeq H
\mathbb{Z} [ 2n (|G|  - 1)]$ in $\fun(BG, \md_{H\mathbb{Z}})$. 
In our setting, we thus need to show that for $n \gg 0$, the map 
\[ \mathrm{id}_R \otimes e^n: R \to R \otimes S^{n \widetilde{\rho_G}}  \simeq
R \otimes_{H\mathbb{Z}} (H \mathbb{Z} \otimes S^{n \widetilde{\rho_G}} )
\simeq
\Sigma^{2n ( |G| - 1)} R\]
is nullhomotopic in $\md(BG, \sp)$. Since this is an $R$-module map, it is
classified by an element in $\pi_{-2n(|G| - 1)} (R^{hG})$. 
However, we have assumed that this group vanishes for $n \gg 0$. This verifies
that $R$ is $G$-nilpotent as desired. 
\end{proof}

\begin{remark} 
\Cref{Gnilpcriterion} fails if $R$ is only assumed to belong to $\mathrm{Alg}(
\fun(BG, \sp))$. For instance, it fails for the trivial $G$-action on the sphere
spectrum; note that the homotopy fixed points are all connective as a
consequence of the Segal
conjecture, cf.\ \cite{Car84}. 
\end{remark}

We now prove a key nilpotence assertion which refines \Cref{nilpandhyp} in the case
of finite group actions (the previous result gives the nilpotence of certain
cosimplicial digrams, which is implied by nilpotence in the group action sense). Note that this result is essentially due to
Tate-Thomason, originally in \cite[Annex 1, Ch. 1]{SerreCG} in the case of
profinite group cohomology and \cite[Remark 2.28]{Th85} for sheaves of spectra
for profinite groups. See also \cite[Sec. 5.3]{Mi97} for a treatment.   

\begin{theorem}[Tate-Thomason]
\label{nilpotenceofetalePostnikov}
Let $\mathcal{X}$ be an $\infty$-topos, and let $\mathcal{A}$ be 
an algebra object of $\mathrm{Sh}(\mathcal{X}, \sp)$ which is Postnikov complete. Suppose given a finite group $G$ and a
$G$-Galois cover $Y\rightarrow X$ in $\mathcal{X}$, such that the cohomological
dimension of $Y/H$ with $\pi_n\mathcal{A}$-coefficients is bounded by some fixed $d \geq 0$ for all $H\subset G$ and $n\geq 0$.

Then any module over $\mathcal{A}(Y) \in \mathrm{Alg}(\fun(BG, \sp)) $ is $d$-nilpotent. 
\end{theorem}
\begin{proof}
By replacing $\mathcal{A}$ by its connective cover if necessary,
we may assume that $\mathcal{A}$ is connective. By \Cref{nilpexpbounded} and
part 3 of \Cref{basicnilplemma}, it suffices to verify that $\mathcal{A}(Y)^{tG} = 0$. 
For each $a\leq b$, we let $\mathcal{A}_{[a,b]}$ be the 
$[a,b]$-Postnikov section of $\mathcal{A}$. 
Note that $(\mathcal{A}_{[0,0]})(X)$ 
is, by assumption on the cohomological dimension, an object of $\sp_{\geq -d}$. 
In addition, $\mathcal{A}_{[0,0]}$ is an $H\mathbb{Z}$-module. 
Applying \Cref{Gnilpcriterion} to the $G$-action on $(\mathcal{A}_{[0,0]})(Y)$, we conclude
that this $G$-action is nilpotent. 
Since each $\mathcal{A}_{[n, n]}(Y)$ is a module in $\fun(BG, \sp)$ over
$\mathcal{A}_{[0, 0]}(Y)$, it follows by induction that $\mathcal{A}_{[0,
n]}(Y)$ is nilpotent and $\mathcal{A}_{[0, n]}(Y)^{tG} = 0$.
We claim now that we can pass to the limit as $n \to \infty$ to conclude that 
$\mathcal{A}(Y)^{tG} =0$. 
To this end, we observe simply that
\begin{equation} \label{Apostlimit} \mathcal{A}(Y) \simeq \varprojlim_n
\mathcal{A}_{[0, n]}(Y)  \end{equation}
because $\mathcal{A}$ is Postnikov complete, and that the assumption on finite
cohomological dimension implies that the homotopy fibers of 
$\mathcal{A}(Y) \to \mathcal{A}_{[0, n]}(Y)$ become arbitrarily highly connected
as $n \to \infty$. Therefore, we can interchange the limit in \eqref{Apostlimit}
with both homotopy fixed points and homotopy orbits, and hence with Tate
constructions; we conclude that $\mathcal{A}(Y)^{tG} = 0$ as desired. 
\end{proof}

Putting things together, we get the main result of this subsection.

\begin{theorem}
\label{Gcompletecrit}
Let $G$ be a profinite group of finite $\mathcal{P}$-local cohomological dimension $d$, and let
$\mathcal{F}$ be a sheaf of $\mathcal{P}$-local spectra on $\mathcal{T}_G$.  The following are equivalent:
\begin{enumerate}
\item $\mathcal{F}$ is hypercomplete.
\item $\mathcal{F}$ is Postnikov complete.
\item For every normal containment $N\subset H$ of open subgroups of $G$, the
spectrum with 
$H/N$-action
$\mathcal{F}(G/N)$ is $d$-nilpotent.
\item There is a $d'\geq 0$ such that for every normal  open subgroup $N\subset
G$, the spectrum 
with $G/N$-ction
$\mathcal{F}(G/N)$ is $d'$-nilpotent.
\end{enumerate}
\end{theorem}
\begin{proof}
1 $\Leftrightarrow$ 2 by \Cref{hypforfinitecd}.  2 $\Rightarrow$ 3 by
\Cref{nilpotenceofetalePostnikov} applied to the hypercompletion of the unit
object.  3 $\Rightarrow$ 4 trivially. Finally, 4 $\Rightarrow$ 2 by
\Cref{Postcompletecrit2} (see also part 4 of \Cref{general:dnilpotentobj}).\end{proof}

\begin{remark}
These conditions imply that the $H/N$-Tate construction vanishes on
$\mathcal{F}(G/N)$ for all normal containments $N\subset H$.  The converse is
false, as \Cref{nonhypcompleteZp} verifies.
\end{remark}

Now we use the nilpotence machinery above to reprove that hypercompletion is
smashing 
under finite cohomological dimension assumptions (which appeared earlier as
\Cref{hypissmashingprofinitegroup2} via an abstract argument). 
\begin{corollary}
Let $G$ be a profinite group of virtual finite $\mathcal{P}$-local cohomological
dimension.  Then hypercompletion is smashing for sheaves of
$\mathcal{P}$-local spectra on $\mathcal{T}_G$, and agrees with Postnikov completion.
\end{corollary}
\begin{proof}
Without loss of generality (by passing to a finite index subgroup, using the
local-global principle \Cref{locglobhyp}), we may
assume 
that $G$ has finite $\mathcal{P}$-local cohomological dimension, because all
these assertions are local. 
Postnikov completion agrees with hypercompletion by \Cref{hypforfinitecd}.   The smashing claim follows from 
the criterion of 
\Cref{Gcompletecrit} 
via nilpotence. Indeed, here we use the fact that 
for finite group $H$, $d$-nilpotent objects of $\fun(BH, \sp)$ form an ideal, as
in 
\Cref{general:dnilpotentobj}. 
In particular, if $\mathcal{A}$ is a $\mathcal{P}$-local sheaf of algebras which is hypercomplete,
then any sheaf of modules $\mathcal{M}$
over it 
is necessarily hypercomplete, which implies that hypercompletion is smashing 
if $G$ has finite $\mathcal{P}$-local cohomological dimension. 
\end{proof}

\begin{corollary}
Let $G$ be a profinite group of virtual finite $\mathcal{P}$-local cohomological dimension.  Then
any sheaf of $H\mathbb{Z}_{\mathcal{P}}$-modules on $\mathcal{T}_G$ is hypercomplete.
\end{corollary}
\begin{proof}
The sheafification 
$\underline{H\mathbb{Z}_{\mathcal{P}}}$ of
the constant sheaf 
$H\mathbb{Z}_{\mathcal{P}}$
is truncated, hence hypercomplete.
Alternatively this follows from \Cref{Zmodulesarehypercomplete}. 
\end{proof}

\subsection{The \'etale topos of an algebraic space}

The main result of this subsection is essentially due to Thomason.  It is an instance of the intuition that ``\'etale = Nisnevich + Galois", or more picturesquely that ``the \'etale topos fibers over the Nisnevich topos with fibers $BG_{k(x)}$".  Before stating the precise result, we give some recollections on pullback functoriality for \'etale and Nisnevich sheaves.
In particular, we review how a Nisnevich sheaf on a qcqs algebraic space $X$
induces a Nisnevich sheaf on each of its residue fields via pullback. 

In \Cref{def:Nisandetale}, we reviewed the Nisnevich and \'etale sites of a 
qcqs algebraic space. 
Let $X$ be a qcqs algebraic space and let $\sF$ be a presheaf of spectra on $\operatorname{et}_X$. 
For every \'etale map $Y \to X$ of qcqs algebraic spaces, we can evaluate
$\sF(Y)$; we wish to extend this to pro-\'etale $Y \to X$. 
For simplicity and ease of notation, we will assume that $X =\spec(A)$ is
affine. Recall that qcqs algebraic spaces are Nisnevich-locally affine, so this
is no loss of generality, cf.\ \Cref{ifNisreducetoaffine}. 

\begin{construction}[Extending Nisnevich sheaves to ind-\'etale objects]
Let $\sF$ be a Nisnevich sheaf (of spaces or spectra) on $X= \spec(A)$. Then $\sF$ defines a functor on the category of \'etale $A$-algebras. 
By left Kan extension, $\sF$ extends to a functor on the category of ind-\'etale
$A$-algebras which commutes with filtered colimits. Given an ind-\'etale $A$-algebra $B$, it follows that $\sF$
(extended in this manner) also
defines a Nisnevich sheaf on $\spec(B)$, since the Nisnevich excision condition is
finitary. 
\end{construction}

This construction of evaluating sheaves on ind-\'etale objects will be extremely
useful in this subsection. For instance, we recall that evaluation on
henselian (resp. strictly henselian) local rings recovers the points
of the Nisnevich and \'etale topoi. Compare \cite[Tag 04GE]{stacks-project} for
a convenient reference on henselian local rings. 

\begin{example}[Points of the Nisnevich topos]
\label{points:Nisnevich}
Let $B$ be an \'etale $A$-algebra and $\mathfrak{p} \in \spec(B)$. 
Then the henselization $B^h_{\mathfrak{p}}$ of $B$ at $\mathfrak{p}$ is an
ind-\'etale $A$-algebra. The evaluation 
$\sF \mapsto \sF(B^h_{\mathfrak{p}})$ gives the points of the $\infty$-topos of
Nisnevich sheaves on $\spec(A)$.
In particular, for sheaves of spaces, the functor 
$$ \sh( \spec(A)_{Nis}) \to \mathcal{S}, \quad \sF \to
\sF(B^h_{\mathfrak{p}})$$
commutes with all colimits and with finite limits, and similarly for sheaves of
spectra. 
Compare \cite[Sec. B.4.4]{SAG}. More precisely, the $\infty$-category of points
of the $\infty$-topos of sheaves of spaces on 
the Nisnevich site of $\spec(A)$ is equivalent to the category of ind-\'etale
$A$-algebras which are henselian local. 
\end{example} 

\begin{example}[Points of the \'etale topos] 
\label{points:etale}
Let $C$ an ind-\'etale $A$-algebra which is \emph{strictly} henselian local. 
Then the functor
\[ \sh( \spec(A)_{et}, \mathcal{S}) \to \mathcal{S}  \] given by evaluation on
$C$ defines a point of  
$\sh( \spec(A)_{et}, \mathcal{S})$. 
This follows similarly as in the Nisnevich case via \cite[Prop. 6.1.5.2]{HTT} as
well as \cite[Cor. B.3.5.4]{SAG} (and of course is classical for sheaves of
sets).  

Let $\sF$ be an \'etale sheaf on $\spec(A)$, and extend to ind-\'etale
$A$-algebras as above. 
By contrast with the Nisnevich case, $\sF$ need not define an \'etale sheaf on
$\spec(B)$ for $A \to B$ ind-\'etale; the filtered colimits involved need not
commute with totalizations. 
\end{example} 

For future reference, we need a slight upgrade of \Cref{points:Nisnevich} above. 
As above, we fix a commutative ring $A$. 

\begin{construction}[Pulling back to the Nisnevich site of a point]
\label{pullbacktopoint}
Let $\mathfrak{p} \in \spec(A)$ and let $A_{\mathfrak{p}}^h$ be the
henselization of $A$ at $\mathfrak{p}$. 
The category of finite \'etale 
$A_{\mathfrak{p}}^h$-algebras is equivalent to the category of \'etale
$k(\mathfrak{p})$-algebras. 
Given a Nisnevich sheaf on $\spec(A)$, it follows that we obtain a Nisnevich
sheaf on $\spec(k(\mathfrak{p}))$ 
by evaluating on finite \'etale $A_{\mathfrak{p}}^h$-algebras. 
In particular, we obtain a product-preserving presheaf on the category of finite continuous
$\mathrm{Gal}(k(\mathfrak{p}))$-sets. 
In view of \Cref{points:Nisnevich}, for sheaves of spaces or spectra, this is a
left exact, cocontinuous functor. 
\end{construction} 

We can also phrase 
\Cref{pullbacktopoint}
more abstractly and generally in the language of pullbacks of topoi. 
A map of qcqs algebraic spaces $f:Y\rightarrow X$ induces a pullback functor ${et}_X\rightarrow {et}_Y$ which gives a morphism of sites both on the \'etale site and on the Nisnevich site.  Thus there are associated pullback morphisms of $\infty$-topoi $f_{et}^\ast$ and $f_{Nis}^\ast$ both on \'etale sheaves and on Nisnevich sheaves. 
We will only consider 
the Nisnevich pullback, so we will drop the subscript when referring to it:
$f^\ast := f^\ast_{Nis}.$

In particular, suppose $x\rightarrow X$ is a residue field of $X$.  By abuse, we denote the map
$x\rightarrow X$ also by $x$, so we obtain a left-exact, cocontinuous functor
$\sh(X_{Nis}) \to \sh(x_{Nis})$. 
By \Cref{Nisfield}, we can view $x^\ast$ as a functor from sheaves on $X_{Nis}$
to product-preserving presheaves on $\mathcal{T}_x$.  
Unwinding the definitions, one finds
$$(x^\ast\mathcal{F})(y\rightarrow x) = \varinjlim_{y\rightarrow U\rightarrow X}\mathcal{F}(U\rightarrow X)$$
where the indexing category is all \emph{\'etale neighborhoods of $y$ in $X$},
i.e.\ factorizations of the composition $y\rightarrow x\rightarrow X$ through a
map $U\rightarrow X$ in $et_X$.  Note that this category is filtered, because
$et_X$ has pullbacks.
By construction of the henselization \cite[Tag 0BSK]{stacks-project}, 
it follows that this recovers precisely \Cref{pullbacktopoint}. 

Next we include some general preliminaries which enable us to reduce from
qcqs algebraic spaces to affine schemes (via Nisnevich descent). 
\begin{proposition}
\label{ifNisreducetoaffine}
Let $X$ be a qcqs algebraic space and $\mathcal{A}$ a
Nisnevich sheaf of spectra on $X$.  If any of the following is known for the restriction
$p^\ast \mathcal{A}$ of $\mathcal{A}$ to every $p:U\rightarrow X$ in $et_X$ with $U$ affine:
\begin{enumerate}
\item $p^\ast\mathcal{A}$ is a sheaf on the affine \'etale site of $U$;
\item $p^\ast\mathcal{A}$ is a hypercomplete sheaf on the affine \'etale site of $U$;
\item $p^\ast\mathcal{A}$ is a Postnikov complete sheaf on the affine \'etale site of $U$;
\item $p^\ast\mathcal{A}$ is zero as a presheaf on the affine \'etale site of $U$.
\end{enumerate}
then the analogous satement holds for $\mathcal{A}$ itself on the \'etale site of $X$.
\end{proposition}
\begin{proof}
Every qcqs algebraic space is Nisnevich-locally affine (cf.\ \cite{Rydh15},
\cite[Theorem 3.4.2.1]{SAG}), which gives the statements if we replace every
occurrence of ``affine \'etale site" with ``\'etale site."  On the other hand,
the affine \'etale site over an affine scheme is a defining 1-categorical site
closed under fiber products just as the \'etale site is, and they clearly
generate the same topos, hence the same $\infty$-topos since 
everything involved is 1-localic (and determined by the 0-truncated objects), cf.~\cite[Sec.
6.4.5]{HTT}.  As the notions of sheaf, hypersheaf, and Postnikov complete sheaf are intrinsic to the $\infty$-topos, this shows that the replacement doesn't affect the content of the statement.\end{proof}
\begin{remark}[Reduction to finite Galois descent]
\label{redtofintieGalois}
By \cite[Theorem B.6.4.1]{SAG}, in the affine case \'etale descent is equivalent to
the combination of Nisnevich descent and finite \'etale (or Galois) descent.  Thus in the situation of this lemma, 1 holds (and hence $\mathcal{A}$ is an \'etale sheaf) if and only if $\mathcal{A}$ satisfies finite Galois descent on affines.
\end{remark}

Now we state the main results of this section.  We emphasize that they
should essentially be attributed to Thomason (compare \cite[Sec. 11]{TT90}).  Our contribution is to have rewritten Thomason's argument in more modern language.

\begin{theorem}
\label{hypetalehypNis}
Let $X$ be a qcqs algebraic space and $\mathcal{F}$ a presheaf of
$\mathcal{P}$-local spectra on
$et_X$. 
Suppose that there is a uniform bound on the $\mathcal{P}$-local \'etale cohomological dimension of
each \'etale $U \to X$ (or for a cofinal collection, e.g., the affines). 
Then $\mathcal{F}$ is a hypercomplete \'etale sheaf if and only if it is a
hypercomplete Nisnevich sheaf, and the presheaf of spectra $x^\ast\mathcal{F}$
on $\mathcal{T}_{x}=\mathcal{T}_{\operatorname{Gal}_{k(x)^{sep}/k(x)}}$ is a
hypercomplete  \'etale sheaf for all points $x\in X$.
\end{theorem}
\begin{remark}
Suppose $X$ is a qcqs algebraic space.  Then by \Cref{etalecohdim},
$$\operatorname{sup}_{x\in
X}\operatorname{CohDim}(\mathrm{Gal}_{k(x)})\leq\operatorname{sup}_{U\rightarrow X \in
et_X}\operatorname{CohDim}(U_{et})\leq  \operatorname{KrullDim}(X)
+\operatorname{sup}_{x\in X}\operatorname{CohDim}(\mathrm{Gal}_{k(x)}).$$
In particular, if $X$ has finite Krull dimension, then $X$ satisfies the
hypotheses of the theorem if and only if there is a uniform bound on the
$\mathcal{P}$-local Galois cohomological dimension of the residue fields of $X$.
\end{remark}

\begin{proof} Without loss of generality (cf.\ \Cref{locglobhyp} and
\Cref{ifNisreducetoaffine}), $X = \spec(A)$ is affine.  
Let $N$ be the bound on the $\mathcal{P}$-local \'etale cohomological dimension. 
First suppose that $\sF$ is a hypercomplete \'etale sheaf. 
We claim that $x^* \sF$ is a hypercomplete  \'etale sheaf on $\spec ( k(x))$. 
Indeed, we extend $\sF$ to all ind-\'etale $A$-algebras. 
For each faithfully flat \'etale map of \'etale $A$-algebras $B_1 \to B_2$, we
have by assumption
\begin{equation} \label{totbig} \sF(B_1) \simeq \mathrm{Tot}( \sF(B_2) \rightrightarrows \sF(B_2
\otimes_{B_1} B_2)  \triplearrows \dots ),  \end{equation}
and furthermore this diagram is $N$-nilpotent (\Cref{nilpandhyp}). 
It follows that if 
$B_1 \to B_2$ is a faithfully flat \'etale map of ind-\'etale $A$-algebras, then
\eqref{totbig} still holds and the diagram is weakly $N$-nilpotent. Indeed, we
can write $B_1 \to B_2$ as a filtered colimit of 
faithfully flat maps of \'etale $A$-algebras, consider the \v{C}ech nerves of
each of those, 
and 
we can commute the totalization and colimit (\Cref{commutetotandcolimit}). 
Applying this fact to a faithfully flat map between finite \'etale
$A_{\mathfrak{p}}^h$-algebras, we conclude that 
$x^* \sF$ is hypercomplete by \Cref{Postcompletecrit2}, as desired. 

Now suppose the Nisnevich pullbacks $x^* \sF$ are hypercomplete for $x \in X$. 
Let $\sF^h$ denote the \'etale hypercompletion of $\sF$.  We have a map $\sF \to
\sF^h$, and we need to see that it is an equivalence. 
We extend both $\sF, \sF^h$ to 
ind-\'etale $A$-algebras. By assumption, 
$\sF(B') \to \sF^h(B')$ is an equivalence
if $B'$ is \emph{strictly} henselian local. 
Since both $\sF, \sF^h$ are hypercomplete Nisnevich sheaves, it suffices to see that
$\sF \to \sF^h$ induces an equivalence on Nisnevich stalks, i.e, on each ind-\'etale $A$-algebra $B$ which is henselian local. 
However,  if we fix a henselian local $B$ which is ind-\'etale over $A$, then by
assumption and the previous paragraph
$\sF, \sF^h$
define hypercomplete \'etale sheaves on the finite \'etale site of $\spec(B)$
with the same stalk; therefore $\sF(B) \simeq \sF^h(B)$ as desired. 
\end{proof} 

In the case where $X$ has finite Krull dimension, the above result simplifies to
the following. 

\begin{theorem}[Hypercompleteness criterion]
\label{main:hypcriterion}
Let $X$ be a qcqs algebraic space of finite Krull dimension $d$ 
and such that the $\mathcal{P}$-local cohomological dimension of each residue field is $\leq
l$. 
Let $\sF$ be a $\mathcal{P}$-local Nisnevich sheaf of spectra on $X$. Then the following are
equivalent: 
\begin{enumerate}
\item $\sF$ is a hypercomplete \'etale sheaf.  
\item $\sF$ is a Postnikov complete \'etale sheaf. 
\item For every finite group $G$ and every $G$-Galois cover $Y \to Y'$ of
algebraic spaces \'etale over $X$, the map 
$\sF(Y') \to \sF(Y)^{hG}$ is an equivalence and 
the $G$-action on $\sF(Y)$ is $(d + l)$-nilpotent. 
\item 
For each point $x \in X$, the presheaf $x^* \sF$ on $\mathcal{T}_x$ is a
hypercomplete sheaf. 
\item 
There exists an integer $N \geq 0$ such that the following holds. 
For each \'etale map $\spec A \to X$ and every finite $G$-Galois cover
$\spec B \to \spec A$, the map 
$\sF(\spec A) \to \sF(\spec B)^{hG}$ is an equivalence and the $G$-action on 
$\sF(\spec B)$ is weakly $N$-nilpotent. 
\end{enumerate}
\end{theorem} 
\begin{proof} 
Combine \Cref{hypetalehypNis}, the fact that Nisnevich sheaves are hypercomplete 
(\Cref{Nishypercomplete}), and the criterion for hypercompleteness for profinite groups in
\Cref{Gcompletecrit}. 
\end{proof} 

We deduce that the property of being a hypersheaf propagates to modules in a fairly strong sense.

\begin{corollary}\label{transfers}
Let $X$ be a qcqs algebraic space of finite Krull dimension and with a global
bound on the $\mathcal{P}$-local Galois cohomological dimension of the residue
fields, and let $\mathcal{F}$ be a sheaf of $\mathcal{P}$-local spectra on $X_{Nis}$.  Suppose that:
\begin{enumerate}
\item $\mathcal{F}$ admits \emph{finite \'etale transfers} in the sense that for
every finite group $G$ and every $G$-Galois cover $V\rightarrow U$ in $et_X$, there is a genuine $G$-spectrum $\mathcal{F}_{V\rightarrow U}$ whose induced presheaf of spectra on the orbit category of $G$ agrees with the restriction of $\mathcal{F}$ to the $V/H=V\times_G G/H$ for $H\subset G$.

(For example, if $\mathcal{F}$ has finite Galois descent, there is a unique such $\mathcal{F}_{V\rightarrow U}$: this follows from the ``Borel-completion" yoga, \cite{MNN17} 6.3.)

\item There is a hypercomplete sheaf of algebras $\mathcal{A}$ on $X_{et}$ such that $\mathcal{F}$ is a \emph{module over $\mathcal{A}$ compatibly with the transfers} in the following sense: $\mathcal{F}$ admits a module structure over $\mathcal{A}$ as a presheaf, such that for every $V\rightarrow U$ as above, this module structure restricted to the orbit category of $G$ extends to a module structure of $\mathcal{F}_{V\rightarrow U}$ over $\mathcal{A}_{V\rightarrow U}$.
\end{enumerate}
Then $\mathcal{F}$ is a hypercomplete sheaf on $X_{et}$.
\end{corollary}
\begin{proof}
Let $V\rightarrow U$ be a $G$-Galois cover  in $X_{et}$.  From Theorem
\ref{main:hypcriterion} we deduce that $\mathcal{A}_{V\rightarrow U}$ is a
$d$-nilpotent $G$-spectrum.  Thus $\mathcal{F}_{V\rightarrow U}$, being  a
module over $\mathcal{A}_{V\rightarrow U}$, is a $d$-nilpotent genuine
$G$-spectrum (cf.~\cite{MNN17}).  In partiucular $\mathcal{F}_{V\rightarrow U}$ is Borel-complete, meaning $\mathcal{F}$ satisfies finite Galois descent, and $\mathcal{F}(Y)$ is $d$-nilpotent as a $G$-spectrum.  We conclude by Theorem \ref{main:hypcriterion}.
\end{proof}

\begin{corollary}
\label{hypsmashingfinitecd}
Let $X$ be a qcqs algebraic space of finite Krull dimension with a global bound
on the $\mathcal{P}$-local virtual Galois cohomological dimension of its residue
fields.  Then hypercompletion is smashing for sheaves of
$\mathcal{P}$-local spectra on $X_{et}$ and agrees with Postnikov completion.
\end{corollary}
\begin{proof}
This holds locally on $X$ by the special case of \Cref{transfers} where each
$\mathcal{F}_{V\rightarrow U}$ is Borel-complete.  The claim then follows by the
local-global principle (\Cref{locglobhyp}).
Alternatively, we can argue without the use of 
\Cref{transfers} (and the implicit use of genuine $G$-spectra and transfers)
simply by invoking 
part 3 of \Cref{main:hypcriterion}, noting that $m$-nilpotent objects in $\mathrm{Fun}(BG,
\mathrm{Sp})$ form a tensor ideal for any $m$. 
\end{proof}

\section{Localizing invariants on spectral algebraic spaces}

This section collects together 
several general descent theorems for localizing invariants on spectral algebraic
spaces. 
First, we introduce the language of (qcqs) 
$E_2$-spectral algebraic spaces for maximal generality; here one has a sheaf of
$E_2$-rings rather than $E_\infty$-rings, as in \cite{SAG}.
We review the formalism of localizing invariants, and recall (\Cref{Nisdescloc}) that localizing invariants automatically satisfy Nisnevich
descent, after Thomason-Trobaugh. 

Our main results are that topological cyclic homology is an \'etale hypersheaf
(\Cref{TCissheaf}),
and that any $L_n^f$-local localizing invariant is an \'etale sheaf
(\Cref{generaletaledesc}). The latter
result is an extension of the main results of \cite{CMNN} to the $E_2$-setting
and is based on a careful transfer argument, together with the use of the
Dundas-Goodwillie-McCarthy theorem to
reduce from the $E_2$-case to the discrete (hence $E_\infty$) case.

The use of $E_2$-structures and the generality of localizing invariants
introduces some additional technicalities. 
The reader
primarily interested in  
the $E_\infty$-case  and in algebraic $K$-theory may skip most of this section. The most
important result is \Cref{TCissheaf}, which 
(even in the $E_\infty$-case) 
will be crucial in the discussion of Selmer $K$-theory below.

\subsection{Preliminaries on $E_2$-spectral algebraic spaces}
\label{secEtwo}
Here we briefly review some preliminaries on spectral algebraic spaces, always
assumed quasi-compact and quasi-separated (qcqs). For maximal
generality, we will state all of our results for spectral algebraic spaces
modeled on the spectra of 
$E_2$-rings rather than the more standard
$E_\infty$-rings (for which see \cite[Ch. 3]{SAG} for a detailed treatment). 
The reader interested primarily in the case of $E_\infty$-rings can skip this
subsection without loss of continuity. 

\begin{definition}[$E_2$-spectral algebraic spaces]
An \emph{$E_2$-spectral algebraic space} $X=(X_{et},\mathcal{O}_X)$ is a pair consisting of a site $X_{et}$ and a sheaf of (possibly non-connective) $E_2$-algebras on $X_{et}$, such that $(X_{et},\pi_0\mathcal{O}_X)$ is the \'etale site and \'etale structure sheaf of a qcqs algebraic space (over $\mathbb{Z}$) in the usual sense, $\pi_n\mathcal{O}_X$ is a quasi-coherent $\pi_0\mathcal{O}_X$-module for all $n\in\mathbb{Z}$, and $\mathcal{O}_X$ is Postnikov-complete.
A map of $E_2$-spectral algebraic spaces $(X_{et}, \mathcal{O}_X)$ is a map of
$E_2$-ringed sites, which on $\pi_0$ arises from a map of algebraic spaces. We let $\algspc$ denote the $\infty$-category of
$E_2$-spectral algebraic spaces. 

Note in particular that $X$ determines an underlying qcqs algebraic space
denoted $\pi_0 X$; any
properties such as noetherian, finite Krull dimension, etc.\ will refer to this
underlying algebraic space. 
\end{definition}

\begin{construction}[The spectrum of an $E_2$-ring]
We can construct $E_2$-spectral algebraic spaces from $E_2$-rings as follows. 
Let $A$ be an $E_2$-ring. 
By the results of \cite[Sec. 7.5.4]{HA}, for every \'etale $\pi_0(A)$-algebra
$A'_0$, we can construct an $E_2$-algebra $A'$ equipped with a map $A \to A'$
such that 
$\pi_0(A') \simeq A'_0$ and such that $\pi_*(A) \otimes_{\pi_0(A)} \pi_0(A')
\simeq \pi_*(A')$. Moreover, $A'$ is characterized by a universal property in
the $\infty$-category of $E_2$-algebras under $A$: maps $A' \to B$ (under $A$) are in
bijection with maps of commutative rings $\pi_0(A') \to\pi_0(B)$ under
$\pi_0(A)$. 
This easily yields a sheaf of $E_2$-rings on the \'etale site of $\spec( \pi_0
A)$ and an $E_2$-spectral algebraic space, which we write simply as
$\spec^{et}(A)$. Any spectral algebraic space is \'etale locally of this
form. 
\end{construction} 

\begin{definition}[\'Etale site] 
The \emph{\'etale site} of an
$E_2$-spectral algebraic space $X = (X_{et}, \mathcal{O}_X)$ is simply the site
$X_{et}$, which is the \'etale site of the underlying algebraic space.
This is a full subcategory of the $\infty$-category of $E_2$-spectral algebraic
spaces over $X$; such objects $Y \to X$ are called \emph{\'etale} over $X$. 

The theory of \'etale morphisms of $E_2$-rings 
\cite[Sec. 7.5]{HA}
gives the following property of \'etale morphisms of $E_2$-spectral algebraic
spaces. 
Given $Y \to X$ \'etale, for any $Z \to X$, we have an equivalence
\begin{equation}  \hom_{\algspc_{/X}}(Z, Y) \simeq
\hom_{\algspc_{/\pi_0X}}(\pi_0Z, \pi_0Y). \label{E2etalespectral}
\end{equation}
\end{definition}

\begin{definition}[Perfect modules] 
Let $X$ be an $E_2$-spectral algebraic space. 
  As usual, we denote by $\operatorname{Perf}(X)$ the full subcategory of
  $\mathcal{O}_X$-modules consisting of those $\mathcal{O}_X$-modules which
  locally lie in the thick subcategory generated by the structure sheaf.
  Equivalently, $\operatorname{Perf}(-)$ is right Kan extended from its value on
  affines, and on an affine $X$, so $X=\operatorname{Spec}^{et}(A)$ for an
  $E_2$-ring $A$, we have $\operatorname{Perf}(X)=\operatorname{Perf}(A)$.  (The
  equivalence of these two descriptions of $\operatorname{Perf}(X)$ follows from
  \'etale descent for perfect complexes on $E_2$-rings, for which see
  \cite[Theorem 5.4, Prop. 6.21]{DAGXI}.)  
  
 Given $X$,  the object $\operatorname{Perf}(X)$ is then a monoidal
 $\infty$-category.  More specifically, it is an associative algebra object in
 the symmetric monoidal $\infty$-category $\operatorname{Cat}_{\mathrm{perf}}$ of small idempotent-complete stable $\infty$-categories with Lurie's tensor product.  It further has the property of being \emph{rigid}: every object is both left and right dualizable.  This is obvious by reduction to the affine case.
\end{definition} 

\begin{definition}[Quasi-coherent sheaves]
Let $A$ be an $E_2$-ring and let $X = \spec^{et}(A)$ be the associated
$E_2$-spectral algebraic space. 
Given an $A$-module $M$, one obtains a sheaf of $\mathcal{O}_X$-modules 
on $X_{et}$ which sends an \'etale $A$-algebra $A'$ to $A' \otimes_A M$.  
Now let $X$ be an arbitrary $E_2$-spectral algebraic space. 
A \emph{quasi-coherent sheaf} on $X$ consists of the datum of an
$\mathcal{O}_X$-module on $X_{et}$ which restricts on each affine to a sheaf of
the above form. 
We let $\mathrm{QCoh}(X)$ denote the presentably monoidal, stable $\infty$-category
of quasi-coherent sheaves on $X$. Using general results, cf.\ 
\cite[Theorem 10.3.2.1]{SAG},\footnote{The results are stated assuming an
$E_\infty$-structure, but the arguments require only an $E_2$-structure.} it follows that (since $X$ is always assumed qcqs),
then $\mathrm{QCoh}(X)$ is compactly
generated and the compact objects are precisely $\mathrm{Perf}(X)$. 
\end{definition}

\begin{definition}[Algebraic $K$-theory] 
\label{def:Ktheory}
We define the \emph{$K$-theory} $K(X)$ of an $E_2$-spectral algebraic space $X$ to be
the (non-connective) $K$-theory of the stable $\infty$-category $\perf(X)$. Similarly, we define 
the topological Hochschild homology $\mathrm{THH}$, topological cyclic homology
$\mathrm{TC}$, etc. via $\perf(X)$. 
\end{definition} 

The main purpose of this section, and the remainder of this paper,  is to investigate the above presheaves of
spectra, and their variants, on $\algspc$. 
We recall first that these are all Nisnevich sheaves by an argument of Thomason
\cite{TT90}. 
It will next be convenient to review the generality of this statement and the
framework of localizing invariants. 

\subsection{Generalities on additive and localizing invariants}
Next we review the basic notions of additive and localizing invariants
\cite{BGT, HSS}; the slight difference is that we work over an $E_2$-base. 
Fix an $E_2$-spectral algebraic space $X$. 

\begin{definition}[Linear $\infty$-categories over $X$] 
Recall that $\perf(X)$ is a stable, monoidal $\infty$-category. 
Let $\mathrm{Cat}_{\mathrm{perf}}$ denote the symmetric monoidal
$\infty$-category of small, idempotent-complete stable $\infty$-categories
with the Lurie tensor product. 
By definition, a \emph{$\perf(X)$-linear $\infty$-category} is a left module over the
associative algebra $\perf(X)$ in $\mathrm{Cat}_{\mathrm{perf}}$. 
Equivalently, it is a left module over $\qcoh(X)$ in the $\infty$-category of
compactly generated, presentable $\infty$-categories. 
\end{definition}

\begin{definition}[Fiber-cofiber and split exact sequences] 
A sequence of $\perf(X)$-linear $\infty$-categories 
\( \mathcal{C} \to \mathcal{D}\to \mathcal{E}  \)
is a \emph{fiber-cofiber sequence} if and only if $\mathrm{Ind}(\mathcal{C}) \to \mathrm{Ind}(\mathcal{D}) \to
\mathrm{Ind}(\mathcal{E})$ is a split exact sequence 
of $\qcoh(X)$-linear presentable $\infty$-categories
(cf.\ \cite[Def. 3.8]{CMNN}). Similarly, we have also the notion of a \emph{split
exact sequence} of 
$\perf(X)$-linear $\infty$-categories. \end{definition} 

\begin{remark} 
{The fiber-cofiber sequences of small idempotent-complete
stable $\infty$-categories are equivalently just  the idempotent completions of
Verdier quotient sequences (i.e., $\mathcal{E}= \mathcal{D}/\mathcal{C}$ above), and this persists in the
$\operatorname{Perf}(X)$-linear context, cf.\ \cite[Remark D.7.4.4]{SAG}, \cite[Prop.
5.4]{HSS}.}
That is, given a sequence $\mathcal{C} \to \mathcal{D} \to \mathcal{E}$ of
$\perf(X)$-linear $\infty$-categories, it is a fiber-cofiber sequence (resp.
split exact sequence) if and only if it is so as a sequence of underlying stable
$\infty$-categories; the $\perf(X)$-linearity of the adjoints (at the level of
$\mathrm{Ind}$-completions) is automatic
\cite[Prop.~4.9(3)]{HSS}.  Moreover, as the terminology suggests a fiber-cofiber sequence is exactly a null-composite sequence which is both a fiber sequence and a cofiber sequence in the $\infty$-category of $\perf(X)$-modules in $\mathrm{Cat}_{\mathrm{perf}}$.
\end{remark}

\begin{definition}[Weakly localizing and additive invariants] 
Recall \cite{BGT} that a \emph{weakly localizing invariant} over $X$ is a
functor from $\operatorname{Perf}(X)$-linear small idempotent complete stable
$\infty$-categories to spectra\footnote{We could allow our localizing invariants
to take values in more general target stable $\infty$-categories than just
$\operatorname{Sp}$, but one can often reduce to the case of $\operatorname{Sp}$
using a Yoneda embedding.} which is \emph{exact}, meaning it sends fiber-cofiber
sequences to fiber-cofiber sequences.
A \emph{weakly additive} invariant over $X$  is a functor 
from $\perf(X)$-linear $\infty$-categories to spectra 
which carries split exact sequences of $\perf(X)$-linear
$\infty$-categories to direct sums. 

For a weakly localizing or additive invariant $\mathcal{A}$ over $X$, we can evaluate $\mathcal{A}$ on any qcqs algebraic space $Y\rightarrow X$ by setting
$$\mathcal{A}(Y\rightarrow X) := \mathcal{A}(\operatorname{Perf}(Y)).$$
This is contravariant in $Y$, via pullback of perfect modules.

\end{definition} 

\begin{remark} 
The ``weakly" qualification refers to that we don't require any commutation with
filtered colimits, whereas in \cite{BGT} this is required for localizing
invariants.  \end{remark}

\begin{example}\label{locwithcoeff}
Recall \cite{BGT} that algebraic $K$-theory $K$ and topological Hochschild and cyclic
homology  $\mathrm{THH}, \mathrm{TC}$ define weakly localizing invariants (on
all of $\mathrm{Cat}_{\mathrm{perf}}$), while connective $K$-theory $K_{\geq 0}$
defines an additive invariant. 
Another way to produce such invariants is as follows. Let $\mathcal{M}$ be a 
\emph{right} $\perf(X)$-module in $\mathrm{Cat}_{\mathrm{perf}}$. Then we obtain
a functor on $\perf(X)$-linear $\infty$-categories
\[ \mathcal{C} \mapsto K( \mathcal{C} \otimes_{\perf(X)} \mathcal{M})  \]
which defines a weakly localizing invariant over $X$. 
To see this, we observe that the construction 
$\mathcal{C} \mapsto 
\mathcal{C} \otimes_{\perf(X)} \mathcal{M}$ preserves fiber-cofiber sequences
(resp. split exact sequences). This follows because it is $(\infty, 2)$-categorical (i.e.,
it is well-defined for functors and natural transformations), and the condition
of a fiber-cofiber sequence is $(\infty, 2)$-categorical. Compare also the
discussion in \cite[Sec. 3.2]{CMNN} and \cite[Sec. 3]{HSS}.  
\end{example}

The primary goal of this section is to analyze \'etale descent and hyperdescent
properties of weakly localizing and additive invariants.  
Here we will prove the weaker Nisnevich descent. We begin with a
reduction to the connective case. 

\begin{lemma}
Let $X$ be an $E_2$-spectral algebraic space, and let $X_{\geq
0}=(X_{et},(\mathcal{O}_X)_{\geq 0})$ be its connective cover.  For every map
$U\rightarrow X$ in $X_{et}$, the natural functor of $\operatorname{Perf}(U_{\geq 0})$-linear stable $\infty$-categories
$$\operatorname{Perf}(X)\otimes_{\operatorname{Perf}(X_{\geq 0})}\operatorname{Perf}(U_{\geq 0})\rightarrow\operatorname{Perf}(U)$$
(adjoint to the $\operatorname{Perf}(X_{\geq 0})$-linear pullback map $\operatorname{Perf}(X)\rightarrow\operatorname{Perf}(U)$, where the target has a $\operatorname{Perf}(U_{\geq 0})$-linear structure) is an equivalence.
\label{}
\end{lemma}
\begin{proof}
This follows from Lurie's theory of quasi-coherent stacks
\cite[Ch. 10]{SAG},\footnote{Sometimes Lurie assumes $E_\infty$-structures on
the structure sheaves, but this is cosmetic: at most a monoidal structure on the
categories of modules is used, and so $E_2$-rings are sufficient.} cf.\  also
\cite{TT90} and \cite[Theorem 0.2]{Toen12} for precursors.  Namely, we can
define a quasi-coherent stack over $X_{\geq 0}$ by having its sections over an
affine \'etale $\operatorname{Spec}(A_{\geq 0})\rightarrow X_{\geq 0}$ be given
by $\operatorname{Mod}(A)$.  The global sections of this stack are
$\operatorname{QCoh}(X)$.  Since $\operatorname{Mod}(A)$ is compactly generated
with compact objects $\operatorname{Perf}(A)$, it follows from \cite[Sec.
10.3.2]{SAG} that the global sections of this compactly generated stack are
compactly generated with compact objects $\operatorname{Perf}(X)$.  There is an
analogous compactly generated quasi-coherent stack over $U_{\geq 0}$ whose
global compact objects are $\operatorname{Perf}(U)$.  Then the tautological fact
that the restriction of the first stack from $X_{\geq 0}$ to $U_{\geq 0}$ yields
the second stack translates, via the equivalence of \cite[Theorem
10.2.0.2]{SAG}, to the claimed equivalence of $\operatorname{Perf}(X)$-linear categories.
\end{proof}

This admits the following reinterpretation:

\begin{corollary}
\label{reducetoconnective}
Let $X$ be an $E_2$-spectral algebraic space, and $\mathcal{A}$ a
weakly localizing invariant over $X$.  Then we can also view $\mathcal{A}$ as a
weakly localizing invariant over the connective cover $X_{\geq 0}$ of $X$ via the base-change
functor $\operatorname{Perf}(X_{\geq 0})\rightarrow\operatorname{Perf}(X)$, and
the induced presheaves on $et_X=et_{(X_{\geq 0})}$ are the same.

Thus any general claim about descent for \emph{all} localizing invariants on a qcqs algebraic space reduces to the connective case.
\end{corollary}

The classical Thomason-Trobaugh argument (which is also an ingredient in the
theory of compactly generated quasicoherent stacks) then gives the following
basic result:

\begin{proposition}[Nisnevich descent for localizing invariants]
\label{Nisdescloc}
Let $X$ be an $E_2$-spectral algebraic space, and $\mathcal{A}$ a weakly localizing invariant over $X$.  Then $\mathcal{A}$ is a sheaf for $X_{Nis}$.
\end{proposition}
\begin{proof}
By \Cref{reducetoconnective}, we can reduce to the connective case; then this is proved in
\cite[Prop.~A.15]{CMNN} following an argument of \cite{TT90}.
Note that while the result in \cite{CMNN} is stated for $E_\infty$-spectral
algebraic spaces, the $E_\infty$-structure is nowhere used in the arguments. 
\end{proof}

\subsection{$\mathrm{THH}$ and $\mathrm{TC}$ of spectral algebraic spaces}

We next study the basic examples of topological Hochschild and cyclic homology. It is known from \cite{GH99} that $\operatorname{TC}$, which stands for
\emph{topological cyclic homology}, satisfies a very strong form of \'etale
descent, cf.\ also \cite{WG} in the more classical context of ordinary cyclic
homology.  Going through that argument in the current context gives the
following crucial result:

\begin{theorem}
\label{TCissheaf}
Let $X \in \algspc$ and let $\mathcal{M}$ be any right $\perf(X)$-module in $\mathrm{Cat}_{\mathrm{perf}}$.
The presheaf of spectra $\operatorname{TC}_\mathcal{M}$ on 
$et_X$ defined by
$$\operatorname{TC}_\mathcal{M}(U) =  \operatorname{TC}(\perf(U)\otimes_{\perf(X)}\mathcal{M})$$
is an \'etale hypersheaf. 
Moreover, for any prime number $p$, $\mathrm{TC}_\mathcal{M}/p$ is an \'etale Postnikov
sheaf. 
\end{theorem}

Since $\mathrm{TC}_\mathcal{M}$ is a localizing invariant, it is a Nisnevich sheaf, so it suffices to treat the affine case $X=\spec^{et}(R)$,
with the affine \'etale site instead of the usual \'etale site
(\Cref{ifNisreducetoaffine}).  Recall that $\operatorname{TC}$ is built out of $\operatorname{THH}$; therefore we start by studying $\operatorname{THH}$.

Suppose $A$ is an $E_1$-ring.  Then $A$ is an $A-A$-bimodule (in spectra) in the usual way, and the spectrum $\operatorname{THH}(A)$ can be defined as the relative tensor product

$$\operatorname{THH}(A)=A\otimes_{A\otimes A^{op}}A.$$

Now suppose that $R$ is an $E_2$-ring and $A$ is promoted to an $E_1$-algebra in $R$-modules.  Then the $R$-module structure on $A$ in particular commutes with its $A-A$-bimodule structure, and this lets us view $\operatorname{THH}(A)$ as an $R$-module (via the left-hand copy of $A$ above, say).

Now we can state the following \'etale base change property of $\mathrm{THH}$. 
Cf.\ also \cite{WG, McMi03, Ma17} for the result for $E_\infty$-algebras. 
\begin{theorem}
\label{THHetalebasechange}
Let $R\rightarrow R'$ be an \'etale map of $E_2$-algebras, let $A$ an $E_1$-algebra over $R$, and let $B = A\otimes_R R'$ be its base change to $R'$.  Then the comparison map
$$R'\otimes_R\operatorname{THH}(A)\rightarrow\operatorname{THH}(B)$$
of $R'$-modules (adjoint to the map of $R$-modules $\operatorname{THH}(A)\rightarrow\operatorname{THH}(B)$ given by functoriality) is an equivalence.
\end{theorem}
\begin{proof}
Associativity of relative tensor products gives
$$\operatorname{THH}(A)\otimes_RR' = A\otimes_{A\otimes A^{op}}A\otimes_RR'\simeq A\otimes_{A\otimes A^{op}}B,$$
which further evaluates to
$$(A\otimes_{A\otimes A^{op}}(B\otimes B^{op}))\otimes_{B\otimes B^{op}}B \simeq (B\otimes_A B)\otimes_{B\otimes B^{op}}B.$$
Via this our comparison map becomes
$$(B\otimes_A B)\otimes_{B\otimes B^{op}}B\rightarrow (B\otimes_B B)\otimes_{B\otimes B^{op}}B\simeq B\otimes_{B\otimes B^{op}}B.$$
Thus it suffices to show that the fiber $F$ of the multiplication map
$B\otimes_AB\rightarrow B$, viewed as a $B-B$-bimodule, satisfies
$F\otimes_{B\otimes B^{op}}B=0$.  By base-change, we can assume $A=R$ and $B=R'$, so we are in the $E_2$-setting.  At this point we can follow the proof of
\cite[Prop. 7.5.3.6]{HA}, which is the dual claim to what we're trying to establish.  Let
$\overline{e}\in\pi_0B\otimes_\mathbb{Z}\pi_0B$ be any element lifting the
idempotent $e\in\pi_0B\otimes_{\pi_0A}\pi_0B$ for which
$(\pi_0B\otimes_{\pi_0A}\pi_0B)[e^{-1}]=\pi_0B$ via the multiplication map.
Then $\overline{e}$ acts by an isomorphism on $\pi_\ast B$, but it acts locally
nilpotently on $\pi_\ast F$.  Now, there is a $\mathrm{Tor}$-spectral sequence converging conditionally to $\pi_\ast(F\otimes_{B\otimes B^{op}}B)$ with $E_2$ page of the form
$$\operatorname{Tor}^{\pi_\ast(B\otimes B^{op})}_{p,q}(\pi_\ast F,\pi_\ast B).$$
Considering the action of $\overline{e}\in \pi_0(B\otimes B^{op})$ on the Tor groups shows that the $E_2$-page vanishes, whence the conclusion.
\end{proof}

We can formally deduce the following extension:

\begin{corollary}\label{THHetalebasechangecor}
Let $R$ be an $E_2$-ring, and $\mathcal{M}$ a right $\perf(R)$-module in
$\mathrm{Cat}_{\mathrm{perf}}$.  Let $R \to R'$ be an \'etale map of
$E_2$-rings. Then the comparison map
$$R'\otimes_R\operatorname{THH}(\mathcal{M})\rightarrow\operatorname{THH}(R'\otimes_R\mathcal{M})$$
of $R'$-modules is an equivalence.
\end{corollary}
\begin{proof}
When $\mathcal{M}$ is generated by a single object, then the Schwede-Shipley
theorem, in the form proved by Lurie in \cite[Theorem 7.1.2.1]{HA} (with a
natural extension to the $R$-linear case), gives $\mathcal{M}=\operatorname{Perf}(A)$ for some $E_1$-algebra $A$ over $R$, and so this is equivalent to the previous Theorem by the Morita invariance of $\operatorname{THH}$.  The same holds if $\mathcal{M}$ is generated by finitely many objects, because then it is generated by the single object given by their direct sum.  In general, $\mathcal{M}$ is a filtered colimit of its full subcategories generated by finitely many objects.  Since $\operatorname{THH}$ and base change commute with filtered colimits, we deduce the claim in full generality.
\end{proof}

\begin{corollary}
In the situation of the previous corollary, the presheaf $\operatorname{THH}_{\mathcal{M}}$ on the opposite of the category of $E_2$-rings over $R$ defined by 
$$\operatorname{THH}_{\mathcal{M}}(R')=\operatorname{THH}(R'\otimes_R\mathcal{M})$$
is a Postnikov complete \'etale sheaf.  

In fact, a stronger claim holds: the presheaf-level Postnikov truncations of $\operatorname{THH}_{\mathcal{M}}$ are already \'etale sheaves.
\end{corollary}
\begin{proof}
It suffices to prove the second, stronger claim.  Since as $\ast$ varies over
$\mathbb{N}$ the functors $\Omega^{\infty-\ast}$ from spectra to spaces detect
equivalences and preserve limits, to show each Postnikov truncation
$(\operatorname{THH}_{\mathcal{M}})_{\leq n}$ is a sheaf is suffices to show that each
$(\operatorname{THH}_{\mathcal{M}})_{[m,n]}$ is a sheaf for $m\leq n$ in $\mathbb{Z}$.  By
d\'evissage up the Postnikov tower, it therefore suffices to see that each
$K(\pi_k\operatorname{THH}_{\mathcal{M}},0)$ is a sheaf for all $k\in\mathbb{Z}$.  But for
$R\rightarrow R'$ \'etale, \Cref{THHetalebasechangecor} on homotopy groups signifies
$$\pi_k\operatorname{THH}_{\mathcal{M}}(R)\otimes_{\pi_0R}\pi_0R'\overset{\sim}{\rightarrow}\pi_k\operatorname{THH}_{\mathcal{M}}(R').$$
In other words, $\pi_k\operatorname{THH}_{\mathcal{M}}$ on $E_2$-rings with an
\'etale map from $R$ identifies with the quasi-coherent sheaf associated to the
$\pi_0 R$-module $\pi_k\operatorname{THH}_{\mathcal{M}}(R)$.  Therefore the claim follows from standard \'etale descent theory in commutative algebra (exactness of the Amitsur complex plus preservation of finite products).
\end{proof}

\begin{proof}[Proof of \Cref{TCissheaf}] 
Since $\mathrm{TC}_{\mathcal{M}}$ is a Nisnevich sheaf, it suffices to treat this claim on the
affine \'etale site over an $E_2$-ring $R$.  
We will first show that it is an \'etale hypersheaf. 
Recall that, by definition (cf.\ \cite[Sec.
6.4.3]{DGM13}), $\operatorname{TC}(-)$ is built from $\operatorname{THH}(-)$
and the $\operatorname{TC}(-;p)/p$ for primes $p$ via homotopy limits and
extensions.  Since $\operatorname{THH}_{\mathcal{M}}(-)$ is an \'etale Postnikov  sheaf by
\Cref{THHetalebasechange}, it suffices to show that $\operatorname{TC}_{\mathcal{M}}(-;p)/p$
is a hypersheaf.  But again, $\operatorname{TC}(-;p)$ is built from the
$\operatorname{TR}^n(-;p)$ via homotopy limits, and the
$\operatorname{TR}^n(-;p)$ can be gotten inductively starting from
$\operatorname{TR}^1(-;p)=\operatorname{THH}(-)$ by the fundamental cofiber
sequence \cite{HM97} 
$$\operatorname{THH}(-)_{hC_{p^n}}\rightarrow \operatorname{TR}^n(-;p)\rightarrow\operatorname{TR}^{n-1}(-;p).$$
Thus it suffices to show that
$(\operatorname{THH}_{\mathcal{M}}(-)/p)_{hC_{p^n}}$ is a hypersheaf.  We know $\operatorname{THH}_{\mathcal{M}}/p$ is a hypersheaf, so the issue is just to check that this property is preserved by the homotopy orbits in this case.

However, letting $k=\pi_0R$ and viewing $\operatorname{THH}(-)$ as a sheaf on
$\operatorname{Spec}(R)_{et}=\operatorname{Spec}(k)_{et}$, we have that the
sheaf $\operatorname{THH}_{\mathcal{M}}(-)/p$ vanishes on any $k[1/p]$-algebra.  Therefore, it
is pushed forward from the \'etale topos of $k\otimes_{\mathbb{Z}}
\mathbb{F}_p$, the latter forming the closed complement of
$\operatorname{Spec}(k[1/p])_{et}$ in $\operatorname{Spec}(k)_{et}$ \cite[Exp.
IV, Sec. 9]{SGA4}. For sheaves supported on a closed subtopos, the question of
hypercompleteness doesn't depend on whether we consider the sheaves on the
original topos or on the subtopos, because the pushforward functor is fully
faithful and t-exact.  Thus it suffices to see that taking homotopy orbits
$(-)_{hC_{p^n}}$ preserves hypercompleteness of $p$-power torsion sheaves on
$\operatorname{Spec}(k\otimes_\mathbb{Z}\mathbb{F}_p)_{et}$.  However, the (mod
$p$) \'etale cohomological dimension of any commutative $\mathbb{F}_p$-algebra
is $\leq 1$ \cite[Theorem 5.1, Exp. X]{SGA4}, so it follows from 
\Cref{critforhypcolimits}
that hypercomplete sheaves are closed under all colimits of $p$-power torsion
presheaves.  This finishes the proof of hypercompleteness. 
Moreover, since any hypercomplete $p$-torsion sheaf of spectra  on $\spec(k
\otimes_{\mathbb{Z}} \mathbb{F}_p)_{et}$ is Postnikov complete
(cf.~\Cref{modpdimensionex}), the Postnikov
completeness claim follows too. 
\end{proof} 

The above argument by reduction to the \'etale site of an $\mathbb{F}_p$-algebra
also gives a cohomological dimension bound for the homotopy group sheaves:

\begin{corollary}
\label{TCPostnikovcomplete}
Let $n>0$.  Consider the \'etale hypersheaf $\operatorname{TC}_{\mathcal{M}}/n$
over an arbitrary $E_2$-algebra $R$.  Then its \'etale homotopy groups sheaves
$\widetilde{\pi}_\ast$ have cohomological dimension $\leq 1$, so
$\operatorname{TC}_{\mathcal{M}}/n$ is actually a Postnikov sheaf, and the resulting descent spectral sequence
$$H^{-p}(\operatorname{Spec}(R)_{et};\widetilde{\pi}_q)\Rightarrow
\pi_{p+q}(\operatorname{TC}_{\mathcal{M}}(R)/n)$$
simply gives short exact sequences
$$0\rightarrow H^1(\operatorname{Spec}(R)_{et};\widetilde{\pi}_{d+1})\rightarrow
\pi_d(\operatorname{TC}_{\mathcal{M}}(R)/n)\rightarrow H^0(\operatorname{Spec}(R)_{et};\widetilde{\pi}_d)\rightarrow 0$$
for all $d\in\mathbb{Z}$.
\end{corollary}

\begin{corollary}
\label{TCGnilpotent}
Let $X$ be a qcqs algebraic space, and $Y\rightarrow X$ a $G$-Galois cover, $G$
a finite group, and $\mathcal{M}$ a $\perf(X)$-module in $\mathrm{Cat}_{\mathrm{perf}}$.  Then $\operatorname{TC}_{\mathcal{M}}(Y) \in \fun(BG, \sp)$ is $d$-nilpotent for some $d\geq 0$.
\end{corollary}
\begin{proof}
Every qcqs algebraic space is glued from affines in a finitary manner in the
Nisnevich topology \cite[Sec.
3.4.2]{SAG}, so
by localization we can reduce to the case of $X$ affine.  In this case we claim
that $d=4$ works.  Any rational spectrum with $G$-action is $0$-nilpotent (being a module
over $H\mathbb{Q}$ with trivial $G$-action, which is a retract of
$H\mathbb{Q}[G]$), so it suffices to see that $\varinjlim_n
\operatorname{TC}_{\mathcal{M}}(Y)/n$ is $3$-nilpotent.  A countable filtered colimit is the
cofiber of a map of countable direct sums, 
and a countable direct sum of $s$-nilpotent objects is $s$-nilpotent (for any
$s$),
so it suffices to see that each
$\operatorname{TC}_{\mathcal{M}}(Y)/n$ is $1$-nilpotent.  For this we can assume $n=p^j$ is a
prime power.  Then again since $\operatorname{TC}_{\mathcal{M}}/p^j$ is zero over any ring
where $p$ is invertible, we can view it as an \'etale sheaf on the mod $p$ locus
of $\operatorname{Spec}(A)$.  It is Postnikov complete by
\Cref{TCPostnikovcomplete}, but on the other
hand it can be considered as a module over the $p$-power torsion $E_1$-ring
$End(\mathbb{S}/p^j)$.  Since the (mod $p$) locus of any commutative ring has (mod $p$)
\'etale cohomological dimension $\leq 1$,
\Cref{nilpotenceofetalePostnikov} implies that $\operatorname{TC}_{\mathcal{M}}(Y)/p^j$ is $1$-nilpotent, finishing the proof.
\end{proof}

\subsection{\'Etale descent for telescopically localized invariants}
Here we prove 
a generalization of the main result of \cite{CMNN} to the $E_2$-case. 
That is, we show that telescopically localized localizing invariants 
satisfy  \'etale descent. 
Since we already know Nisnevich descent (\Cref{Nisdescloc}), it suffices to
handle finite \'etale descent by \Cref{redtofintieGalois}. For this, we can even
work with additive invariants. 

\begin{definition}[The finite \'etale site] 
Let $A$ be an $E_2$-algebra. The \emph{finite \'etale site} of $A$ is the
$\infty$-category of \'etale $E_2$-algebras $B$ under $A$ 
(cf.~\cite[Corollary 7.5.4.3]{HA})
such that $\pi_0(B)$
is finite projective as a $\pi_0(A)$-module; this is equivalent to the finite
\'etale site of $\pi_0(A)$. We say that a finite \'etale $A$-algebra $B$ is
\emph{$G$-Galois} (for a finite group $G$) if $\pi_0(B)$ is a $G$-Galois
extension of $\pi_0(A)$.\footnote{In particular, we do not consider here the
more general Galois extensions of ring spectra of \cite{Rog08}.}
\end{definition}

It is well-known (via an elementary transfer argument) that 
with \emph{rational} coefficients, additive invariants satisfy finite \'etale
descent. Compare the following result, which will be proved more generally
below. 

\begin{proposition}[{Cf.\ \cite[Prop. 2.14]{Th85}}] 
Let $A$ be an $E_2$-algebra, and let $\mathcal{A}$ be a weakly additive invariant for
$A$-linear $\infty$-categories valued in $\mathbb{Q}$-module spectra. 
Then $\mathcal{A}$ defines a sheaf on the finite \'etale site of
$\spec(A)$. 
\end{proposition}

Note that the conditions in 
\cite[Prop. 2.14]{Th85} are slightly stronger; however, it is not difficult to
deduce the more general result. The key observation to run the argument is that if $A \to B$ is
finite \'etale, then the class that $B$ defines in $K_0(A) \otimes \mathbb{Q}$
is a unit (cf.\ also \cite[Prop. 5.4]{CMNN}). 
In this subsection, we give an extension 
of this transfer argument to additive invariants that take values in
$L_n^f$-local spectra rather than 
rational spectra. For $E_\infty$-rings, this was done in \cite{CMNN}. We extend
to the $E_2$-case using the Dundas-Goodwillie-McCarthy theorem, starting with
$K$-theory. 

Throughout this section we fix an implicit prime $p$. 
We use the theory of finite localizations, cf.\ \cite{Mi92}. In particular, as is
conventional, we let $L_n^f: \sp \to\sp_{(p)}$ denote $p$-localization
together with the finite localization away
from a type $(n + 1)$ ($p$-local) finite spectrum.

\begin{proposition} 
\label{finitedescLnK}
Let $A$ be a connective $E_2$-algebra. 
Then the functor $B \mapsto L_n^f (K_{\geq 0}(B))$ defines a sheaf on the 
finite \'etale site of $A$. 
\end{proposition} 
\begin{proof} 
This follows using the Dundas-Goodwillie-McCarthy theorem 
and \cite{CMNN}. 
Note that for the finite \'etale site, the results of \cite{CMNN} do not require
a localizing invariant, only an additive invariant. 

Indeed, \cite[Theorem 5.1, Proposition 5.4]{CMNN} shows that $B \mapsto L_n^f ( K_{\geq 0}(\pi_0 B))$ 
defines a sheaf on the finite \'etale site (of $A$ or equivalently of
$\pi_0(A)$); note that this is a purely
algebraic statement about the ring $\pi_0(A)$, and the distinction between
$E_2$ and $E_\infty$ disappears. 
By the Dundas-Goodwillie-McCarthy theorem, it now suffices to show that
$B \mapsto L_n^f ( \mathrm{TC}(B))$ and $B \mapsto L_n^f (\mathrm{TC}(\pi_0 B)
)$ are sheaves on the finite \'etale site. 
Now \Cref{TCissheaf} shows that $B \mapsto \mathrm{TC}(B), B \mapsto \mathrm{TC}(\pi_0
B)$ are \'etale hypersheaves, so it remains to see that we still have finite
\'etale descent after applying $L_n^f$-localization, i.e., that we can commute
$L_n^f$-localization and finite Galois homotopy fixed points. However, this follows from 
\Cref{TCGnilpotent}. 
\end{proof}

To extend the result from connective $K$-theory to arbitrary additive
invariants, we recall the construction of the $\infty$-category of
noncommutative motives. The basic theorem about this $\infty$-category is that
algebraic $K$-theory becomes representable here. 

Fix an $E_2$-ring $A$. We will abbreviate ``$\perf(A)$-linear
$\infty$-category'' to ``$A$-linear $\infty$-category.''

\newcommand{\mot}{\mathrm{Mot}}
\begin{construction}[Noncommutative motives, cf.\ \cite{Tab08, BGT, HSS}]
Let $\mot_A$ denote the presentable, stable $\infty$-category of
\emph{noncommutative motives} over $A$. By construction, we have a functor 
$\mathcal{C} \mapsto [\mathcal{C}]$ from $A$-linear $\infty$-categories 
to $\mot_A$ with the following two properties: 
\begin{enumerate}
\item  $\mathcal{C} \mapsto [\mathcal{C}]$ preserves filtered colimits. 
\item Given an $A$-linear $\infty$-category $\mathcal{C}$ admitting a
semiorthogonal decomposition into subcategories $\mathcal{C}_1, \mathcal{C}_2$
(i.e., one has a strict exact sequence $\mathcal{C}_1 \to \mathcal{C} \to
\mathcal{C}_2$),
we have that
the natural map 
$[\mathcal{C}] \to [\mathcal{C}_1] \times [\mathcal{C}_2]$ is an equivalence in
$\mot_A$. 
\end{enumerate}
Furthermore, $\mot_A$ is initial among presentable, stable $\infty$-categories admitting 
a functor $[\cdot]$ with the above two properties. 
\end{construction}

In the following, we write $\mathrm{Fun}_A(-, -)$ for $A$-linear, exact functors
between $A$-linear $\infty$-categories. 

\begin{theorem}[{Cf.\ Tabuada \cite{Tab08}, Blumberg-Gepner-Tabuada \cite{BGT},
Hoyois-Schereztoke-Sibilla \cite{HSS}}]\label{additivekviamotives}
Let $\mathcal{C}_1, \mathcal{C}_2$ be small $A$-linear $\infty$-categories
and suppose that $\mathcal{C}_1$ is compact as an $A$-linear $\infty$-category. 
Then there is a natural equivalence $\hom_{\mot_A}( [\mathcal{C}_1],
[\mathcal{C}_2]) \simeq K_{\geq 0}(
\fun_A(\mathcal{C}_1, \mathcal{C}_2))$. 
\end{theorem} 

\begin{remark} 
The reference \cite{BGT} considers the case $A = \mathbb{S}$. The paper
\cite{HSS} 
proves the result  for a stably symmetric monoidal $\infty$-category with all objects
dualizable (e.g., $\perf(A)$ for $A$ an $E_\infty$-ring). However, the arguments
all go through with only a monoidal structure, and hence $A$ is allowed to only be $E_2$. 
\end{remark}

\begin{remark} 
The use here of the language of noncommutative motives is not essential. 
For our purposes, it would be sufficient to start with 
the additive $\infty$-category of finite \emph{projective} $B$-modules (for each $B$ which is
finite \'etale over $A$) and $A$-linear additive functors; then one can use
additive $K$-theory. It plays a larger role in \cite{CMNN} because of the use of
extensions such as $KO \to KU$, which are not \'etale at the level of $\pi_0$. 
\end{remark}

Our goal is to upgrade 
\Cref{finitedescLnK} to a statement at the level of $\mot_A$, and then to deduce a
corresponding statement for an arbitrary additive invariant. 

\begin{proposition} 
\label{kernelfunctor}
Let $A$ be an $E_2$-ring, $B$ a finite \'etale
$E_2$-$A$-algebra, and $\mathcal{M}$ an $A$-linear $\infty$-category.
Then there is a natural equivalence
of stable $\infty$-categories
\[ \fun_A( \perf_B, \mathcal{M}) \simeq \perf_{B^{op}} \otimes_{\perf(A)} \mathcal{M} . \]
\end{proposition} 
\begin{proof} 
This is a standard consequence of the fact that $B$ is \emph{proper and smooth} as an $A$-algebra, i.e.:
\begin{enumerate}
\item $B$ is perfect as an $A$-module;
\item $B$ is perfect as a $B\otimes_A B^{op}$-module.
\end{enumerate}
Namely, passing to Ind-categories and use Lurie's tensor product on presentable $\infty$-categories, we always have
$$\operatorname{Fun}_A^L(\operatorname{Mod}_B,\operatorname{Ind}(\mathcal{M}))\simeq \operatorname{Mod}_{B^{op}}(\operatorname{Ind}(\mathcal{M})) \simeq \operatorname{Mod}_{B^{op}}\otimes_{\operatorname{Mod}_A}\operatorname{Ind}(\mathcal{M})$$
without any hypothesis on the $A$-algebra $B$, see \cite[Theorems 4.8.4.1 and 4.8.4.6]{HA}.  Thus what one needs to see is that an object of $\operatorname{Mod}_{B^{op}}(\operatorname{Ind}(\mathcal{M}))$ is compact if and only if the underlying object of $\operatorname{Ind}(\mathcal{M})$ is compact.  ``Only if" follows from condition 1 and ``if" follows from condition 2.

In fact, in our case $B$ is not just perfect but finitely generated projective over both $A$ and $B\otimes_A B^{op}$, as one checks on homotopy.
\end{proof}

\begin{remark} 
Let $A$ be an $E_2$-ring, and let $B$ be an $E_2$-algebra which is \'etale over
$A$. 
The above construction relied upon $B$ as an $E_1$-algebra in $A$-modules
(which form a monoidal $\infty$-category), which
enabled us to form the opposite algebra $B^{op}$. 
However, in fact $B$ is canonically identified with $B^{op}$, and both have the
structure of $E_2$-algebras under $A$; this follows from the theory of \'etale
morphisms as treated in \cite[Sec. 7.5]{HA}.  
\end{remark} 

\begin{corollary} 
\label{kernelfiniteetale}
Let $A$ be a connective $E_2$-ring, and let $B, B'$ be finite \'etale
$E_2$-algebras under $A$. 
By a slight abuse of notation, let $B \otimes_A B'$  denote the finite
\'etale $E_2$-algebra  over $A$ whose $\pi_0$ is $\pi_0(B) \otimes_{\pi_0(A)}
\pi_0(B')$. 
Then $\fun_A( \perf_B, \perf_{B'})$ is identified  with the stable
$\infty$-category of perfect modules over $B \otimes_A B'$. 
\end{corollary} 
\begin{proof} 
We have natural maps of $E_2$-algebras under $A$, 
$f_1: B \to B \otimes_A B'$ and $f_2: B' \to B \otimes_A B'$, thanks to the
general theory of \'etale morphisms \cite[Sec. 7.5]{HA}.  These maps induce
extension of scalars functors
$f_1^*: \perf_B \to \perf_{B \otimes_A B'}, 
f_2^*: \perf_{B'} \to \perf_{B \otimes_A B'}
$ 
and right adjoint restriction of scalars functors $f_{1*}, f_{2*}$; all of
these are naturally $A$-linear.  
Given a perfect $B \otimes_A B'$-module $P$, 
we define a functor $\perf_B \to \perf_{B'}$ which sends 
$M \mapsto f_{2*}(f_1^*(M) \otimes_{B \otimes_A B'} P)$. 
Via \Cref{kernelfunctor}, this implies that functors are precisely modules
over this tensor product, whence the result. 
\end{proof} 
\begin{corollary} 
\label{sheafforfiniteetalehom}
Let $A$ be a connective $E_2$-ring. 
Then for any finite \'etale $E_2$-algebra $B$, the functor
\[ B' \mapsto L_n^f \hom_{\mot_A}( [\perf(B)], [\perf(B')])  \]
is a sheaf for the finite \'etale topology. 
\end{corollary} 
\begin{proof} 
From the proposition we deduce that $\perf(B)$ is compact as an $A$-linear $\infty$-category, so \Cref{additivekviamotives} shows that $\hom_{\mot_A}( [\perf(B)], [\perf(B')]) = K_{\geq 0}\fun_A(\perf(B),\perf(B'))$.  Applying \Cref{kernelfiniteetale} again identifies this with $K_{\geq 0}(\perf(B\otimes_AB'))$.  Then the claim follows from \Cref{finitedescLnK}. 
\end{proof} 

To extend this to more general localizing invariants, we will need a bit more
nilpotence technology. 
That is, for a $G$-Galois extension $A \to B$, we claim that the noncommutative motive 
$L_n^f [\perf(B)]$ is actually nilpotent, as an object of $\fun(BG, \mot_A)$.  

\begin{proposition} 
\label{nilpotenceGactionK}
Let $B$ be a $G$-Galois extension of the connective $E_2$-algebra
$A$.\footnote{That is, $B$ is \'etale over $A$, and $\pi_0(A) \to \pi_0(B)$ is $G$-Galois
as a map of commutative rings.}
Then the $G$-action on $L_n^f K_{\geq 0}(B)$ is nilpotent. 
\end{proposition} 
\begin{proof} 
Using the Dundas-Goodwillie-McCarthy theorem and \Cref{TCGnilpotent}, 
we see that it suffices to  assume that $A, B$ are discrete and in particular
$E_\infty$. 
In this case, it follows that $K_{\geq 0}(B)$ is an $E_\infty$-algebra in
$\fun(BG, \sp)$.
It suffices to show that for each $n$ and prime $p$, we have
$(L_n^f  K_{\geq 0}(B))^{tG} = 0$, thanks to \Cref{nilpexpbounded}. 
This effectively follows from the main results 
of \cite{CMNN}, though since it is not spelled out there explicitly we indicate
the argument.
Indeed,  a transfer argument (cf.\ \cite[Theorem 5.1]{CMNN}) shows that we have
that
$$(K_{\geq 0}(B)^{hG})_{\mathbb{Q}} \simeq (K_{\geq 0}(B)_{\mathbb{Q}})^{hG}
\simeq K_{\geq 0}(A)_{\mathbb{Q}}.$$  
The slightly subtle point is the first equivalence. 
In particular, taking $G$-homotopy fixed points commutes with rationalization on
$K_{\geq 0}(B)$. 
Taking $G$-homotopy orbits \emph{always} commutes with rationalization.
Combining these observations, it follows that 
\[ (K_{\geq 0}(B)^{tG})_{\mathbb{Q}} \simeq 
(K_{\geq 0}(B)_{\mathbb{Q}})^{tG} =0 .
\]
But by the May nilpotence conjecture \cite{MNN15}, since everything is now
$E_\infty$, this implies that for each 
$n$, we have $L_n^f ( K_{\geq 0}(B)^{tG}) = 0$. 
As we have a natural ring map 
$L_n^f ( K_{\geq 0}(B)^{tG}) \to (L_n^f K_{\geq 0}(B))^{tG}$, it follows that
the latter vanishes,
as desired. 
\end{proof} 

\begin{construction}[$L_n^f$-localized noncommutative motives]
The presentable stable $\infty$-category $\mot_A$ is compactly generated via the classes
$[\mathcal{C}]$, for $\mathcal{C}$ a compact $A$-linear $\infty$-category
(thanks to \cite[Prop. 5.5]{HSS}). 
It follows that when we form the $L_n^f$-localized $\infty$-category $L_n^f
\mot_A$ (i.e., Bousfield localization at the maps $F \otimes X \to X, X \in
\mot_A$, for $F$ a type $(n+1)$ complex), then the mapping spaces between
compact objects are simply $L_n^f$-localized. That is, we have
$$\hom_{L_n^f \mot_A} ( L_n^f [\mathcal{C}], L_n^f [\mathcal{D}]) \simeq L_n^f
\hom_{\mot_A}( [\mathcal{C}], [\mathcal{D}]) \simeq L_n^f K_{\geq 0}( \fun_A(
\mathcal{C}, \mathcal{D}))$$
when $\mathcal{C}, \mathcal{D}$ are compact $A$-linear $\infty$-categories. 
\end{construction}
\begin{corollary} 
\label{nilpotenceGactionmot}
Let $B$ be a $G$-Galois extension of the connective $E_2$-algebra $A$. 
Then the $G$-action on $L_n^f [\perf(B)]$ is nilpotent. 
\end{corollary} 
\begin{proof} 
Combine \Cref{nilpotenceGactionK}, \Cref{kernelfiniteetale}, and
\Cref{nilpendocrit} (and the above construction). 
In particular, note that $B \otimes_A B$ is a $(G \times G)$-Galois extension of
$A$. 
\end{proof}

\begin{proposition} 
Let $\mathcal{C}$ be a stable $\infty$-category. 
Let $Y \in \fun(BG, \mathcal{C})$ be nilpotent and let $X \to Y^{hG}$ be a
morphism. 
Then the following are equivalent: 
\begin{enumerate}
\item $X \to Y^{hG} $ is an equivalence. 
\item For $Z $ either $X$ or $Y$, the map $\hom_{\mathcal{C}}(Z, X) \to
\hom_{\mathcal{C}}(Z,
Y)^{hG}$ is an equivalence. 
\end{enumerate}

\label{equivalencecritnilp}
\end{proposition} 
\begin{proof} 
The hypothesis of nilpotence implies that 
$Y^{hG} \in \mathcal{C}$ belongs to the thick subcategory generated by $Y$; in
fact,  this holds because the cosimplicial diagram computing $Y^{hG}$ is rapidly converging. 
Therefore, the cofiber $C$ of $X \to Y^{hG}$ belongs to 
the thick subcategory generated by $X, Y$, which easily implies the claim,
since our hypotheses show that $\mathrm{Hom}_{\mathcal{C}}(X, C)
=\mathrm{Hom}_{\mathcal{C}}(Y, C) =0$. 
\end{proof}

\begin{proposition} 
\label{finiteGaldescadditive}
Let $A$ be an $E_2$-ring, and let $\mathcal{A}$ be a weakly additive invariant of  
$A$-linear $\infty$-categories with values in $L_n^f$-local spectra. 
Let  $A \to B$ be a $G$-Galois extension of $E_2$-rings. 
Then 
$\mathcal{A}(\perf(A)) \simeq \mathcal{A}(\perf(B))^{hG}$.  Moreover $\mathcal{A}(\perf(B))$ is $d$-nilpotent as a $G$-object for some $d\geq 0$ depending only on $A\to B$, $n$, and the implicit prime $p$.
\end{proposition} 

\begin{proof}  
Since we are only evaluating on compact objects, we may as well assume that
$\mathcal{A}$ is additive. 
Moreover, we can assume $A$ connective as in \Cref{reducetoconnective}. 
By \Cref{equivalencecritnilp}, \Cref{nilpotenceGactionmot}, and
\Cref{sheafforfiniteetalehom}, it follows that  in the $\infty$-category $L_n^f \mot_A$, 
we have $L_n^f [\perf(A)] \simeq (L_n^f [\perf(B)])^{hG}$. 
Since the $G$-action on 
$L_n^f [\perf(B)]$ is nilpotent (\Cref{nilpotenceGactionmot}), 
it follows that after applying any exact functor $\mathcal{A}: L_n^f \mot_A \to
L_n^f\sp$, we have that
\[ \mathcal{A}( L_n^f [\perf(A)]) 
\simeq \mathcal{A}( L_n^f [\perf(B)]^{hG}) \simeq \mathcal{A}( L_n^f
[\perf(B)])^{hG}
,\]
where the last equivalence uses the nilpotence. This is the desired claim,
since any additive invariant naturally factors through $\mot_A$.  
\end{proof} 

\begin{remark} 
Using the language of $G$-equivariant stable homotopy theory and equivariant
algebraic $K$-theory (cf.\ \cite{Me17, B17, BGS20}), one can streamline
the above arguments. 
Since the constructions given there are not exactly in the generality we need,
we have followed the approach above.
\end{remark} 

\begin{theorem}[\'Etale descent]
\label{generaletaledesc}
Let $X$ be an $E_2$-spectral algebraic space. 
Let $\mathcal{A}$ be a weakly localizing invariant for $\perf(X)$-linear
$\infty$-categories which takes values in $L_n^f$-local spectra. 
Then $Y \mapsto \mathcal{A}(\perf(Y))$ defines an \'etale sheaf on $X$. 
\end{theorem} 
\begin{proof} 
By the Thomason-Trobaugh argument, $Y \mapsto \mathcal{A}(\perf(Y))$ is a
Nisnevich sheaf (\Cref{Nisdescloc}). Thus it suffices to see that it satisfies finite Galois
descent on affines, by \Cref{redtofintieGalois}. But this follows from 
\Cref{finiteGaldescadditive}. 
\end{proof}

\section{Selmer $K$-theory and hyperdescent results}

In this section we 
first study the basic properties of Selmer $K$-theory
(\Cref{SelKdef}). Our main result is that it always satisfies \'etale descent,
and satisfies hyperdescent under mild finite-dimensionality assumptions
(\Cref{QLSelgeneral}). 
We show also a version of the Lichtenbaum-Quillen conjecture, that the map from
$K$ to $K^{Sel}$ is an equivalence in high enough degrees 
under mild assumptions. 
A crucial ingredient in proving these statements is the Beilinson-Lichtenbaum
conjecture, proved by Voevodsky-Rost \cite{Voe03, Voe11}. 
Finally, we use the smashing property of \'etale hypercompletion to extend these
hyperdescent results to $L_n^f$-local localizing invariants (\Cref{mainhypLn}). 

\subsection{Basic properties of Selmer $K$-theory}

Here we recall the definition of  \emph{Selmer
$K$-theory}, introduced in \cite{Artinmaps}. 
Our main result (\Cref{selmer:basicproperties}) is that  Selmer $K$-theory satisfies \'etale descent
and that it commutes with filtered colimits for connective ring spectra. 

\begin{definition}[Selmer $K$-theory, cf.\ \cite{Artinmaps}]
\label{SelKdef}
Let $\mathcal{C}$ be a small stable $\infty$-category.
The \emph{Selmer $K$-theory} $K^{Sel}(\mathcal{C})$ is defined to be the
homotopy pullback
\begin{equation} \label{Seldef}  K^{Sel}(\mathcal{C}) = L_1 K(\mathcal{C}) \times_{L_1
\mathrm{TC}(\mathcal{C})} \mathrm{TC}(\mathcal{C}). \end{equation}
For an $E_2$-spectral algebraic space $X$, we write $K^{Sel}(X) = K^{Sel}(
\perf(X))$. 
We define $K^{Sel}$ of an $E_1$-ring $R$ by $K^{Sel}(R)=K^{Sel}(\operatorname{Perf}(R))$.

\end{definition} 

Note that applying the localization natural transformation $id\rightarrow L_1$ to the cyclotomic trace $K\rightarrow \operatorname{TC}$ gives rise to a natural map $K\rightarrow K^{Sel}$ which factors both the trace $K\rightarrow \operatorname{TC}$ and the localization $K\rightarrow L_1K$.

\begin{example}[Rational Selmer $K$-theory] 
Recall that $T_\mathbb{Q}\overset{\sim}{\rightarrow} (L_1T)_{\mathbb{Q}}$ for any spectrum $T$.  Hence the above formula \eqref{Seldef} shows that 
the rationalization of $K^{Sel}$ is simply the rationalization of $K$.
\end{example} 
Consequently, the difference between $K$ and $K^{Sel}$ is seen after $p$-adic
completion (for any prime $p$). 
In this case, the idea of Selmer $K$-theory roughly is to glue prime-to-$p$ phenomena (which are seen
via $L_1 K$) and $p$-adic phenomena (which are handled by $\mathrm{TC}$). 
We illustrate this with two basic examples. 

\begin{example}[Selmer $K$-theory for rings with $p$ inverted] 
Suppose $\mathcal{C}$ is a stable $\infty$-category on which $p$ is
invertible. 
In this case, $p$ is also invertible on $\mathrm{THH}(\mathcal{C})$ and hence on
$\mathrm{TC}(\mathcal{C})$. It follows that the
$p$-adic completion ${K}^{Sel}(\mathcal{C})_{\hat{p}} $ is identified with $L_{K(1)}
K(\mathcal{C})$, where $K(1)$ is at the same prime $p$. 
\label{SelKawayp}
\end{example} 

\begin{example}[Selmer $K$-theory in the $p$-complete case] 
\label{SelKatp}
Suppose that $R$ is a connective associative ring spectrum such that
$\pi_0(R) $ is commutative and $p$-henselian (e.g., $p$-adically complete). Then applying
$K(1)$-localization at the prime $p$ (which we note erases the difference between connective
and nonconnective $K$-theory) to the main result of
\cite{CMM} combined with the Dundas-Goodwillie-McCarthy theorem \cite{DGM13} gives a 
homotopy pullback square
\begin{equation} \label{k1pa} \xymatrix{
L_{K(1)} K(R) \ar[d]  \ar[r] &  L_{K(1)} \TC(R) \ar[d]  \\
L_{K(1)} K(\pi_0(R)/p) \ar[r] &  L_{K(1)} \TC( \pi_0(R)/p).
}.\end{equation}
Note that $L_{K(1)} K(\mathbb{F}_p) =0 $ since the
$p$-completion ${K}(\mathbb{F}_p)_{\hat{p}}$  vanishes in positive degrees thanks to 
\cite{Qui72}.  
The bottom row in 
\eqref{k1pa} consists of ${K}(\mathbb{F}_p)_{\hat{p}}$-modules, and therefore
vanishes. It follows that in this case, we have 
that the top horizontal arrow in \eqref{k1pa} is an equivalence, and
consequently
${K}^{Sel}(R)_{\hat{p}} = {\TC}(R)_{\hat{p}}$. 
\end{example} 

\begin{remark}
For any $E_1$-ring $R$ we have a pullback square
\begin{equation} \xymatrix{
K^{Sel}(R) \ar[d]  \ar[r] &  K^{Sel}(R_{\hat{p}}) \ar[d]  \\
K^{Sel}(R[1/p]) \ar[r] &  K^{Sel}(R_{\hat{p}}[1/p]).
}\end{equation}
Actually, this holds for any localizing invariant: it can be obtained by
identifying the fiber terms in the localization sequences on perfect complexes
associated to the vertical maps of rings.  Supposing $R$ is connective with
$\pi_0R$ commutative of bounded $p$-torsion, then after $p$-completion the terms
in this pullback can be described in terms of either $L_{K(1)}K$ or
$\operatorname{TC}$ by the previous examples.  In \cite{BCM} it is shown in  the same setting that this pullback square identifies with the $p$-completion of the pullback square defining $K^{Sel}(R)$, and in particular each of the terms is a localizing invariant of $\operatorname{Perf}(R)$.
\end{remark}

We now prove some basic general properties of 
Selmer $K$-theory. 
\begin{theorem} 
\label{selmer:basicproperties}
\begin{enumerate}
\item  
The functor $R\mapsto K^{Sel}(R)$, from connective $E_1$-rings to spectra, commutes with filtered colimits.
\item
On $E_2$-spectral algebraic spaces, $K^{Sel}$ is an \'etale sheaf.
\end{enumerate}
\end{theorem} 
\begin{remark} 
Part 2 of this result is not the best possible. 
In \Cref{QLSelgeneral} below, we will prove hyperdescent under the hypotheses
of finite Krull dimension and bounded virtual cohomological dimension of the
residue fields. We can also generalize beyond the \'etale topology, see \Cref{fppfKSel}
below. 
\end{remark} 
\begin{proof} 
For 1, since $K$-theory itself commutes with filtered colimits and
$(L_1T)\otimes\mathbb{Q}=T\otimes\mathbb{Q}$ for any spectrum $T$, we reduce to
the claim that $\operatorname{TC}/p$ commutes with filtered colimits for any
prime $p$, which is \cite[Theorem G]{CMM}.

For 2, 
we observe that by \Cref{generaletaledesc}, $L_1 K, L_1 \mathrm{TC}$ are \'etale
sheaves. 
Since $\mathrm{TC}$ is an \'etale hypersheaf by \Cref{TCissheaf}, it follows now
that the pullback $K^{Sel}$ is an \'etale sheaf. 
\end{proof}

\begin{remark}\label{moregeneralkseldesc}
More generally, for any set of primes $\mathcal{P}$, $K^{Sel}_{\mathcal{P}}$ is
an \'etale sheaf. 
This follows because we can check things rationally and mod $p$ (for each $p \in
\mathcal{P}$), and we have just done this above. 
There is also the same result ``with coefficients,'' again with the same proof: if $X$ is an $E_2$-spectral algebraic space and $\mathcal{M}$ is $\perf(X)$-module in $\mathrm{Cat}_{\mathrm{perf}}$, then $K^{Sel}_{\mathcal{P}}(\mathcal{M}\otimes_{\perf(X)}\perf(-))$ is an \'{e}tale sheaf on $X$.
\end{remark}

We can actually strengthen this descent statement slightly when working with  
(discrete) rings or algebraic spaces.

\begin{theorem} 
\label{fppfKSel}
For ordinary qcqs algebraic spaces, $K^{Sel}$ is a sheaf for the fppf topology.
\end{theorem}
\begin{proof} 
We reduce to the affine case. 
We first observe that $K^{Sel}$ is a sheaf for the finite flat topology. 
This assertion follows from the finite flat descent for $L_1 K, L_1
\mathrm{TC}$ proved in \cite[Theorem 5.1]{CMNN}, and 
faithfully flat descent for $\mathrm{TC}$, cf.\ \cite[Sec. 3]{BMS2}. 
Now we use the fact that \'etale descent together with finite flat descent
implies fppf descent, cf.\ \cite[Tag 05WM]{stacks-project}. 
\end{proof} 

\begin{remark} 
Let $R$ be a connective $\mathbb{Z}$-algebra, and consider the $p$-adic
completion ${K^{Sel}}(R)_{\hat{p}}$. 
The work \cite{BCM} shows that the first two ingredients
of ${K^{Sel}}(R)_{\hat{p}}$ depend only on $\pi_0(R)[1/p], \pi_0(\hat{R})[1/p]$ respectively. 
Therefore, the only part of ($p$-complete) Selmer $K$-theory that sees the
higher homotopy groups of $R$ arises from the topological cyclic homology of
$R$. 
\end{remark}

The proof that Selmer $K$-theory is an \'etale {hyper}sheaf (and not only an
\'etale sheaf) will rely on the norm
residue isomorphism theorem, and will be given in the next subsection. 
Here, we identify the \'etale stalks of Selmer $K$-theory, via rigidity
results, and show that they are very close to those of $K$-theory itself.

\begin{theorem}
\label{stalksofKSel}
Let $A$ be a connective $E_1$-ring spectrum such that $\pi_0(A)$ is  a strictly
henselian local (commutative) ring with residue field $k$.  
Then the map
$$K(A)\rightarrow K^{Sel}(A)$$
is an isomorphism on homotopy in degrees $\geq -1$.
Furthermore, for 
a prime $p$, we have:
\begin{enumerate}
\item If $k$ has characteristic $p$, then
$$K^{Sel}(A)_{\widehat{p}}\overset{\sim}{\rightarrow}
\operatorname{TC}(A)_{\widehat{p}} \simeq (K_{\geq 0}(A))_{\hat{p}}.$$
\item If $k$ has characteristic $\neq p$, then
$$K^{Sel}(A)_{\widehat{p}}\overset{\sim}{\rightarrow} (L_1K(A))_{\widehat{p}} \simeq KU_{\widehat{p}}.$$
Here $\simeq$ means a non-canonical equivalence of $E_\infty$-ring spectra, but
$\pi_2 (L_1K(A)_{\widehat{p}})$ is canonically identified with
$\mathbb{Z}_p(1)$, and this precisely pins down the non-canonicity of $\simeq$
since $\operatorname{Aut}(KU_{\widehat{p}})=\mathbb{Z}_p^\times$ given by the
Adams operations \cite{GH04}, which act by scalar multiplication on $\pi_2 KU_{\widehat{p}}=\mathbb{Z}_p$.
\end{enumerate}
\end{theorem}
\begin{proof} 
Assume first $A$ is discrete. 
We first prove the two identifications of $K^{Sel}_{\hat{p}}$. 

If $k$ has characteristic $p$, then by 
\Cref{SelKatp}, we have that 
$K^{Sel}(A)_{\hat{p}} \simeq \TC(A)_{\hat{p}}$, and we can identify this with
the $p$-completion of connective $K$-theory via \cite[Theorem C]{CMM}. 

If $k$ has characteristic $\neq p$, then 
$p$ is invertible in $A$, so we have by \Cref{SelKawayp} that 
$K^{Sel}(A)_{\hat{p}} \simeq L_{K(1)} K(A)$. 
By Gabber-Suslin rigidity 
\cite{Gabber92, Su83}
$K_{\geq 0}(-)_{\widehat{p}}$ is invariant under all ring homomorphisms between strictly
henselian local $\mathbb{Z}[1/p]$-algebras; again this is only true for
connective $K$-theory, but after applying $L_{K(1)}$ we find that this also
holds for $K(-)_{\widehat{p}}$.  Thus to
produce the equivalence $(L_1K(A))_{\widehat{p}} \simeq KU_{\widehat{p}}$ we can
use a zig-zag of maps to connect $A$ to $\mathbb{C}$, where the conclusion
follows from Suslin's comparison with topological $K$-theory
$K(\mathbb{C})/p\overset{\sim}{\rightarrow} ku/p$ \cite{Su84}.  The identification $\pi_2 L_1K(A)_{\widehat{p}}=\mathbb{Z}_p(1)$ is then given in the standard manner using $\mu_{p^\infty}\subset A^\times = K_1(A)$.

Finally, we claim that $K(A) \to K^{Sel}(A)$ is an isomorphism in degrees $\geq
-1$.
We first note that $K\otimes\mathbb{Q}\overset{\sim}{\rightarrow} K^{Sel}\otimes\mathbb{Q}$ in general, since $T\otimes\mathbb{Q}\overset{\sim}{\rightarrow} (L_1T)\otimes\mathbb{Q}$ for any spectrum $T$.  Moreover $K_{-1}(A)=0$ by 
\cite[Theorem 3.7]{Dri06}.  It follows that $\pi_{-1} K^{Sel}(A)$ is torsion.  But the above (mod $p$) descriptions show that $\pi_0 (K^{Sel}(A)/p)$ is generated by the unit, which comes from $\pi_0 K^{Sel}(A)$, so we deduce that $\pi_{-1}K^{Sel}(A)=0$ as well.
Thus it suffices to show that $K_{\geq 0}(A)/p\rightarrow K^{Sel}(A)/p$ is an
isomorphism in degrees $\geq 0$ for all primes $p$, with $K_{\geq 0}$ here
meaning \emph{connective} $K$-theory.  When $k$ has characteristic $p$ this
follows from 1.  When $k$ has characteristic $\neq p$ this follows from Gabber-Suslin rigidity again, which reduces us to the obvious fact $ku/p\overset{\sim}{\rightarrow} (KU/p)_{\geq 0}$.

Now suppose $A$ is a connective $E_1$-ring with $\pi_0(A)$ strictly henselian.
Using the Dundas-Goodwillie-McCarthy theorem \cite{DGM13} twice, we 
have a homotopy pullback square
\[ \xymatrix{
K(A) \ar[d]  \ar[r] &  K^{Sel}(A) \ar[d] \\
K(\pi_0(A)) \ar[r] &  K^{Sel}(\pi_0(A))
}.\]
In addition, if $p$ is invertible in $A$, then the map $K(A)_{\hat{p}} \to
K(\pi_0 A)_{\hat{p}}$ is an equivalence (e.g., via Dundas-Goodwillie-McCarthy or
more simply by a group homology calculation). 
Using these facts, we 
easily reduce to
the case where $A$ is discrete treated above.
\end{proof}

\begin{corollary}\label{KSelHomotopy}
Let $d\in\mathbb{Z}$ and $p^n$ a prime power.  Then on \'etale homotopy group
sheaves with (mod $p^n$) coefficients over an arbitrary qcqs algebraic space $X$ with connective structure sheaf, the natural map $K^{Sel}\rightarrow \operatorname{TC}$ induces a surjection
$$\pi^{et}_d (K^{Sel}/p^n) \rightarrow \pi^{et}_d(\operatorname{TC}/p^n),$$
with kernel $\mathcal{F}_d$ described as follows:
\begin{enumerate}
\item For $d$ odd, $\mathcal{F}_d=0$;
\item For $d=2q$ even, $\mathcal{F}_d = j_!\mathbb{Z}/p^n\mathbb{Z}(q)$ with $j$ the open inclusion of the characteristic $\neq p$ locus of $X$ into $X$.
\end{enumerate}
\end{corollary}

\begin{proof}
Because the map from the units of a ring to $K_1$ of the ring is an isomorphism
on local rings, the $p^n$-torsion in $\pi_1^{et}K$ identifies with
$\mu_{p^n}=\mathbb{Z}/p^n\mathbb{Z}(1)$. Comparing with Proposition
\ref{stalksofKSel}, we deduce that the Bockstein map
$\pi_2^{et}(K/p^n)\rightarrow(\pi_1^{et}K)[p^n]=\mathbb{Z}/p^n\mathbb{Z}(1)$ is
an isomorphism over rings in which $p$ is invertible.  The homotopy group
sheaves of $j^\ast K/p^n$ are even periodic by Proposition \ref{stalksofKSel},
so by multiplicativity we deduce a comparison isomorphism $\mathbb{Z}/p^n\mathbb{Z}(q)\overset{\sim}{\rightarrow} j^\ast \pi_{2q}^{et}(K/p^n)$ for all $q\in\mathbb{Z}$.  This is adjoint to a map $\mathcal{F}_d \rightarrow \pi_{2q}^{et}(K/p^n)$.  We need to see that
$$0\rightarrow \mathcal{F}_d\rightarrow \pi_d^{et}(K^{Sel}/p^n)\rightarrow \pi_d^{et}(\operatorname{TC}/p^n)\rightarrow 0$$
is exact, but it suffices to check this on stalks, where it follows from
\Cref{stalksofKSel}.
\end{proof}

\subsection{Hyperdescent and Lichtenbaum-Quillen for Selmer $K$-theory}

Here   we indicate the Lichtenbaum-Quillen style statements one obtains for 
the map $K \to K^{Sel}$, showing that the map is often an equivalence in high
enough degrees. We begin with the case of fields, when the result is
known.

\begin{theorem} 
\label{QLSelmerfields}
Let $k$ be a field and let $p$ be a prime number. Let $d$ be the virtual
cohomological dimension (mod $p$) of $k$ if $p \neq \mathrm{char}(k)$, and
$1+\log_p [k:k^p]$ if $p = \mathrm{char}(k)$. 
Then: 
\begin{enumerate}
\item  
The map $K(k)_{(p)} \to K^{Sel}(k)_{(p)}$ induces an equivalence on $(d-2)$-connective
covers: more precisely, its homotopy fiber is concentrated in degrees $\leq d-4$. 
\item
The construction $A \mapsto K^{Sel}(A)_{(p)}$ defines a hypercomplete \'etale sheaf
on $\spec(k)_{et}$. 
\end{enumerate}

\end{theorem} 
\begin{proof} 
Suppose first that $p \neq \mathrm{char}(k)$. Then the result is
essentially contained in
\cite{RO05, RO06}.  Indeed, $K^{Sel}(k)/p \simeq L_{K(1)}
K(k)/p$. 
The results of \emph{loc. cit.} show that 
$K(k)/p \to K^{Sel}(k)/p$ is an equivalence in degrees $\geq d-2$. 
Therefore, 
the homotopy fiber $F$ of $K(k) \to K^{Sel}(k)$ 
has mod $p$ homotopy in degrees $\leq d - 3$. 
Since this fiber $F$ is torsion, it follows 
that $F_{(p)}$ has homotopy in degrees $\leq d-4$. 
The hypercompleteness is also proved there, but we verify it below as well.

Suppose $\operatorname{char}(k)=p$, so that in this case $K^{Sel}(k)/p =
\mathrm{TC}(k)/p$
by \Cref{SelKatp}. 
Then the result follows from
\cite{GL00, GH99}. 
In fact, $K(k)/p$ and $\mathrm{TC}(k)/p$ are both $(d-1)$-truncated and the map
$K(k)/p \to \mathrm{TC}(k)/p$ is an isomorphism on top homotopy groups (in
degree $(d-1)$) and an injection in all degrees (in particular, in degree $d-2$), 
so the homotopy fiber $F$ of $K(k) \to K^{Sel}(k)$ has mod $p$ homotopy in
degrees $\leq d-3$, and thus $F_{(p)}$ has homotopy in degrees $\leq d-4$ since
it is torsion. 
Explicitly, the homotopy groups are given by 
$\pi_i(K(k)/p) \simeq \Omega^i_{k, \mathrm{log}}$ and $\pi_i( \mathrm{TC}(k)/p)
\simeq
\Omega^i_{k, \mathrm{log}} \oplus H^1_{et}(\spec(k),
\Omega^{i+1}_{\mathrm{log}})$. 
Thus, if $F$ denotes the fiber of the map $K(k) \to K^{Sel}(k)$ as before, then 
$F/p$ is concentrated in degrees $\leq d-3$, so $F$ is concentrated in degrees
$\leq d-4$ (since $F$ is torsion). 
It also follows that $K^{Sel}/p$ defines a hypercomplete \'etale sheaf since it
is a truncated \'etale sheaf; therefore, $K^{Sel}_{(p)}$ is a hypercomplete \'etale sheaf too. 
\end{proof}

For the convenience of the reader, we briefly include a version of the argument
(slightly reformulated) in the
case $\mathrm{char}(k) \neq p$. 
In this case, the goal is to prove that if $k$ has virtual cohomological
dimension $d$, then the homotopy fiber of $K(k)/p \to L_1 K(k)/p$ belongs to
$\sp_{\leq d-3}$. 

\begin{construction}[Review of Beilinson-Lichtenbaum]
Let $k$ be a field of characteristic $\neq p$. 
The motivic or slice filtration
\cite{FS02, Lev08} of $K$ restricts to a decreasing, multiplicative $\mathbb{Z}_{\geq 0}$-indexed filtration
$$\ldots \rightarrow F^{\geq n+1}(K)\rightarrow F^{\geq n}(K) \rightarrow \ldots \rightarrow F^{\geq 0}(K) = K$$
of Nisnevich sheaves on $\operatorname{Spec}(k)$ with the following properties:
\begin{enumerate}
\item $\varprojlim_n F^{\geq n}(K) = 0$;
\item $F^{=n}(K):=\mathrm{cofib}(F^{\geq n+1}(K)\rightarrow F^{\geq n}(K))$ identifies with $\Sigma^{2n}H\mathbb{Z}(n)$, the Nisnevich sheaf representing motivic cohomology in degree $2n$ and weight $n$.
\end{enumerate}
We will primarily work with mod $p$ coefficients, so that we obtain 
a filtration $\left\{F^{\geq n}(K/p)\right\}_{n \geq 0}$
with associated graded given by $F^{=n}(K/p) = \Sigma^{2n}H \mathbb{F}_p(n)$,
i.e., one obtains motivic cohomology with mod $p$ coefficients. 
By the norm residue isomorphism theorem (cf.\ \cite{HW19} for a textbook
reference), 
we have an equivalence for any \'etale $k$-algebra $A$ and for each $n \geq 0$,
\begin{equation} \label{BLthm} H \mathbb{F}_p(n) (A) \simeq \tau_{\geq -n} R \Gamma_{et}( \spec(A),
\mu_p^{\otimes n}) .\end{equation}
\end{construction}

Next we need an elementary connectivity lemma. 
\begin{lemma} 
\label{sheafconnectivitylemma}
Let $\sF$ be a sheaf of spectra on a site. Let $\mathcal{G}$ denote the presheaf
$\tau_{\geq j} \sF$ and let $\mathcal{G}^{sh}$ denote its sheafification. 
Then the fiber of the map $\mathcal{G} \to \mathcal{G}^{sh}$ is a presheaf of
spectra which is concentrated in
degrees $\leq j-2$. 
\end{lemma} 
\begin{proof} 
We have a sequence of presheaves of spectra $\mathcal{G} \to \mathcal{G}^{sh}
\to \mathcal{F}$. 
Since sheafification is $t$-exact, the map $\mathcal{G}^{sh}\to
\mathcal{F}$ is an isomorphism on homotopy group sheaves in degrees $\geq j$
and an injection in degree $j-1$. Thus, 
$\mathrm{fib}( \mathcal{G}^{sh} \to \mathcal{F})$ is $(j-2)$-truncated as a
sheaf of spectra, and hence as a presheaf of spectra. 
Using the sequence
$\mathrm{fib}(\mathcal{G} \to \mathcal{G}^{sh}) \to \mathrm{fib}(\mathcal{G} \to
\mathcal{F}) \to \mathrm{fib}( \mathcal{G}^{sh} \to \mathcal{F})$, 
we see that the second and third terms are $(j-2)$-truncated, and thus so is the
first. 
\end{proof}

\begin{proof}[Proof of \Cref{QLSelmerfields} in case of finite cohomological dimension] 

Suppose first that $k$ has cohomological dimension $d$ and $p \neq
\mathrm{char}(k)$. 
In this case, it follows from the above filtration 
and \eqref{BLthm}
that for $n \geq d$,
the functor that $F^{=n}(K/p)$ defines  on $\spec(k)_{et}$ is a truncated
\'etale sheaf (hence a hypersheaf). 
Taking the inverse limit up the tower, we conclude that $F^{\geq d}(K/p)$
defines a hypersheaf 
on $\spec(k)_{et}$. 
Since $L_1 (K/p) \simeq L_1 F^{\geq d} (K/p)$ (because $L_1$ annihilates the
thick subcategory of $\mathrm{Sp}$ generated by $ \mathbb{F}_p$-modules) and hypercomplete sheaves are closed
under all colimits in presheaves (\Cref{critforhypcolimits}),
it follows that $L_1 (K/p) = K^{Sel}/p$ is a hypercomplete \'etale sheaf; here
we also use that $L_1$ is smashing, and hence given by tensoring with $L_1
\mathbb{S}$. This
proves 2 of \Cref{QLSelmerfields} for $K^{Sel}/p$; the assertion for
$K^{Sel}_{(p)}$ follows using the arithmetic fracture square and using that
rational $K$-theory is a hypercomplete \'etale sheaf in this case. 

It remains to prove 1. 
Note that the map $F^{=i}(K/p) \to F^{=i} (K/p)^{et}$ (where the target denotes
the sheafification of the source) has homotopy fiber in $\sp_{\leq i-2}$, 
again by \eqref{BLthm} and \Cref{sheafconnectivitylemma}. 
We conclude that if $(K/p)^{et}$ denotes the \'etale sheafification of $(K/p)$
on $\spec(k)_{et}$, then 
$(K/p)^{et}$ fits into a fiber sequence
\[ F^{\geq d}(K/p) \to  (K/p)^{et} \to F_{<d} (K/p)^{et},  \]
where the first term is already a hypercomplete \'etale sheaf. Since the last term
is truncated, we conclude that $(K/p)^{et}$ is hypercomplete
and that the homotopy fiber of $(K/p) \to (K/p)^{et}$ is a 
$(d-3)$-truncated presheaf on $\spec(k)_{\et}$. 
Finally, $(K/p)^{et} \to L_1 (K/p)^{et}$ is a map of 
hypercomplete sheaves of spectra, since hypercompletion and $L_1$-localization
are smashing. 
For a separably closed field $\ell$, we have that $(K/p)(\ell) \to  L_1
(K/p)(\ell)$ has homotopy fiber in degrees $\leq -3$, 
and therefore $(K^{et}/p)(k) \to L_1 (K^{et}/p)(k)$ has homotopy fiber in
degrees $\leq -3$.
Combining these observations, we conclude that 
$(K/p)(k) \to (K^{Sel}/p)(k)$ has homotopy fiber in degrees $\leq d-3$ as
desired. 
\end{proof}

\begin{proof}[Proof of \Cref{QLSelmerfields} at 2, cf.\ \cite{RO05}] 
Let $k$ be a field of characteristic $\neq 2$ and of (mod $2$) virtual cohomological
dimension $d$. 
Note that this implies that $k( \sqrt{-1})$ has cohomological dimension $d$,
cf.\ 
\cite[Sec. 4]{SerreCG}. 
Here we indicate an argument for 
\Cref{QLSelmerfields} which should be equivalent to that of \cite{RO05}, but
which does not involve any explicit spectral sequence calculations. 

With mod $2$ coefficients, the norm residue isomorphism theorem \eqref{BLthm}
simply becomes
\[ H \mathbb{F}_2(n)(A) \simeq \tau_{\geq -n} R \Gamma_{et}( \spec(A),
\mathbb{F}_2),  \]
as $A$ ranges over \'etale $k$-algebras. 
Since $k$ is only assumed of \emph{virtual} cohomological dimension $d$,
$H\mathbb{F}_2(n)$ need not be an \'etale sheaf for any $n$, so the
previous argument does not directly apply. We need to use additionally the
nilpotence of the Hopf map. In other words, unlike previously, the motivic
spectral sequence does not degenerate. 

Given an \'etale $k$-algebra $A$, the class $-1 \in A^{\times}$ defines an
element $u \in K_1(A)$, which lives in $F^{\geq 1} K$. 
With mod $2$ coefficients, it defines a class in $F^{\geq 1} (K/2)(A)$. 
In associated gradeds, 
this gives the class $t$ in $\pi_1 
(F^{=1} (K/2)(A)) =
H^1( \spec(A)_{et}, \mathbb{F}_2)$ arising from $-1 \in A^{\times}$ via the
Kummer sequence. 
We can also identify this class as follows. 
For any $A$, we have that 
\[ R \Gamma( \spec(A)_{et}, \mathbb{F}_2) 
\simeq 
R \Gamma( \spec(A[i])_{et}, \mathbb{F}_2)^{hC_2}
\]
where $A[i] = A \otimes_{\mathbb{Z}[1/2]} \mathbb{Z}[1/2, i]$ is a $C_2$-Galois
extension of $A$. Note in addition that $A[i]$ has \'etale cohomological
dimension $\leq d$. 
Unwinding the definitions, we conclude that 
the class $u$ arises from the 
map 
\[ \mathbb{F}_2^{hC_2} \to  
R \Gamma( \spec(A[i])_{et}, \mathbb{F}_2)^{hC_2}
\]
as the image of the generator $t \in H^1( C_2; \mathbb{F}_2)$. 

Now consider the filtered spectrum $\left\{F^{\geq i} (K/2)(A)\right\}$, for $A
$ an \'etale $k$-algebra. 
This is a module over the filtered spectrum $\left\{F^{\geq i} K(A)\right\}$. 
Since $u$ defines a class in filtration $1$, it follows that we get a complete
exhaustive filtration 
$\{F^{\geq i} (K/(2, u))\}$
on $(K/2)(A)/(u)$ whose associated graded terms are given by 
$F^{=i}( K/(2, u))
= \mathrm{cofib}( \Sigma F^{= i-1} (K/2)  \to
F^{=i}(K/2))$ (i.e., we take the cofiber of multiplication by $u$, but
in filtered spectra, recording that it raises filtration by $1$). 
Unwinding the definition, we have
that the associated graded terms are 
\[ F^{=i}( K/(2, u))(A)  \simeq
\Sigma^{2i}\mathrm{cofib}( 
\Sigma^{-1} \tau_{\geq -(i-1)} 
R \Gamma( \spec(A[i])_{et}, \mathbb{F}_2)^{hC_2} \to \tau_{\geq -i}
R \Gamma( \spec(A[i])_{et}, \mathbb{F}_2)^{hC_2})
\]
where the map is multiplication by $t$. 
By \Cref{htpygplemma}, for $i \geq d+1$, we can remove the truncations and therefore obtain an
object which is actually an \'etale hypersheaf on $\spec(k)_{et}$. 
Passing up the limit, it follows that $F^{\geq d+1} (K/ (2, u))$ is an \'etale
hypersheaf.

 Note that the map $F^{=i}(K/(2, u)) \to (F^{=i}(K/(2, u)))^{et}$
 has homotopy fiber which is $(i-1)$-truncated, by checking on each of the terms
 in the cofiber. 
It thus follows that the map 
$$ K/(2, u)(A) \to (K/(2, u)^{et}(A)$$
has homotopy fiber in degrees $\leq d-1$, and that the target is hypercomplete. 
Since, however, $u$ is nilpotent --- it arises from the Hopf map $\eta$\footnote{Two proofs: one, according to the Barratt-Priddy-Quillen theorem, the Hopf map is detected in homology by the sign of a permutation, so it suffices to note that a permutation matrix has determinant given by the sign of the permutation; two, it suffices to verify the claim in $K_1(\mathbb{Z})\overset{\sim}{\rightarrow} \pi_1ko$ where it is classical.} --- we
can conclude the result for $K/2$ itself, i.e., that 
$K/2 (A) \to (K/2)^{et}(A)$ has homotopy fiber which is $\leq d-3$-truncated,
via \Cref{etatruncated} below. 
This shows that $K/2(k) \to (K/2)^{et}(k)$ has homotopy fiber in degrees $\leq
d-3$ as desired. As in the previous proof, we can make the same conclusion with
$(K/2)^{et}$ replaced by $K^{Sel}/2$. 
\end{proof} 

\begin{lemma} 
\label{etatruncated}
Let $X$ be a spectrum. Suppose that the cofiber $C$ of the Hopf map $\eta: \Sigma X
\to X$ belongs to $\sp_{\leq n}$. Then $X \in \sp_{\leq n-2}$. 
\end{lemma} 
\begin{proof} 
Suppose that there exists $x \in \pi_i(X)$ for $i \geq n-1$. 
Recall that $\eta$ is nilpotent in the stable stems; therefore, 
up to replacing $x$ by an $\eta$-multiple (and thus raising $i$), we may assume
$\eta x = 0$. 
The cofiber sequence
 $\Sigma X \to X \to C$ shows that $x$, considered as an
 element $\pi_{i+1}( \Sigma X)$
must be the image of a class from $\pi_{i+2}(C)$. This contradicts the
assumption that $C \in \sp_{\leq n}$. 
\end{proof} 

\begin{lemma} 
\label{htpygplemma}
Let $M$ be an $\mathbb{F}_2$-module spectrum with a $C_2$-action.
Suppose that the homotopy groups of $M$ are concentrated in degrees $[-d, 0]$. 
Then for each $i \geq d+1$, the map 
\[ \Sigma^{-1}\tau_{\geq -(i-1)}(M^{hC_2} ) \to \tau_{\geq -i}( M^{hC_2}) \]
given by multiplication by $t \in H^1(C_2; \mathbb{F}_2)$
is identified with the cofiber of the map $t: \Sigma^{-1} M^{hC_2} \to
M^{hC_2}$, which is $M$. 
\end{lemma} 
\begin{proof} 
This follows from the fact that multiplication by $t$ induces an isomorphism 
$\pi_{-j}(M^{hC_2}) \to \pi_{-j-1}(M^{hC_2})$ for $j> d$.
In fact, we have a cofiber sequence
\[ \Sigma^{-1}M^{hC_2} \to M^{hC_2} \to M  \]
where the first map is multiplication by $t$, using the 
fiber sequence $\mathbb{F}_2 \to \mathbb{F}_2[C_2]\to \mathbb{F}_2$ in
$\fun(BC_2, \md_{\mathbb{F}_2})$, tensoring with $M$, and taking $C_2$-homotopy
fixed points. 
\end{proof}

Now, finally, we can state the main results about Selmer $K$-theory for 
connective
$E_2$-spectral algebraic spaces. 
Note that when $p$ is invertible, the result 
appears in \cite{RO06}. 
\begin{theorem} 
\label{QLSelgeneral}
Let $X$ be a connective $E_2$-spectral algebraic space and fix a prime number
$p$. 
Suppose $X$ has finite Krull dimension.
For a point $x \in X$, let 
$d_x = \mathrm{vcd}_p(x)$ if $\mathrm{char}(k(x))$ is prime to $p$, and 
$d_x = \log_p [k(x):k(x)^p] + 1$ if 
$\mathrm{char}(k(x)) = p$. 
Let $d = \sup_{x \in X} d_x$, and suppose $d < \infty$. 
Then: 
\begin{enumerate}
\item  
The map  $K(X)_{(p)} \to K^{Sel}(X)_{(p)}$ 
has homotopy fiber 
which is concentrated in degrees $\leq \max(d-4, -2)$. 
In particular, it is an isomorphism in degrees $\geq \max(d-2, 0)$. 
\item
$K^{Sel}_{(p)}$ defines a hypersheaf on $X_{et}$. 
\end{enumerate}
\end{theorem} 
\begin{proof} 
The strategy is to 
use rigidity to reduce both 1 and 2 to the case of fields. 
By working Nisnevich locally, and using the fact that Nisnevich sheaves are
Postnikov complete (\Cref{Nishypercomplete}), we find that 
for 1, it suffices to treat the case where $X $ is the spectrum of a 
connective $E_2$-ring $R$ such that $\pi_0(R)$ is henselian with residue field $k$. 
We can make the same reduction for 2 thanks to \Cref{main:hypcriterion} and the
fact that Selmer $K$-theory commutes with filtered colimits and is
a Nisnevich (even \'etale) sheaf (\Cref{selmer:basicproperties}): it
suffices to prove hypercompleteness when $X$ is replaced by one of its
henselizations.

Thus, let $R$ be a henselian local ring with residue field $k$. 
We need to show that 
the homotopy fiber $F$ of $K(R)_{(p)} \to K^{Sel}(R)_{(p)}$ lives in degrees
$\leq \max(d - 4, -2)$. 
Consider the composite maps
$K_{\geq -1}(R)_{(p)} \to K(R)_{(p)} \to K^{Sel}(R)_{(p)}$. It suffices to show
that the fiber $\widetilde{F}$ of the composite lives in degrees 
$\max(d - 4, -2)$ because the fiber of the first map lives in degrees $\leq -3$.

Now with mod $p$ coefficients, we have a homotopy pullback square
\[  \xymatrix{
K_{\geq -1}(R)/p \ar[d]   \ar[r] &   K^{Sel}(R)/p \ar[d] \\ 
K_{\geq -1}(k)/p \ar[r] &  K^{Sel}(k)/p
}
\]
by \cite{DGM13} to reduce to the discrete case and then the main result of
\cite{CMM} as well as 
\cite[Theorem 3.7]{Dri06}, 
which states that $K_{-1}(R) = 0$. 
By \Cref{QLSelmerfields}, we conclude that $\widetilde{F}/p$ is concentrated in degrees $\leq (d-3)$. But
$\widetilde{F}_{\mathbb{Q}}$ is concentrated in degrees $\leq -2$. 
It follows that $F_{(p)}$ is concentrated in degrees 
$\leq \max( d-4, -2)$ by \Cref{lem:concentratedmodp} below. 
This completes the proof of 1. 

For 2, we observe that 
if $R$ is a henselian local ring with residue field $k$, then the above shows
that 
$K^{Sel}(R)/p \to K^{Sel}(k)/p$
has homotopy fiber concentrated in degrees $\leq 0$. 
More generally, 
if $R'$
is a finite \'etale $R$-algebra, then $R$ is a finite product of henselian local
rings, so similarly the fiber of 
\begin{equation}
K^{Sel}(R')/p  \to K^{Sel}(R' \otimes_R k)/p 
\label{selrigidmap}
\end{equation}
lives in degrees $\leq 0$. 
Now if $R$ is the henselization of $X$ at a point, then the residue field $k$
has finite mod $p$ virtual cohomological dimension. 
Thus, the right-hand-side of \eqref{selrigidmap} defines a hypersheaf on the
finite \'etale site of $R$ (or the \'etale site of $k$) by
\Cref{QLSelmerfields}. 
Since the left-hand-side is an \'etale sheaf, it follows that it must be a
hypersheaf too, since truncated sheaves are always hypercomplete. Now using the
main hypercompleteness criterion (\Cref{main:hypcriterion}), and since
$K^{Sel}_{\mathbb{Q}}$ is a hypersheaf, we conclude that
$K^{Sel}_{(p)}$ is a hypersheaf on $X_{et}$. 
\end{proof}

\begin{lemma} 
\label{lem:concentratedmodp}
Let $Y$ be a $p$-local spectrum. 
Suppose $Y_{\mathbb{Q}}$ is concentrated in degrees $\leq d_1$ and $Y/p$ is
concentrated in degrees $\leq d_2$. Then $Y$ is concentrated in degrees $\leq
\max(d_1, d_2 - 1)$. 
\end{lemma} 
\begin{proof} 
Suppose that there exists a nonzero $x \in \pi_i(Y)$ for $i > \max(d_1, d_2 - 1)$. 
Then by assumption $x$ is $p$-power torsion. Multiplying $x$ by a power of $p$,
we can assume that $px = 0$. Then there is a  nonzero element in
$\pi_{i+1}(Y/p)$ which Bocksteins to $x$, contradicting the assumptions. 
\end{proof}

\section{\'Etale $K$-theory}

Here we formally define \'etale $K$-theory as a functor, and prove its basic
properties. The basic ingredient is the properties of $K^{Sel}$ proved in the
previous section. 

\subsection{Big versus small topoi}

We define \'etale $K$-theory formally as a functor on $E_2$-spectral
algebraic spaces, via the big site. An observation is that the
sheafification process can be done using either the small or the big site; by
contrast if one defines \'etale $K$-theory directly using the small site one
has to prove functoriality, although the small site is much more convenient
given our previous discussion. 
Note that this is very classical at least when one works with sheaf cohomology,
cf.\ \cite[Exp. VII, Sec. 4]{SGA4} for closely related results for the \'etale
topoi, and cf.\
\cite[Exp. III]{SGA4} for the general results in the case of sheaves of sets.

Let $\mathcal{T}$ be a site, so that for each $t \in \mathcal{T}$, we are given
some family of subobjects of the Yoneda functor $h_t \in \psh(\mathcal{T})$. 
Let $u: \mathcal{T}'  \to  \mathcal{T}$ 
be a morphism of sites, i.e., precomposition with $u$ induces a functor
$u_*: \psh(\mathcal{T}) \to \psh(\mathcal{T}')$ which 
carries sheaves on
$\mathcal{T}$ to sheaves on $\mathcal{T}'$. 
\begin{proposition} 
\label{sheafificationintwocases}
Suppose for each $t \in \mathcal{T}$ and each covering sieve $\widetilde{h_t}\to h_t$,
the induced map of functors
$\pi_0(u_*\widetilde{h_t}) \to \pi_0(u_*h_t) \in \psh(\mathcal{T}') $ induces a surjection
after sheafification.
 Then: 
\begin{enumerate}
\item  
The functor of precomposition with $u$,
$u_*: \psh(\mathcal{T}) \to \psh(\mathcal{T}')$ commutes with sheafification. 
\item
The functor $u_*: \sh(\mathcal{T}) \to \sh(\mathcal{T}')$ commutes with
colimits and with taking
$n$-truncations, for each $n$ (and thus with Postnikov completion). 
\item The functor $u_*: \sh(\mathcal{T}) \to \sh(\mathcal{T}')$ commutes with
taking hypercompletion. 
\end{enumerate}
\end{proposition} 
\begin{proof} 
First we verify 1. 
Let $\mathcal{F} \in \psh(\mathcal{T})$ and let $\mathcal{F}'$ be its
sheafification. Then $f: \mathcal{F} \to \mathcal{F}'$ is uniquely characterized by
the fact that $f$ induces an equivalence after sheafification and that
$\mathcal{F}'$ is a sheaf: that is, sheafification is a Bousfield localization. 
We know now that $u_* \mathcal{F}'$ is a sheaf. To see that $u_* \mathcal{F}'$
is the sheafification of $u_* \mathcal{F}$, 
it suffices to show that $u_*$ 
carries morphisms 
which become equivalences upon sheafification 
to morphisms which become equivalences upon sheafification. 

To this end, consider first the case of maps of the form $\widetilde{h_t} \to
h_t$ associated to a covering sieve of an object $t \in \mathcal{T}$; by
construction these induce equivalences upon sheafifying in $\mathcal{T}$. Now we have a homotopy cartesian diagram
in $\psh(\mathcal{T})$,
\[ \xymatrix{
\widetilde{h_t} \ar[d] \ar[r] &  h_t \ar[d]  \\
\pi_0( \widetilde{h_t}) \ar[r] &  \pi_0(h_t)
}.\]
Here the bottom arrow (which is a monomorphism of presheaves of sets) induces an
epimorphism, hence an equivalence after applying $u_*$
and sheafifying by assumption. Since sheafification and $u_*$
are left exact, it follows that 
$u_*\widetilde{h_t}\to u_*h_t$ induces an equivalence after sheafification. 
Now we appeal to the theory of strongly saturated classes and Bousfield
localizations, cf.\ \cite[Sec. 5.5.4]{HTT}. 
Recall that $\sh(\mathcal{T})$ is the Bousfield localization of
$\psh(\mathcal{T})$ at the class 
of arrows $\widetilde{h_t} \to h_t$, over all covering sieves. 
Now the class of morphisms in
$\psh(\mathcal{T})$ which induce an equivalence after sheafification is the
strongly saturated class generated by the arrows $\widetilde{h_t} \to h_t$, cf.\
\cite[Prop. 5.5.4.15]{HTT}. Now $u_*$ preserves colimits and the class of morphisms
in $\psh(\mathcal{T}')$ which induce an equivalence after sheafification is
again strongly saturated. It thus follows from the above that $u_*$ carries
the strongly saturated class of morphisms in $\psh(\mathcal{T})$ which induce
equivalences upon sheafification into the 
strongly saturated class of morphisms in $\psh(\mathcal{T}')$ which induce
equivalences upon sheafification, proving 1. 

Now 2 follows from 1 because for any sheaf $\sF \in \sh(\mathcal{T})$, the $n$-truncation
$\tau_{\leq n}\mathcal{F}$ in $\sh(\mathcal{T})$ is the sheafification of the
presheaf $n$-truncation. Moreover, $u_*$ preserves all limits. 

Finally, 3 follows from 2 because if $\sF \to \sF'$ induces an equivalence on
homotopy groups (or on $n$-truncations for each $n$), then $u_*( \sF) \to u_*(\sF')$ does as well by 2. 
\end{proof}

\begin{corollary} 
Let $\mathcal{D}$ be a presentable $\infty$-category. 
Under the hypotheses of 
\Cref{sheafificationintwocases}, 
the functor $u_*: \psh(\mathcal{T}, \mathcal{D}) \to \psh(\mathcal{T}',
\mathcal{D})$ commutes with sheafification. 
\end{corollary} 
\begin{proof} 
Keep the notation of 
\Cref{sheafificationintwocases}. 
Then the result follows from the above by taking the tensor product in
presentable $\infty$-categories with $\mathcal{D}$. 
Alternatively, one can argue directly: here $\sh(\mathcal{T}, \mathcal{D}) \subset \psh(\mathcal{T}, \mathcal{D})$ is a
Bousfield localization at the class of maps $d \otimes \widetilde{h_t} \to d
\otimes h_t$, as $\widetilde{h_t} \to h_t$ ranges over covering sieves and $d$
ranges over the objects in $\mathcal{D}$. As in  
\Cref{sheafificationintwocases}, it suffices to show that each of these maps 
is carried to a map in 
$\psh(\mathcal{T}, \mathcal{D})$ which becomes an equivalence after
sheafification. However, this follows from the assumptions. 
\end{proof}

\begin{example} 
\label{resGcommutesheaf}
Suppose the functor $u: \mathcal{T}' \to \mathcal{T}$ admits a right adjoint $g:
\mathcal{T} \to \mathcal{T}'$. 
Unwinding the definitions (or using the Hochschild-Serre spectral sequence), it follows that if $t \in \mathcal{T}$, then $u_*$
carries the representable presheaf $h_t \in \psh(\mathcal{T})$ to the
representable presheaf $h_{g(t)} \in \psh(\mathcal{T}')$. 
If $g$ carries covering families to covering families, then it follows that the
hypotheses of 
\Cref{sheafificationintwocases} apply. As an instance of this, let $H \leq G$ be
an open subgroup of a profinite group $G$, 
then the functor $\mathrm{Ind}_H^G: \mathcal{T}_H \to \mathcal{T}_G$ admits a
right adjoint, given by the forgetful functor. Since this clearly preserves
covering families, it follows that the induced functor on presheaves of spaces
commutes with sheafification.
\end{example}

\begin{definition}[The big \'etale site] 
The \emph{big \'etale site} of an $E_2$-spectral algebraic space $X$ consists of all 
$E_2$-spectral algebraic spaces $Y$ over $X$ and maps between them. The topology is generated
by finite families $\left\{Y_i \to Y\right\}_{i \in I}$ such that each $Y_i \to
Y$ is \'etale and they generate a covering in the \'etale site $Y_{et}$. 
\end{definition}

\begin{remark} 
Here we should restrict to objects of some bounded cardinality $\kappa$ for
set-theoretic reasons, but the
cardinal number $\kappa$ does not affect any of the statements (in view of
\Cref{sheafinsmallorbig} below), so we omit it. 
\end{remark}

\begin{proposition} 
\label{sheafinsmallorbig}
Fix an $E_2$-spectral algebraic space $X$. 
Let $\sF: \algspc_X^{op} \to \sp$ be a functor. 
Let $\widetilde{\sF}$ be the sheafification of $\sF$ (in the big \'etale site). 
Then for any $E_2$-spectral algebraic space 
$Y$ over $X$, the map 
$\sF|_{Y_{et}} \to \widetilde{\sF}_{Y_{et}}$ exhibits the target as the
sheafification of the source (on the site $Y_{et}$).
\end{proposition} 
\begin{proof} 
This follows from \Cref{sheafificationintwocases}. 	
One has to check that 
given an \'etale surjective map $Z' \to Z$ in $\algspc_X^{op}$, the 
map $h_{Z'} \to h_Z$ induces a surjection on $\pi_0$ after sheafifying in $Y_{et}$. 
However, given any $Y' \in Y_{et}$, we observe that any map $Y \to Z$ \'etale
locally lifts to $Z'$, as desired. Note that this is a purely algebraic fact (on
$\pi_0$), thanks to  \eqref{E2etalespectral}. 
\end{proof}

\begin{proposition} 
Let $\sF: \algspc_X^{op} \to \sp$ be a functor. 
Then the following are equivalent: 
\begin{enumerate}
\item $\sF$ is an \'etale sheaf (resp. \'etale hypersheaf, resp. \'etale
Postnikov sheaf). 
\item For each $Y \in \algspc_X$, $\sF|_{Y_{et}}$ is an \'etale sheaf (resp.
\'etale hypersheaf, resp. \'etale Postnikov sheaf). 
\end{enumerate}
\end{proposition}

\subsection{Properties of \'etale $K$-theory}
Here we give the main applications to \'etale $K$-theory. 
We begin with the definition. 
\newcommand{\algetwo}{\mathrm{Alg}_{E_2}(\mathrm{Sp})}

\begin{definition}[\'Etale $K$-theory] 
We consider the functor of algebraic $K$-theory (\Cref{def:Ktheory}) on $\algspc$, and denote its
\'etale 
sheafification by $K^{et}$. 
We call this functor \emph{\'etale $K$-theory.}
\end{definition}

\begin{remark} 
By \Cref{sheafinsmallorbig}, the \'etale sheafification defining $K^{et}$ can be
carried out either on the big or small \'etale sites, and the results are equivalent.
\end{remark} 

\begin{construction}[The trace from $K^{et}$]
\label{etaletrace}
We have a natural trace map $K \to K^{Sel}$. Since $K^{Sel}$ is an \'etale sheaf
(\Cref{selmer:basicproperties}), it follows that we obtain a factorization
$K^{et}\to K^{Sel}$ of sheaves of spectra. 
\end{construction}

Our main goal is to control \'etale $K$-theory $K^{et}$. In doing so, we use the
map $K^{et} \to K^{Sel}$, and the fact that we have a good handle on $K^{Sel}$:
it is an \'etale sheaf, and commutes with filtered colimits. 
The following argument is very general; it would work with the category of
commutative rings, or connective $E_\infty$-ring spectra. 
Let $\algetwo_{\geq 0}$ denote the $\infty$-category of connective $E_2$-ring
spectra. 
This is a compactly generated $\infty$-category and for a compact object $R \in
\algetwo_{\geq 0}$, $\pi_0(R)$ is a finitely generated commutative ring. 
Compare \cite[Sec.~7.2.4]{HA}. 

\begin{proposition} 
\label{generalaxiomaticarg}
Let $\sF: 
\algetwo_{\geq 0} \to \sp
$  be a functor with the following properties: 
\begin{enumerate}
\item $\sF$ commutes with filtered colimits.  
\item 
When $R \in \algetwo_{\geq 0}$ is such that $\pi_0(R)$ is a strictly henselian ring, $\sF(R) \in \sp_{\leq 0}$. 

\item When $R \in \algetwo_{\geq 0}$ is compact, then $\sF$
restricted to $\spec(R)_{et}$ is $d$-truncated (as a presheaf) for some $d$
(possibly depending on $R$).

\end{enumerate}

Then: 
\begin{enumerate}
\item  
The \'etale sheafification $\sF^{et}: \algetwo_{\geq 0} \to
\sp$ commutes with filtered colimits.
\item
$\sF^{et}$
takes
values in $\sp_{\leq 0}$. 
\item 
$\sF^{et}$ 
is hypercomplete as an \'etale sheaf. 
\end{enumerate}

\end{proposition} 
\begin{proof} 
Let $\widetilde{\sF^{et}}$ be obtained by restricting the \'etale sheafification
$\sF^{et}$ to compact objects of $\algetwo_{\geq 0}$ and then Kan extending to
all of $\algetwo_{\geq 0}$, so $\widetilde{\sF^{et}}$ commutes with filtered
colimits. 
We will verify 1, 2, and 3 for $\widetilde{\sF^{et}}$ first. 
Of course, 1 is automatic by construction.

By construction and \Cref{sheafinsmallorbig}, $\widetilde{\mathcal{F}^{et}}$ when restricted to the \'etale site
of a compact object $R \in \algetwo_{\geq 0}$ is the sheafification of
$\mathcal{F}|_{\spec(R)_{et}}$; this by assumption is truncated, so
$\widetilde{\mathcal{F}^{et}}$ is Postnikov complete. 
Since the stalks of $\sF$ belong to $\sp_{\leq 0}$, 
we conclude that 
$\widetilde{\mathcal{F}^{et}}(R) \in \sp_{\leq 0}$. Kan extending, we conclude
that $\widetilde{\sF^{et}}$ takes values in $\sp_{\leq 0}$, verifying 2.

Given a faithfully flat \'etale
map $f: R \to R'$ in $\algetwo_{\geq 0}$, we can write $f$ as a filtered colimit
of faithfully flat \'etale maps $f_\alpha:  R_\alpha \to R'_\alpha$ between
compact objects in $\algetwo_{\geq 0}$. 
By assumption,
$\widetilde{\sF^{et}}$ satisfies the sheaf condition for each map $f_\alpha$, 
i.e., $\widetilde{\sF^{et}}(R_\alpha)$ is the totalization of
$\widetilde{\sF^{et}}$ applied to the \v{C}ech nerve of 
$f_\alpha$. 
Now filtered colimits commute with totalizations in $\sp_{\leq 0}$, so 
it follows that $\widetilde{\sF^{et}} $ satisfies the sheaf condition
for $R \to R'$, verifying that $\widetilde{\sF^{et}}$ is a sheaf; it is
automatically hypercomplete since it is truncated. 

Finally, it remains to show that 
$\widetilde{\sF^{et}}$ actually is the sheafification of $\sF$. 
To this end, we have a map $\widetilde{\sF^{et}} \to \sF^{et}$ by left Kan
extension. Since $\sF^{et}$ is the sheafification of $\sF$ and
$\widetilde{\sF^{et}}$ is an \'etale sheaf receiving a map from $\sF$, we obtain a map
$\sF^{et} \to \widetilde{\sF^{et}}$. A diagram 
chase shows that both maps are inverses to each other, as desired. 
\end{proof} 

\begin{theorem}
\label{mainetaleKdesc}
Let $X$ be an $E_2$-spectral algebraic space with connective structure sheaf.  Then: 
\begin{enumerate}
\item  
The map
$$K^{et}(X)\rightarrow K^{Sel}(X)$$
induced by \Cref{etaletrace} is an isomorphism on homotopy in degrees $\geq -1$.
\item If $X$ has finite Krull dimension and admits a global bound on the virtual
$\mathcal{P}$-local cohomological dimensions of the residue fields, then
$K^{et}_{\mathcal{P}}$ is a hypercomplete
sheaf on $X_{et}$. 
\item
The construction $K^{et}(\cdot)$, on $\algetwo_{\geq 0}$, commutes with filtered
colimits. 
\end{enumerate}
\label{KetvsKSelcomp}
\end{theorem}
\begin{proof}
Consider the fiber $\sF$ of the map $K \to K^{Sel}$. 
Then $\sF$ satisfies the conditions of \Cref{generalaxiomaticarg}, 
in view of the results \Cref{selmer:basicproperties}, \Cref{QLSelgeneral}, and
\Cref{stalksofKSel}. 
Recall that the residue fields of a scheme (or algebraic space) of finite type
over $\mathbb{Z}$ have 
virtual cohomological dimension at most $d + 1$, for $d$ the Krull dimension (which is also finite)
\cite[Theorem 6.2, Exp. X]{SGA4}, so these results apply. 
It follows that $\sF^{et} = \mathrm{fib}(K^{et} \to K^{Sel})$  commutes with filtered colimits of connective
$E_2$-rings, takes values in $\sp_{\leq 0} $, and is a hypercomplete \'etale
sheaf. 

Using the above results that we have already proved for $K^{Sel}$ again, 
it follows that $K^{et}$ has all of the above desired properties too, from the
fiber sequence $\sF^{et} \to K^{et} \to K^{Sel}$. 
The fact that $K^{et}(X) \to K^{Sel}(X)$ is an isomorphism in degrees $\geq -1$
can now be tested for $X$ finite type over $\mathbb{Z}$, then for $X$ strictly
henselian local by hypercompleteness, and then follows from \Cref{stalksofKSel}. 
\end{proof}

Finally, we can state and prove the relevant Lichtenbaum-Quillen statement for
$K \to  K^{et}$. 

\begin{theorem}
\label{LQstatementmain}
Let $X$ be an $E_2$-spectral  algebraic space of finite Krull dimension, and $p$ a
prime.  
For a point $x \in X$, let 
$d_x = \mathrm{vcd}_p(x)$ if $\mathrm{char}(k(x))$ is prime to $p$, and 
$d_x = \log_p [k(x):k(x)^p] + 1$ if 
$\mathrm{char}(k(x)) = p$. 
Let $d = \sup_{x \in X} d_x$, and suppose $d < \infty$. 

Then the map
$$K(X)\rightarrow K^{et}(X)$$
is an isomorphism on $p$-local homotopy groups in degrees $\geq
\operatorname{max}(\operatorname{sup}_{x\in X} d_{k(x)} - 2, 0)$.
\end{theorem}
\begin{proof} 
Combine the above comparison of \'etale and Selmer $K$-theory
(\Cref{KetvsKSelcomp})
as well as \Cref{QLSelgeneral}. 
\end{proof}

\subsection{Hyperdescent for telescopically localized  invariants}

Putting everything together, we can prove our most general \'etale hyperdescent results.

\begin{theorem}[\'Etale hyperdescent]
\label{mainhypLn}
Let $X$ be an $E_2$-spectral algebraic space. 
Suppose that $X$ has finite Krull dimension and that there is a 
global bound on the mod $p$ virtual cohomological dimensions of the residue
fields of $X$. Let $\mathcal{A}$ be a weakly localizing invariant for $\perf(X)$-linear
$\infty$-categories which takes values in $L_n^f$-local spectra. 
Then $Y \mapsto \mathcal{A}(\perf(Y))$ defines an \'etale hypersheaf on $X$. 
\end{theorem} 
\begin{proof} 
As any additive invariant factors through non-commutative motives,
$\mathcal{A}(\perf(X)) = \mathcal{A}([\perf(X)])$ carries an action of
$\operatorname{End}_{\mot(X)}([\perf(X)],[\perf(X)])=K_{\geq 0}(X)$, see
\Cref{additivekviamotives}.  This is functorial in the $E_2$-algebraic space $X$
via pullbacks, so in particular we get that the presheaf $Y \mapsto
\mathcal{A}(\perf(Y))$ carries an action of the presheaf $K_{\geq 0}$.  However,
by \Cref{generaletaledesc} $\mathcal{A}$ is an \'etale sheaf; hence this action
extends to an action of $(K_{\geq 0})^{et}$.  The latter is a hypersheaf by
\Cref{mainetaleKdesc} (note that the discrepancy between connective and
non-connective $K$-theory is irrelevant because truncated sheaves are automatically hypercomplete), hence so is $\mathcal{A}$ by \Cref{hypsmashingfinitecd}.
\end{proof}

Finally, we can extend the hyperdescent properties of Selmer
$K$-theory (cf.\ \Cref{selmer:basicproperties} and \Cref{QLSelmerfields}) to the nonconnective case. 

\begin{corollary} 
Let $X$ be an $E_2$-spectral algebraic space. Then 
 if $X$ has finite Krull dimension and the $\mathcal{P}$-local 
cohomological dimension of the residue fields of $X$ are bounded, then 
$K^{Sel}(-)_{\mathcal{P}}$ is an \'etale hypersheaf over $X$.

Moreover, this also holds ``with coefficients" $\mathcal{M}$ as in  \Cref{moregeneralkseldesc}.
\label{KSel:innonconndesc}
\end{corollary} 
\begin{proof} 
In fact, \Cref{mainhypLn} and \Cref{TCissheaf} implies that 
$K^{Sel}/p$ is a hypersheaf for each $p \in \mathcal{P}$. 
It follows that the pro-$\mathcal{P}$-completion $K^{Sel}_{\hat{\mathcal{P}}}$ of $K^{Sel}$ is a hypersheaf,
and so are the rationalizations $K^{Sel}_{\mathbb{Q}}$ and 
$(K^{Sel}_{\hat{\mathcal{P}}})_{\mathbb{Q}}$ by the $n=0$ case of \Cref{mainhypLn}. 
Using the arithmetic square, it follows that $K^{Sel}(-)_{\mathcal{P}}$ is an
\'etale hypersheaf as desired. 
\end{proof}

\bibliography{hyper}

\end{document}